\documentclass[11pt,oneside]{amsart}

\advance \topmargin by -\headheight
\advance \topmargin by -\headsep
\evensidemargin \oddsidemargin
\usepackage{todonotes}
\usepackage{color}
\usepackage{tikz-cd}
\usepackage{mathrsfs}
\usepackage[mathscr]{euscript}
\usepackage{bm}
\usepackage{bbm}
\usepackage{enumerate}
\usepackage{changepage}
\usepackage{wasysym}
\usepackage[utf8]{inputenc}
\inputencoding{utf8}
\usepackage{comment}
\usepackage{graphicx}
\usepackage{amsthm,amssymb,amsmath}
\usepackage{subfigure}
\allowdisplaybreaks[4]

\newtheorem{theorem}{Theorem}[section]
\newtheorem{lemma}[theorem]{Lemma}
\newtheorem{fact}[theorem]{Fact}
\newtheorem{proposition}[theorem]{Proposition}
\newtheorem{corollary}[theorem]{Corollary}

\theoremstyle{definition}
\newtheorem{definition}[theorem]{Definition}
\newtheorem{remarkx}[theorem]{Remark}
\newenvironment{remark}
  {\pushQED{\qed}\remarkx}
  {\popQED\endremarkx}

\newtheorem*{claimx}{Claim}
\newenvironment{claim}
  {\pushQED{\qed}\claimx}
  {\popQED\endclaimx}
\def\NN{\mathbb{N}}

\def\sgn{\mathrm{sgn}}

\def\RR{\mathbb{R}}
\def\TT{\mathbb{T}}
\def\ZZ{\mathbb{Z}}
\def\d{\,\mathrm{d}}
\def\e{\mathbbm{1}}

\def\id{\mathrm{id}_G}
\def\exp{\mathrm{exp}}
\def\mut{\widetilde{\mu}}
\def\dis{\mathfrak{d}}
\def\tri{\,\triangle\,}
\def\Stab{\mathrm{Stab}}

\title[Measure growth in compact groups and the Kemperman Inverse Problem]{Measure growth in compact semisimple  Lie groups and\\ the Kemperman Inverse Problem}

\author{Yifan Jing}
\address{Mathematical Institute, University of Oxford, UK}
\email{yifan.jing@maths.ox.ac.uk}

\author{Chieu-Minh Tran}
\address{Department of Mathematics, National University of Singapore, Singapore}
\email{trancm@nus.edu.sg}

\thanks{YJ was supported by Ben Green’s Simons Investigator Grant, ID:376201.}
\thanks{CMT was supported by Anand Pillay's NSF Grant-2054271.}

\date{}
\begin{document}

\begin{abstract}
Suppose $G$ is a compact semisimple Lie group, $\mu$ is the normalized Haar measure on $G$, and $A, A^2 \subseteq G$ are measurable. We show that  
$$\mu(A^2)\geq  \min\{1, 2\mu(A)+\eta\mu(A)(1-2\mu(A))\}$$ with the absolute constant $\eta>0$ (independent from the choice of $G$) quantitatively determined. We also show a more general result for connected compact groups without a toric quotient and resolve the Kemperman Inverse Problem from 1964.
\end{abstract}




\maketitle

\section{Introduction} 

\subsection{Statement of the main results} Throughout this introduction, $G$ is a {\it connected} and {\it compact} group,  $\mu$ is the normalized Haar measure on $G$, and $A, B \subseteq G$ are measurable. We also assume that the product sets $A^2:=\{a_1a_2 : a_1, a_2 \in A\}$ and   $AB:=\{ab : a\in A, b \in B\}$ are measurable whenever they appear.  Our first result reads:




\begin{theorem}\label{thm: maingap}
There is a constant $\eta>0$ such that the following hold. Suppose $G$ is a compact semisimple  Lie group, and $A\subseteq G$ a compact set with positive measure. Then: 
\begin{enumerate}[\rm (i)]
    \item 
$\mu(A^2)\geq (2+10^{-12})\mu(A)$ provided that $\mu(A) < 10^{-10^{1560}} $;
\item  $\mu(A^2)\geq \min\{1, 2\mu_G(A)+\eta\mu_G(A)|1-2\mu_G(A)|\}$.
\end{enumerate}
\end{theorem}

One may expect to obtain a product-theorem type result, that $\mu(A^2)\geq\min\{1,(2+\eta)\mu(A)\}$ in Theorem~\ref{thm: maingap}(ii). This is impossible, and a counterexample is constructed in Section 3: when $\mu(A)$ is close to $1/2$, $\mu(A^2)-2\mu(A)$ can be arbitrarily small. Thus the statement of Theorem~\ref{thm: maingap}(ii) is of the best possible form. On the other hand, a combination of Theorem~\ref{thm: maingap}(ii) and Kemperman's inequality does give us $\mu(A^3)>\min\{1,(3+\eta)\mu(A)\}$. For an asymmetric formulation involves two sets $A$ and $B$, see Theorem~\ref{thm: asymm 1.2}. 

Theorem~\ref{thm: maingap} can be contrasted with the situation when $G$ is a connected and compact solvable Lie group, equivalently, $G$ is the $d$-dimensional torus $\TT^d = (\RR/\ZZ)^d$. Here, the lower bound  $\mu(A^2)\geq \min\{1, 2\mu(A)\}$ given by the symmetric version of Kemperman inequality ~\cite{Kemperman} cannot be further improved: With $I \subseteq \TT$ the interval $[0, c]_{\TT}$, $\chi: \TT^d \to \TT$ a continuous surjective group homomorphism, and $A =\chi^{-1}(I)$,  one can check that $\mu(A^2)= 2\mu(A)$.

For the compact simple Lie group $\text{SO}_3(\RR)$, Theorem~\ref{thm: maingap} is a step toward the following conjecture by Breuillard and Green~\cite[Problem 78]{BenBook}: If $A \subseteq  \text{SO}_3(\RR)$ has sufficiently small measure, then 
$$\mu(A^2)> 3.99 \mu(A).$$ 
Note that measure growth rate around $4=2^2=2^{3-1}$ does happen in $\text{SO}_3(\RR)$ (3-dimensional), when $A$ is a small neighborhood around a maximal proper closed subgroup $H$ (1-dimensional) of $\text{SO}_3(\RR)$. Hence, one has an optimistic generalization of the above conjecture: the minimum measure growth rate in a connected compact Lie group $G$ is close to $2^{d-m}$ with $d$ the dimension of $G$ and $m$ the maximal dimension of a proper closed subgroup. Our later proof of Theorem~\ref{thm: maingap}, which goes through Theorem~\ref{thm: mainassymmetric}, essentially verifies this when $d-m=1$. We remark that recently the authors and Zhang~\cite{JTZ} shows that, in a {\it noncompact, connected}, and {\it unimodular} real algebraic group, the measure growth rate is at least $2^{d-m}$ with $d$ the topological dimension and $m$ the dimension of a maximal compact subgroup. This is another suggestion that the more general assertion is true. The problem for compact groups is expected to be much harder.


Theorem~\ref{thm: maingap} can also be seen as the continuous counterpart of  growth phenomena in finite simple groups of Lie type. In a celebrated work, Helfgott~\cite{Helfgott08} established the following: If $X \subseteq \mathrm{SL}_2(\mathbb{F}_p)$ is not contained in any proper subgroup, $|X| \leq p^{3-\delta}$ for some $\delta>0$,  then there are $c=c(\delta)$, and $\varepsilon=\varepsilon(\delta)$ such that
$$ |X^3| \geq c |X|^{1+\varepsilon}. $$
Generalizations to all finite simple groups of Lie type were obtained in~\cite{BGT11} and ~\cite{PS16} independently. These results provide important inputs for the so-called Bourgain--Gamburd expander machine~\cite{BG08,BG12}; see \cite{Breuillard18,Tao-expansion} for more background.


Our later proof, in fact, yields the following more general Theorem~\ref{thm: mainassymmetric} for groups that are not necessarily Lie.   The earlier Theorem~\ref{thm: maingap} is a consequence of Theorem~\ref{thm: mainassymmetric} and the fact that  $\TT$ is not a quotient of any simple Lie group.

\begin{theorem}\label{thm: mainassymmetric}
Let $K\geq 1$, then there is $\eta>0$ such that the following hold. Suppose $G$ is a connected compact group, $A,B$ are compact subsets of $G$ such that $\mu(AB)<1$, $0<\mu_G(B)^K\leq \mu_G(A)\leq \mu_G(B)^{1/K}$ and
\[
\mu(AB)<\mu(A) + \mu(B) +\eta \min\{\mu(A), \mu(B)\}(1-\mu(A)-\mu(B)).
\]
Then there is a continuous surjective group homomorphism $\chi: G \to \TT$ with 
$$\ker(\chi) \subseteq AA^{-1} \cap BB^{-1}.$$
\end{theorem}

 Theorem~\ref{thm: mainassymmetric}  strengthens some aspects of Gleason--Yamabe theorem in the  setting of small growth. For a connected and compact group $G$ and an open neighborhood $U$ of the identity, the Gleason--Yamabe theorem establishes the existence of a continuous and surjective group homomorphism $\pi: G \to L$ with $L$ a Lie group and $$\ker(\pi) \subseteq U.$$ Under the further assumption that $\mu(U^2) < (2+\eta)\mu(U)$, the open version of Theorem~\ref{thm: mainassymmetric} tells us we can choose $L$ to be $\TT$ but only having the weaker conclusion of $\ker(\pi) \subseteq UU^{-1}$. 

Prior to our work, Carolino~\cite{thesis} shows that if $U$ is a sufficiently small approximate subgroups (i.e., $U$ is open and precompact, $U= U^{-1}$, $\mu(U) \leq c(K)$, and $U^2$ can be covered by $K$ left translates of $U$ for a fixed constant $K$), then one can choose $L$ with $\dim(L)$ bounded from a bove by $d = d(K)$ and $\ker(\pi) \subseteq U^4$. Carolino's proof makes use of ultraproduct, and the dependency of $d$ on $K$ cannot be made effective. In view of  the assertion generalizing the Breuillard--Green conjecture, one would hope to choose $d(K) = \lfloor\log K \rfloor (\lfloor \log K \rfloor+1)/2 $.  The open version of  Theorem~\ref{thm: mainassymmetric} verifies this guess for $K \leq 2+ 10^{-12}$.

Theorem~\ref{thm: mainassymmetric} can be combined  with some geometric arguments to resolve the \emph{Kemperman Inverse Problem}, which we will now explain. The Kemperman inequality, proven in 1964, says that in a connected compact group
\begin{equation}\label{eq: intro kemperman}
\mu(AB) \geq \min\{\mu(A)+\mu(B),1\};
\end{equation}
Kemperman result is in fact more general, and applicable to connected and unimodular locally compact groups. One can view Kemperman's theorem as an analog of the Cauchy--Davenport theorem, and Kneser's theorem for abelian groups.

For the class of all connected compact groups,  the Kemperman inequality is sharp. When $I, J \subseteq \TT$ with $\mu_\TT(I)+\mu_\TT(J)<1$, $\chi: G \to \TT$ a continuous surjective group homomorphism, $A =\chi^{-1}(I)$, and $B =\chi^{-1}(J)$, one can check that $\mu(AB)= \mu(A) +\mu(B)$. If $I$, $J$, and $\chi$ are as above, $A$ is almost $\chi^{-1}(I) $ and $B$ is almost $\chi^{-1}(J)$, one then have $  \mu(AB)$ is almost $ \mu(A)+\mu(B)$. Combining Theorem~\ref{thm: mainassymmetric} and some geometric arguments, we can show that the only situations where the equality happens or nearly happens in the Kemperman inequality are these obvious ones. This answers a question by Kemperman since 1964~\cite{Kemperman}.

\begin{theorem}\label{thm:mainapproximate}
 Let $G$ be a connected compact group, and $A,B$ compact subsets of $G$ with $$0< \lambda=\min\{\mu_G(A),\mu_G(B),1-\mu_G(A)-\mu_G(B)\}.$$
  There is a constant $K=K(\lambda )$, not depending on $G$, such that for any $0\leq \varepsilon<1$, whenever we have $\delta\leq K\varepsilon$ and
  \[
  \mu(AB)\leq \mu(A)+\mu(B)+\delta\min\{\mu(A),\mu(B)\},
  \]
 there is a surjective continuous group homomorphism $\chi: G \to \TT$ together with two compact intervals $I,J\in \mathbb T$ with
 \[
 \mu_\TT(I)\leq (1+\varepsilon)\mu(A),\quad \mu_\TT(J)\leq (1+\varepsilon)\mu(B),
 \]
 and $A\subseteq \chi^{-1}(I)$, $B\subseteq\chi^{-1}(J)$. Moreover, with $\varepsilon = 0$, we can strengthen the conclusion to  $A= \chi^{-1}(I)$, $B =\chi^{-1}(J)$.
\end{theorem}

Some special cases of Theorem~\ref{thm:mainapproximate} for abelian groups are known before. See Tao~\cite{T18} for compact abelian groups, Griesmer~\cite{G19} for the disconnected setting, and~\cite{TT18} for an application of the result in~\cite{T18}  in Chowla's problem.  Christ--Iliopoulou~\cite{ChristIliopoulou} proved a sharp exponent result for compact abelian groups. Theorem~\ref{thm:mainapproximate} in particular answers a question by Tao for connected groups. We also obtain characterizations when the equality happens in~\eqref{eq: intro kemperman}, for the precise statements see Theorem~\ref{thm:mainequal} and Theorem~\ref{thm: kemperman noncompact}.

\subsection{Overview of the proof of Theorem~\ref{thm: mainassymmetric}} Given the data of a compact connected group $G$ and $A, B \subseteq G$ with sufficiently small measure and satisfying the ``nearly minimal expansion'' condition 
\[
\mu(AB) \approx \mu(A) + \mu(B),
\]
Theorem~\ref{thm: mainassymmetric} calls for the construction of an appropriate continuous and surjective group homomorphism  $\chi: G \to \TT$. 


To motivate our construction, let us assume we already know  there is some continuous and surjective group homomorphism $\chi: G \to \TT$ and interval $I\subseteq \TT$ such that we have $A \approx \chi^{-1}(I)$ (i.e.,  $\mu(A \triangle \chi^{-1}(I )) \approx 0$), but the precise description of $\chi$ is somehow forgotten. If $g\in G$ is an element such that $\chi(g)$ is in $(0, \mu(A)/10)_\TT$ but still much larger than the error in $\approx$, one has
$$ \Vert\chi(g)\Vert_\TT = \frac{1}{2} \mu_\TT( I \triangle \chi(g) I) \approx \frac{1}{2}  \mu(A \triangle g A). $$
Hence, for such element $g$, its image in $\chi(\TT)$ can be approximately determined up to a sign by considering $\mu(A \triangle g A)$. Of course, we do not know a priory that such $\chi$ exists as it is essentially the conclusion of Theorem~\ref{thm:mainapproximate} which is itself based on Theorem~\ref{thm: mainassymmetric}. However, the idea of constructing a group homomorphism from the measure difference $(g_1, g_2) \mapsto \mu(g_1A \triangle g_2A)$ does work. 

We implement the above idea in two steps.  In the first step,  we deduce from $\mu(A+B) \approx \mu(A)+\mu(B)$ enough information about the ``shapes'' of $A$ and $B$ to show that the pseudometric
$$d_A: G \times G \to \RR^{>0}, (g_1, g_2) \mapsto \mu(g_1A \triangle g_2A)  $$
resembles the Euclidean metric on $\TT$: With $\lambda =\mu(A)/17$,  for $g_1$, $g_2$, and $g_3$ in the neighborhood 
$$N(\lambda) := \{ g\in G : d_A(\id, g)< \lambda),$$ we have
$$ d_A(g_1, g_2) \approx | d_A(g_1, g_3) \pm d_A(g_2, g_3)| \text{ and } d_A(\id, g^2_1) \approx 2 d_A(\id, g_1). $$
We refer to the above property as ``almost linearity''.
Toward constructing the desired group homomorphism, we think of the above $N(\lambda)$ as the Lie algebra of $G$ assuming it is a Lie group, of $\RR^{>{0}}$ as the Lie algebra of $\TT$, and of the ``almost linearity'' property as saying that 
the map $G \to \RR^{>0}, g \mapsto d_A(\id, g)$ is approximately a homomorphism of Lie algebra without sign. In  the second step, we ``integrate''  $g \mapsto d_A(\id, g)$ to obtain the desired group homomorphism $\chi: G \to \TT$ handling the difficulty that $g \mapsto d_A(\id, g)$ is not quite a homomorphism of Lie algebra.

We now discuss the above two steps in more details focusing on the case where $G$ is already a (connected and compact) Lie group. The reduction to this case can be obtained using the Gleason--Yamabe theorem plus some geometric arguments.  Through a submodularity argument by Tao, we also reduce the problem to the case where $\mu(A) =\mu(B)$ is rather small, say $< 1/10^3$. Hence,  $\mu(AB)\approx 2\mu(A) $.






\subsection*{Step 1: Shape of minimally growth sets} \label{Sec: Overviewstep1} In this step, we will choose a suitable 1-dimensional torus subgroup $H$ of $G$, gain increasingly detailed geometric information about $A$ in relation to the cosets of $H$, and exploit the known conclusion of the Kemperman inverse problem of $H$ to show the ``almost linearity'' of $d_A$.

First, we choose a 1-dimensional torus subgroup $H$ of $G$ such that the  cosets of $H$ intersects $A$ and $B$``transversally in measure'': letting $\mu_H$ denoted the normalized Haar measure on $H$, and visualizing $G$ as the rectangle $H \times G/H$ as in Figure~\ref{fig:intuition}, we want the ``length'' $\mu_H(x^{-1} A \cap H)$ of each ``left fiber'' $A \cap xH$ of $A$ to be small (say $<1/10$), and a similar condition hold for ``right fibers'' of $B$. 

 \begin{figure}[h]
     \centering
     \includegraphics[width=5.9in]{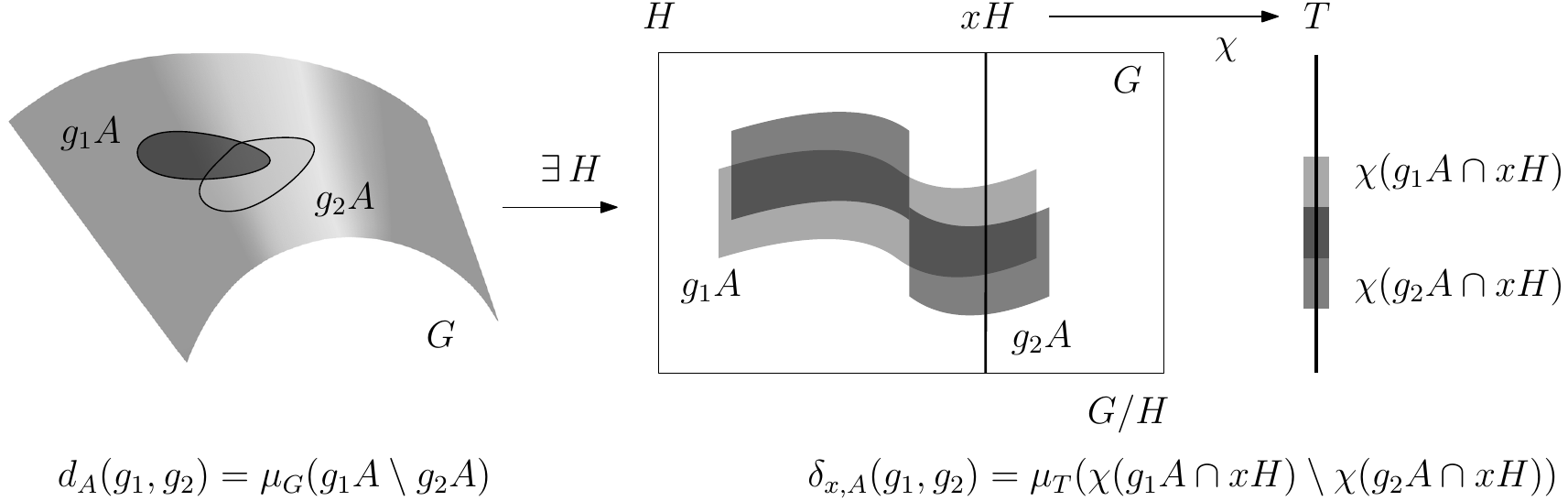}
     \caption{$d_A$ is a pseudometric, $H$ is the subgroup that intersects $A$ transversally in measure. We will show that the linear metric in $H$ can be lifted to a almost linear pseudometric in $G$.}
     \label{fig:intuition}
 \end{figure}

The argument for the existence of such $H$ is as follows. Using results in additive combinatorics, we can further arrange that $A$ is an open approximate group, i.e., $\id \in A$, $A = A^{-1}$, $A^2$ can be covered by $K$ left translates of $A$. The effect of this arrangement is that $\mu(A^n) < O_{K, n}(1)\mu(A)$. Suppose there is no $H$ as described above. Then one can show that $A^{N_1}$ contains all $1$-dimensional torus subgroups of $G$, implying $A^{2N_1} = G$. This brings a contradiction to the assumption that $\mu(A)$ is very small. 

Viewing the left cosets of  a 1-dimension torus subgroup of $G$ as the counterpart of lines with the same directions in $\RR^d$, the above can be seen as saying that $A$ cannot be a ``left Kakeya set'' and, likewise, $B$ cannot be a ``right Kakeya set''. We also note that this part can be generalized to show a similar conclusion when the growth rate $\approx 2$ is replaced by another fixed constant.

In the second part, we fix a 1-dimensional torus subgroup $H$ of $G$, and obtain a fine understanding of the shape of $A$ and its left translates relative to the left cosets of $H$: We want  $g_1A$ and $g_2A$ geometrically look like in the picture with $g_1, g_2 \in G$ in a suitable neighborhood of $\id$, and behave rigidly under translations.  
(For instance, we want the ``fibers'' of $A$ and $B$, i.e. intersections of $A$ or $B$ with coset of $H$, to be preimages of intervals of $T$, and all the nonempty ``fibers'' of $g_1A$ to have similar ``lengths''.)

The idea to get the second part is as follows. Suppose $g_1A$ does not have such shape. Choose $x$ uniformly at random, we expect that $\mu_H(Bx^{-1} \cap H)$ to be $\mu(B)$. It then follows from the Kemperman inequality and the solution of the Kemperman Inverse Problem for $H$ (via induction on dimension) that the measure of $g_1A(B \cap Hx)$, a subset of $g_1AB$, will already be much larger than $2\mu(A)$, a contradiction.

 The third and final part of Step 1 is to deduce the desired property of $d_A$ from the picture obtained in the second part. This step is rather technical. The key point is that, the nice geometric shape of $A$ can pass some of the properties of $\delta_{x,A}$ to the global pseudometric $d_A$.

 \subsection*{Step 2: From an almost linear pseudometric to a group homomorphism} \label{Sec: Overviewstep2} In this step, constructing the group homomorphism $\chi_A: G \to \TT$ from the ``almost linear'' pseudometric $d_A$ in  step 1. As mentioned earlier, the proof can be viewed as a discretization of the usual construction of a Lie group homomorphism from a Lie algebra homomorphism.

First, we construct a multi-valued function $\Phi_A$ from $G$ to  $\RR$, seen as the universal cover of $\TT$, with a suitable ``almost homomorphism'' property.  For an element $g \in G$, we represent it in ``the shortest way'' as a product of elements in $N(\lambda) := \{ g\in G : d_A(\id, g)< \lambda)$. The following notion capture this idea.
\begin{definition}
A sequence $(g_1,\dots,g_m)$  with  $g_i\in N(\lambda)$ is \emph{irreducible} if $g_{i+1}\cdots g_{i+j}\notin N(\lambda)$  for $2\leq j\leq 4$. 
\end{definition} 
Now, we set $\Omega_m(g)$ be the subset of $\RR$ consists of sums of the form
\begin{equation}
\sum_{i=1}^m \sgn(g_i)d(g_i,\id),
\end{equation}
where $(g_1,\dots,g_m)$ is an irreducible sequence with $g=g_1\cdots g_m$, and  $\sgn(g_i)$ is the \emph{relative sign} defined in Section~\ref{section:pseudometric}. Heuristically, $\sgn(g_i)$  specifies the ``direction'' of the translation by $g_i$ relative to a fixed element. Finally, we set $$\Omega(g)  =\bigcup_{m=1}^M \Omega_m(g),$$ where we will describe how to choose $M$ in the next paragraph. We do not quite have $\Omega(gg') = \Omega(g)+ \Omega(g')$, but a result along this line can be proven.

Call the maximum distance between two elements of $\Omega(g)$ the error of $\Omega(g)$. We want this error
to be small so as to be able to extract $\chi(g)$ from it eventually.
In the construction, each $d(g_i,\id)$ will increase the error of the image of $\Omega(g)$ by at least $\varepsilon$. Since the error propagate very fast, to get a sharp exponent error bound, we cannot choose very large $M$. To show that a moderate value of $M$ suffices, we construct a ``monitoring system'',  an irreducible sequence of bounded length, and every element in $G$ is ``close to'' one of the elements in the sequence. 
The knowledge required for the construction amounts to an understanding of the size and expansion rate of small stabilizers of $A$, which can also be seen as a refined Sanders--Croot--Sisask-type result for nearly minimally expanding sets.

From this we get a multi-valued almost group homomorphism $\Psi_A$ from $G$ to $\TT$ by quotienting $\RR$ by $\omega_A\ZZ$ for a value of $\omega_A \in \RR^{>0}$  which we now describe. In the special case where $d$ is linear (i.e. $\varepsilon =0$), the metric induced by $d$ in $G/\ker d$ is a multiple of the standard Euclidean metric in $\RR/\ZZ$ by a constant $\omega$. 
If we apply the machinery of irreducible sequence to this special case that $\varepsilon=0$, we will get for each $g\in G$ that $\Omega(g) = \psi(g) + \omega\ZZ$ for $\psi(g) \in [0, \omega)$. In particular, $\Omega(\id)= \omega\ZZ$, hence 
$
\omega=\inf\big|\Omega(\id) \setminus\{0\}\big|.
$

In general case when $\varepsilon>0$, we can define
\[
\omega=\inf\big|\Omega(\id) \setminus(-M\varepsilon,M\varepsilon)\big|.
\]
By further applying properties of irreducible sequences and the monitor lemma, we show  that for each $g$, $\Omega(g) \subseteq \psi(g) \omega\ZZ + (-M\varepsilon,M\varepsilon)$. For each $g\in G$, we then set $\Psi(g) \subseteq \TT$ to be $\Omega(g)/\omega\ZZ$.

Note that $\Psi$ is ``continuous'' from the way we construct it. We obtain $\chi$ from $\Psi$ by first extracting from $\Psi$ a universal measurable single valued almost-homomorphism $\psi$, then modifying $\psi$ to get a universal measurable group homomorphism $\chi$, and show that $\chi$ is automatically continuous.

\subsection{Structure of the paper}

The paper is organized as follows. Section~\ref{sec: Preliminaries} includes some facts about Haar measures and unimodular groups, which will be used in the subsequent part of the paper. A general version of the Kemperman inequality and its inverse theorem in tori are also included in Section~\ref{sec: Preliminaries}. In Section 3, we construct an example in $\mathrm{SO}_3(\RR)$ to show that when $1/2-\mu(A)$ is small, $\mu(A^2)-2\mu(A)$ can be arbitrarily small (Proposition~\ref{prop: example}). 
Section~4 allows us to arrange that in a minimally or a nearly minimally expanding pair $(A,B)$, the sets $A$ and $B$ have small measure (Proposition~\ref{prop: red to small sets}). 

In Section 5, when $\mu(AB)$ is small, we show that there is a closed one-dimensional torus $H$ intersects $A$ and $B$ that ``transversally in measures'' simultaneously provided they have small measures (Theorem~\ref{thm: kakeya new}). The smallness assumptions can be removed when $\mu(AB)-\mu(A)-\mu(B)$ is small (Proposition~\ref{prop: Torictransversal}). This section corresponds the first part of `Step 1' discussed above. Section 6 corresponds the second part of `Step 1', where we construct a locally path monotone almost linear pseudometric from the geometric shape of one of the given small measure growth sets. The geometric properties of small expansion sets are proved in Theorem~\ref{thm: fibers of same length 1newnew}. Using that, finally in Proposition~\ref{prop: almost linear metric from local} we manage to lift a metric from a smaller dimensional subgroup to a desired pseudometric over the original Lie group. 
Sections 7 and 8 correspond to `Step 2'. In Section 7, we prove some global properties from the path monotone locally almost linear pseudometric. In particular, we show that path monotonicity implies global monotonicity (Proposition~\ref{prop: localmonotoneimplyglobalmonotone}). In Section 8, we `integrate' the almost linear almost monotone pseudometric to extract a group homomorphism onto tori (Theorem~\ref{thm: homfrommeasurecompact}). The monitor lemma (Lemma~\ref{prop: lower bound on n}) is also proved in this section, which plays a key role on controlling the error terms. In Section 9, we prove the quotient domination theorem (Theorem~\ref{thm: criticality transfer}), which allow us to transfer the problem from locally compact groups into Lie groups. This can be seen as a Lie model theorem for sets with very small doublings (Corollary~\ref{cor: kernel control}). 

We prove Theorems~\ref{thm: maingap} and \ref{thm: mainassymmetric} in Section 10. In addition to the previous steps, we need to apply probabilistic arguments to control the images in the quotient (Lemma~\ref{lem: making projection small}). In Section 11, we prove Theorem~\ref{thm:mainapproximate}. Characterizations for sets that satisfy the equality in~\eqref{eq: intro kemperman} are also given in this section (Theorem~\ref{thm:mainequal}). In Section 12, we discuss the Kemperman Inverse Problem in connected unimodular noncomapct groups, and characterize sets and groups when the equality happen (Theorem~\ref{thm: kemperman noncompact}). 

\subsection{Notation and convention} 
Throughout, let $k$ and $l$ range over the set $\ZZ$ of integers, and $m$ and $n$ range over the set $\NN=\{0,1,\ldots\}$ of natural numbers. A constant in this paper is always a positive real number. For real valued quantities $r$ and $s$, we will use the standard notation $r= O(s)$ and $s= \Omega(r)$ to denote the statement that $r< Ks$ for an absolute constant $K$ independent of the choice of the parameters underlying $r$ and $s$.  If we wish to indicate dependence of the constant on an additional
parameter, we will subscript the notation appropriately.

We let $G$ be a locally compact group, and $\mu_G$ a left Haar  measure on $G$. We normalize $\mu_G$, i.e., scaling by a constant to get $\mu_G(G)=1$, when $G$ is compact. For $\mu_G$-measurable $A\subseteq G$ and a constant $\varepsilon$, 
we set
\[  \Stab^{\varepsilon}_{G}(A) =\{ g \in G : \mu_G( A \tri gA) \leq \varepsilon \} \  \text{ and } \  \Stab^{<\varepsilon}_G(A) =\{ g \in G : \mu_G( A \tri gA) < \varepsilon\} .
\] 
Suppose $A$, $B$, and $AB$ are $\mu_G$-measurable sets in $G$. The {\bf discrepancy} of $A,B$ in $G$ is defined by
\[
\dis_G(A,B)=\mu_G(AB)-\mu_G(A)-\mu_G(B). 
\]
When $G$ is connected, we always have $\dis_G(A,B)\geq0$.

Let $H$ range over closed subgroups of $G$. We let  $\mu_H$ denote  a left Haar measure of $H$, and normalize $\mu_H$ when $H$ is compact. Let $G/H$ and $H\backslash G$ be the left coset space and the right coset space with quotient maps
\[ \pi: G \to G/H\  \text{ and }\ \widetilde{\pi}: G \to H\backslash G
\] Given a coset decomposition of $G$, say $G/H$, a left-fiber of a set $A\subseteq G$ refers to $A\cap xH$ for some $xH\in G/H$. We also use $\mu_H$ to denote the fiber lengths in the paper, that is, we sometimes write $\mu_H(A\cap xH)$ to denote $\mu_H(x^{-1}A\cap H)$, as we want to capture the coset we studied. This will not cause problems as in the paper the ambient group is always unimodular. By saying that $G$ is Lie group, we mean $G$ is a real Lie group with finite dimension, and we denote $\dim(G)$ the real dimension of $G$.

\section{Preliminaries}\label{sec: Preliminaries}
 Throughout this section, we assume that $G$ is a \emph{connected} locally compact group (in particular, Hausdorff) equipped with a left Haar measure $\mu_G$, and $A,B \subseteq G$ are nonempty. 

\subsection{Locally compact groups and Haar measures}\label{sec: 3.1}

  Below are some basic facts about $\mu_G$ that we will use; see~\cite[Chapter~1]{Harmonicanalysis} for details:
\begin{fact} \label{fact: Haarmeasurenew}
Suppose $\mu_G$ is either a left or a right Haar measure on $G$. Then:
\begin{enumerate}[\rm (i)]
    \item If $A$ is compact, then $A$ is $\mu_G$-measurable and $\mu_G(A)< \infty$.
    \item If $A$ is open, then $A$ is $\mu_G$-measurable and $\mu_G(A)>0$.
    \item \emph{(Outer regularity)} If $A$ is Borel, then there is a decreasing sequence $(U_n)$ of open subsets of $G$ with $A \subseteq U_n$ for all $n$, and $\mu_G(A) = \lim_{n \to \infty} \mu_G(U_n). $
    \item \emph{(Inner regularity)} If $A$ is open, then there is an increasing sequence $(K_n)$ of compact subsets of $A$ such that $\mu_G(A) = \lim_{n \to \infty} \mu_G(K_n). $
     \item \emph{(Measurability characterization)} If there is an increasing sequence $(K_n)$ of compact subsets of $A$, and a decreasing sequence $(U_n)$ of open subsets of $G$ with $A \subseteq U_n$ for all $n$ such that $\lim_{n \to \infty} \mu_G(K_n) = \lim_{n \to \infty} \mu_G(U_n) $, then $A$ is measurable.
     \item \emph{(Uniqueness)} If $\mu'_G$ is another measure on $G$ satisfying the properties (1-5), then there is $C \in \RR^{>0}$ such that $\mu'_G = C\mu_G$.
    \item  \emph{(Continuity of measure under symmetric difference)} Suppose $A \subseteq G$ is measurable, then the function $G \to \RR, g \mapsto \mu_G(A \tri gA)$ is continuous.  
\end{enumerate}
\end{fact}
  We remark that the assumption that $G$ is connected implies that every measurable set is $\sigma$-finite (i.e., countable union of sets with finite $\mu_G$-measure). Without the connected assumption, we only have inner regularity for $\sigma$-finite sets. From Fact~\ref{fact: Haarmeasurenew}(vii), we get the following easy corollary:
\begin{corollary} \label{cor: Stabilizerisclosed}
Suppose $A$ is $\mu_G$-measurable and $\varepsilon$ is a constant. Then $\Stab^{\varepsilon}_G(A)$ is closed in $G$, while $\Stab^{<\varepsilon}_G(A)$ is open in $G$. In particular, $\Stab^{0}_G(A)$ is a closed subgroup of $G$.
\end{corollary}
  We say that $G$ is {\bf unimodular} if $\mu_G$ (and hence every left Haar measure on $G$)  is also a right Haar measure. The following is well known and can be easily verified: 

\begin{fact} \label{fact: measureofinverse}
If $G$ is unimodular, $A$ is $\mu_G$-measurable, then $A^{-1}$ is also $\mu_G$-measurable and $\mu_G(A) = \mu_G(A^{-1})$.
\end{fact}

 We use the following isomorphism theorem of topological groups. 
\begin{fact}\label{fact: first iso thm}
Suppose $G$ is a locally compact group, $H$ is a closed normal subgroup of $G$. Then we have the following.
\begin{enumerate}[\rm (i)]
    \item \emph{(First isomorphism theorem)} Suppose $\phi: G \to Q$ is a continuous surjective group homomorphism with $\ker \phi = H$.  Then the exact sequence of groups
        $$  1 \to H \to G \to Q \to 1 $$ 
    is an exact sequence of topological groups if and only if  $\phi$ is open; the former condition is equivalent to saying that  $Q$ is canonically isomorphic to $G/H$ as topological groups. 
    \item \emph{(Third isomorphism theorem)} Suppose $S \leq  G$ is closed, and $H \leq S$. Then $S/H$ is a closed subgroup of $G/H$. If $S\vartriangleleft G$ is normal, then $S/H$ is a normal subgroup of $G/H$, and we have the exact sequence of topological groups
    $$  1 \to S/H \to G/H \to G/S \to 1; $$
    this is the same as saying that $(G/H)/(S/H)$ is canonically isomorphic to $G/S$ as topological groups. 
\end{enumerate}
\end{fact}

Suppose $H$ is a closed subgroup of $G$. The following fact allows us to link Haar measures on $G$ with the Haar measures on $H$ for unimodular $G$ and $H$:

\begin{fact}[Quotient integral formula]\label{fact: QuotientIF} 
Suppose $H$ is a closed subgroup of $G$ with a left Haar measure $\mu_H$. If $f$ is a continuous function on $G$ with compact support, then
$$ xH \mapsto \int_H f(xh) \d\mu_H(x). $$
defines a function $f^H: G/H \to \RR$ which is continuous and has compact support. 
If both $G$ and $H$ are unimodular, then there  is unique invariant Radon measures $\mu_{G/H}$ on $G \slash H$ such that for all continuous function $f: G \to \RR$ with compact support, the following integral formula holds
    $$ \int_G f(x) \d\mu_G(x) = \int_{G/H} \int_H f(xh) \d\mu_H(h) \d\mu_{G /H}(xH). $$
A similar statement applies replacing the left homogeneous space $G/H$ with the right homogeneous space $H \backslash G$.
\end{fact}

  We can extend Fact~\ref{fact: QuotientIF} to measurable functions on $G$, but the function $f^H$ in the statement can  only be defined and is $\mu_{G/H}$-measurable $\mu_G$-almost everywhere. So, in particular, this problem applies to indicator function $\e_A$ of a measurable set $A$. This causes problem in our later proof and prompts us to sometimes restrict our attention to a better behaved subcollection of measurable subsets of $G$. We say that a subset of $G$ is {\bf $\sigma$-compact} if it is a countable union of compact subsets of $G$. 
 \begin{lemma} We have the following:
 \label{lem: mesurability}
\begin{enumerate}[\rm (i)]
    \item $\sigma$-compact sets are measurable.
    \item the collection of $\sigma$-compact sets is closed under taking countable union, taking finite intersection, and taking product set.
    \item For all $\mu_G$-measurable $A$, we can find a $\sigma$-compact subset $A'$ of $A$ such that $\mu_G(A') =\mu_G(A)$.
    \item Suppose $G$ is unimodular, $H$ is a closed subgroup of $G$ with a left Haar measure $\mu_H$, $A \subseteq G$ is $\sigma$-compact, and $\e_A$ is the indicator function of $A$. Then
 $aH \mapsto \mu_H(A \cap aH)$
defines a measurable function $\e^{H}_A: G/H \to \RR$. If $H$ is unimodular and, $\mu_{G/H}$ is the Radon measure given in Fact~\ref{fact: QuotientIF}, then  $$ \mu_G(A)=  \int_{G/H} \int_H \mu_H(A \cap aH) \d\mu_H(h) \d\mu_{G /H}(xH). $$
A similar statement applies replacing the left homogeneous space $G/H$ with the right homogeneous space $H \backslash G$.
\end{enumerate}

\end{lemma}
\begin{proof}
The verification of (i-iii) is straightforward. We now prove (iv). First consider the case where $A$ is compact. By Baire's Theorem, $\e_A$ is the pointwise limit of a monotone nondecreasing sequence of continuous function of compact support. If $f: G \to \RR$ is a continuous function of compact support, then the function $$f^H: G/H \to R, aH \mapsto \int_{H} f(ax) dx $$ is continuous with compact support, and hence measurable; see, for example, \cite[Lemma~1.5.1]{Harmonicanalysis}.  Noting that $\mu_H(A \cap aH) = \int_H \e_A(ax) dx$, and applying monotone convergence theorem, we get that $\e^H_A$
is the pointwise limit of a monotone nondecreasing sequence of continuous function of compact support. Using monotone convergence theorem again, we get $\e^H_A$ is integrable, and hence measurable. Also, by monotone convergence theorem, we get the quotient integral formula in the statement. 

Finally, the general case where $A$ is only $\sigma$-compact can be handled similarly, noting that $\e_A$ is then the pointwise limit of a monotone nondecreasing sequence of indicator functions of compact sets.
\end{proof}

  Suppose $H$ is a closed subgroup of $G$. Then $H$ is locally compact, but not necessarily unimodular. We use the following fact in order to apply induction arguments in the later proofs. 

\begin{fact}\label{fact: unimodular}
Let $G$ be a unimodular group. If $H$ is a closed normal subgroup of $G$, then $H$ is unimodular. Moreover, if $H$ is compact, then $G/H$ is unimodular.  
\end{fact}

  \subsection{Kemperman's inequality and the inverse problems} We will need a version of Kemperman's inequality for arbitrary sets.   Recall that the \emph{inner Haar measure} $\mut_{G}$ associated to $\mu_G$ is given by
  $$ \mut_G(A) = \sup \{ \mu_G(K) : K \subseteq A  \text{ is compact}.  \}$$
  The following is well known and can be easily verified:
 \begin{fact} \label{fact: Innermeasure}
 Suppose $\mut_G$ is the inner Haar measure associated to $\mu_G$. Then we have the following:
 \begin{enumerate}[\rm (i)]
     \item \emph{(Agreement with $\mu_G$)} If $A$ is measurable, then $\mut_G(A)=\mu(A)$. 
     \item \emph{(Inner regularity)} There is $\sigma$-compact $A' \subseteq A$ such that $$\mut_G(A)=\mut_G(A') = \mu_G(A).$$
     \item \emph{(Superadditivity)} If $A$ and $B$ are disjoint, then 
     $$\mut_G(A\cup B) \geq \mut_G(A)+\mut_G(B).$$
     \item \emph{(Left invariance)} For all $g\in G$, $\mut_G(gA)=\mut(A)$.
     \item \emph{(Right invariance)} If $G$ is unimodular, then for all $g\in G$, $\mut_G(Ag)=\mut(A)$.
 \end{enumerate}
 \end{fact}

It is easy to see that we can replace the assumption that $A$ and $B$ are compact in Kemperman's inequality in the introduction with the weaker assumption that $A$ and $B$ are $\sigma$-compact. Together with the inner regularity of $\mut_G$ (Fact~\ref{fact: Innermeasure}.2), this give us the first part of the following Fact~\ref{fact: GeneralKemperman}. The second part of Fact~\ref{fact: GeneralKemperman} follows from the fact that taking product sets preserves compactness, $\sigma$-compactness, and analyticity. Note that taking product sets in general does not preserve measurability, so we still need inner measure in this case.

\begin{fact}[Generalized Kemperman's inequality for connected groups] \label{fact: GeneralKemperman} 
Suppose $\widetilde{\mu}_G$ is the inner Haar measure on $G$, and $A, B \subseteq G$ are nonempty. Then
$$ \widetilde{\mu}_G(AB) \geq \min\{ \mut_G(A)+ \mut_G(B), \mut_G(G) \}. $$
Moreover, if $A$ and $B$ are compact, $\sigma$-compact, or analytic, then we can replace $\widetilde{\mu}_G$ with $\mu_G$.
\end{fact}

We need the following special case of Kneser's classification result~\cite{Kneser}, and the sharp dependence between $\varepsilon$ and $\delta$ is essentially due to Bilu~\cite{Bilu}.

\begin{fact}[Inverse theorem for $\TT^d$]\label{fact:inverse theorem torus}
Let $A,B$ be compact subsets of $\TT^d$. For every $\tau>0$, there is a constant $c=c(\tau)$ such that if
\[
\tau^{-1}\mu_{\TT^d}(A)\leq \mu_{\TT^d}(B)\leq \mu_{\TT^d}(A)\leq c,
\]
then either $\mu_{\TT^d}(A+B)\geq \mu_{\TT^d}(A)+2\mu_{\TT^d}(B)$, or there are compact intervals $I,J$ in $\TT$ with $\mu_\TT(I)=\mu_{\TT^d}(A+B)-\mu_{\TT^d}(B)$ and $\mu_\TT(J)=\mu_{\TT^d}(A+B)-\mu_{\TT^d}(A)$, and a continuous surjective group homomorphism $\chi:\TT^d\to\TT$, such that $A\subseteq \chi^{-1}(I)$ and $B\subseteq \chi^{-1}(J)$.
\end{fact}

When the group is a one dimensional torus $\TT$, a sharper result is recently obtained by Candela and de Roton~\cite{circle19}. The constant $c_\TT=3.1\cdot 10^{-1549}$ in the following fact is from an earlier result in $\ZZ/p\ZZ$ by Grynkiewicz~\cite{Grynkiewicz}. 

\begin{fact}[Inverse theorem for $\TT$]\label{fact: new inverse theorem torus}
There is a constant $c_\TT>0$ such that the following holds. If $A,B$ are compact subsets of $\TT$, with $\dis_\TT(A,B)<c_\TT$, and $\mu_\TT(A+B)+\dis_\TT(A,B)\leq 1$. Then there is a continuous surjective group homomorphism $\chi:\TT\to\TT$, and two compact intervals $I,J\subseteq \TT$ such that 
\[
\mu_\TT(I)\leq \mu_\TT(A)+\dis_G(A,B),\quad \mu_\TT(J)\leq \mu_\TT(B)+\dis_G(A,B),
\]
and $A\subseteq \chi^{-1}(I)$, $B\subseteq \chi^{-1}(J)$. 
\end{fact}

\section{An example in $\mathrm{SO}_3(\RR)$}

Given a compact group $G$ with a normalized Haar measure $\mu_G$, $A$ a compact subsets of $G$. As we mentioned in the introduction, one may expect the following product-theorem type inequality:
\[
\mu_G(A^2)\geq \min\{(2+\eta)\mu_G(A),1\},
\]
where $\eta>0$ is some absolute constant that does not depend on $A$ and $G$. 
In this section we will provide constructions to show that the above inequality is not true in general. More precisely, for any sufficiently small $\varepsilon>0$, we will construct a set $A\subseteq G:=\mathrm{SO}_3(\RR)$ with $\mu_G(A)<1/2-\varepsilon$, and $\mu_G(A^2)<1$.

Throughout this section, $u$, $v$, and $w$ possibly with decorations, range over unit vectors in $\RR^3$ (i.e., $u,v, w \in \RR^3$ have Euclidean norm $\|u\| = \|v\|= \|w\| =1$). We let $\alpha(u,v):= \arccos(u\cdot v)$ be the angle between $u$ and $v$. Let $\phi$ and $\theta$, possibly with decorations, range over $\RR$. We denote 
 $R_u^\phi$ the counter-clockwise rotation with signed angle $\phi$ where the axis and the positive direction (using right-hand rule) is specified by the unit vector $u$. We need the following Euler's rotation theorem:

\begin{fact} \label{fact: axisangle}
    Each nontrivial element  $g \in \mathrm{SO}_3(\RR)$ is of the form $R_u^\phi$ with $\phi \in (0, 2\pi)$. Moreover, the set of elements in $\RR^3$ fixed by such $g$ is  $\mathrm{span}(u) =\{\lambda u \mid \lambda \in \RR\}$.    
\end{fact}

 The following lemma is essentially a variation of the above fact.

\begin{lemma} \label{lem: good representation}
Let $u$ be a fixed unit vector. Then every $g\in \mathrm{SO}_3(\RR)$ is of the form $R^\theta_v R^\phi_u $ with $v$ a unit vector orthogonal to $u$. Likewise, every $g\in \mathrm{SO}_3(\RR)$ is of the form $R^{\phi'}_u R^{\theta'}_{w} $ with $w$ a unit vector orthogonal to $u$.
\end{lemma}

\begin{proof}
We prove the first assertion. Choose $v$ to be the normal vector of $\mathrm{span}(u, g(u))$. Then $v$ is orthogonal to $u$, and there is $\theta$ such that $g(u)  = R^{\theta}_v (u) $. Now, $g^{-1} R^\theta_v$ fixes $u$, so by Fact~\ref{fact: axisangle}, $g^{-1} R^\theta_v = R^{-\phi}_u$ for some $\phi \in \RR$. Thus, $R^\theta_v R^\phi_u$.

The second assertion can be obtained by applying the first assertion to $g^{-1}$. 
\end{proof}

We will need the following inequality:

\begin{lemma}\label{lem: example 2} Let $u$ be a fixed unit vector. 
    Suppose $g_1, g_2 \in \mathrm{SO_3}(\RR)$ are such that $\alpha(u, g_i(u)) \leq  \pi/2$ for $i \in {1, 2}$. Then we have $\alpha(z, g_1g_2(z)) \leq \alpha(z, g_1(z)) + \alpha(z, g_2(z)) .$
\end{lemma}

\begin{proof}
Applying Lemma~\ref{lem: good representation}, we can write $g_2$ as $R^{\theta_2}_vR^{\phi_2}_u $ and $g_1$ as $R^{\phi_1}_uR^{\theta_1}_w  $ with $v$ and $w$ orthogonal to $u$. 
It is easy to see that 
$$ \alpha(u, g_1(u)) = \alpha(u, R^{\theta_1}_w(u)) \text{ and } \alpha(u, g_2(u)) = \alpha(u, R^{\theta_2}_v(u)).$$
On the other hand, 
$$ \alpha(u, g_1g_2(u)) = \alpha(u, (R^{\theta_1}_wR^{\theta_2}_v)(u)). $$ 
The triangle inequality in term of angles gives us 
$$\alpha(u, (R^{\theta_1}_wR^{\theta_2}_v)(u)) \leq \alpha(u, R^{\theta_2}_v(u)) +  \alpha(R^{\theta_2}_v(u), (R^{\theta_1}_wR^{\theta_2}_v)(u)).   $$
As $w$ is orthogonal to $u$ and possibly not to to $R^{\theta_2}_v(u)$, we have 
$$\alpha(R^{\theta_2}_v(u), (R^{\theta_1}_wR^{\theta_2}_v)(u)) \leq \alpha(u, R^{\theta_1}_w(u)). $$
The desired conclusion follows.
\end{proof}

The next proposition is our construction.

\begin{lemma} \label{lem: example}
Let $u$ be a unit vector, $\varepsilon\in (0,\pi/2)$, and $A = \{g \in \mathrm{SO}_3(\RR) \mid \alpha(u, g(u)) \leq \pi/2 -\varepsilon\} $. Then 
$$A^2 = \{g \in \mathrm{SO}_3(\RR) \mid \alpha(u, g(u)) \leq \pi -2\varepsilon \}. $$
\end{lemma}
\begin{proof}
The inclusion $A^2 \subseteq \{g \in \mathrm{SO}_3(\RR) \mid \alpha(u, g(u)) \leq \pi -2\varepsilon \}$ is immediate from Lemma~\ref{lem: example 2}. Now suppose $g \in \mathrm{SO}_3(\RR)$ satisfies $\alpha(u, g(u)) \leq \pi -2\varepsilon$. Then, Lemma~\ref{lem: good representation} yields $g= R^\theta_v R^\phi_u $ with $v$ orthogonal to $u$ and $\theta \in [2\varepsilon- \pi, \pi -2\varepsilon]$. We can rewrite $g = R^{\theta/2}_v (R^{\theta/2}_vR^\phi_u) $.
The other inclusion follows.
\end{proof}

We now prove the main statement of this section:
\begin{proposition} \label{prop: example}
  For every $\varepsilon>0$ there is a $\delta\gg \varepsilon^2$ and $A\subseteq \mathrm{SO}_3(\RR)$ such that $\mu(A)\geq 1/2-\varepsilon$ and $\mu(A^2)<1-\delta$. 
\end{proposition}
\begin{proof}
 Let $u$ and $A$ be as in Lemma~\ref{lem: example}. 
 Recall that the group $\mathrm{SO}_3(\RR)$ acts transitively on the 2-sphere $S^2$ consisting of all unit vectors in $\RR^3$. Let $ T \leq \mathrm{SO}_3(\RR)$  be the stabilizer of $u$. Then, $\mathrm{SO}_3(\RR)/T$ can be identified with $S^2$. With $\pi A$ and $\pi A^2$ be the projections of $A$ and $A^2$ in $\mathrm{SO}_3(\RR)/T$ respectively, we have 
 $$ \pi A =\{v \mid \alpha(u,v) \leq \frac{\pi}{2}-\varepsilon\} \text{ and } \pi A^2 =\{v \mid \alpha(u,v) \leq \pi-2\varepsilon\}. $$
 Let $\nu$ be the Radon measure induced by $\mu$ on $\mathrm{SO}_3(\RR)/T$ via the quotient integral formula.
From the way we construct $A$ and Lemma~\ref{lem: example}, we have $A=AT$ and $A^2=A^2T$.  Hence, $\mu(A)= \nu(\pi A)$ and $\mu(A^2)=\nu(\pi A^2)$. Finally, note that $\nu$ is the normalized Euclidean measure, so an obvious computation yields the desired conclusion.
\end{proof}

By Proposition~\ref{prop: example}, one can verify that our statement in Theorem~\ref{thm: maingap}(ii),
\begin{equation}\label{eq: example thm 1.1}
\mu_G(A^2)\geq \min\{1, 2\mu_G(A)+\eta\mu_G(A)|1-2\mu_G(A)|\},
\end{equation}
 is of the best possible form up to a constant factor.
More precisely, when $\mu_G(A)$ is small, the lower bound on $\mu_G(A^2)-2\mu_G(A)$ given in~\eqref{eq: example thm 1.1} is linear on $\mu_G(A)$. This is best possible, as a small neighborhood of identity in a compact simple Lie group of dimension $d$ is $2^d$-approximate as the multiplication at a small scale behaves like addition in $\RR^d$. On the other hand, when $1/2-\mu(A)$ is small, the lower bound on $\mu_G(A^2)-2\mu_G(A)$ given in~\eqref{eq: example thm 1.1} is linear on $1-2\mu_G(A)$, while Proposition~\ref{prop: example} shows that this is best possible.

\section{Some reductions}\label{sec: reduction to small}
  Throughout this section, $G$ is a connected compact (not necessarily Lie) group, $\mu_G$ is the normalized Haar measure on $G$ (that is $\mu_G(G)=1$), and $A,B \subseteq G$ are $\sigma$-compact sets with positive measure. We will prove a submodularity inequality and several applications in this section. Similar results were also proved by Tao \cite{T18} to obtain an inverse theorem of Kneser's inequality in connected compact abelian groups, and by Christ and Iliopoulou~\cite{ChristIliopoulou} to prove an inverse theorem of the Riesz--Sobolev inequality in connected compact abelian groups. 
  
   Let $f,g: G\to \mathbb C$ be functions. For every $x\in G$, we define the {\bf convolution} of $f$ and $g$ to be
\[
f*g(x)=\int_G f(y)g(y^{-1}x) \d\mu_G(y).
\]
Note that $f*g$ is not commutative in general, but associative by Fubini's Theorem. We first prove the following easy fact, which will be used several times later in the paper. 
  
\begin{lemma}\label{lem:trans}
Let $t$ be any real numbers such that $\mu_G(A)^2\leq t\leq \mu_G(A)$. Then there are $x,y\in G$ such that $\mu_G(A\cap(xA))=\mu_G(A\cap(Ay))=t$.
\end{lemma}
\begin{proof}
Consider the maps: 
$$\pi_1: x\mapsto \e_A*\e_{A^{-1}}(x)=\mu_G(A\cap (xA)), \text{ and } \pi_2:y\mapsto \e_{A^{-1}}*\e_{A}(y)=\mu_G(A\cap (Ay)).$$  
By Fact~\ref{fact: Haarmeasurenew}, both $\pi_1$ and $\pi_2$ are continuous functions, and equals to $\mu_G(A)$ when $x=y=\text{id}_G$. By Fubini's theorem,
\[
\mathbb E \, (\e_A*\e_{A^{-1}})=\mu_G(A)^2=\mathbb E \, (\e_{A^{-1}}*\e_{A}).
\]
Then the lemma follows from the intermediate value theorem, and the fact that $G$ is connected.
\end{proof}

Recall that $\dis_G(A,B)=\mu_G(AB)-\mu_G(A)-\mu_G(B)$ is the discrepancy of $A$ and $B$ on $G$.  The following property is sometimes referred to as {\bf submodularity} in the literature. Note that this is not related to modular functions in locally compact groups or the notion of modularity in model theory.

\begin{lemma}[Submodularity]\label{lem:iep}
Let $\gamma_1,\gamma_2>0$, and $A,B_1,B_2$ are $\sigma$-compact subsets of $G$. 
Suppose that $\dis_G(A,B_1)\leq \gamma_1$, $\dis_G(A,B_2)\leq \gamma_2$, and
\[
\mu_G(B_1\cap B_2)>0, \quad \text{ and }\ \mu_G(A)+\mu_G(B_1\cup B_2)\leq 1.
\]
Then both $\dis_G(A,B_1\cap B_2)$ and $\dis_G(A,B_1\cup B_2)$ are at most $\gamma_1+\gamma_2$.
\end{lemma}
\begin{proof}

Observe that for every $x\in G$ we have
\begin{equation*}
\e_{AB_1}(x)+\e_{AB_2}(x)\geq \e_{A(B_1\cap B_2)}(x)+\e_{A(B_1\cup B_2)}(x),
\end{equation*}
which implies 
\begin{equation}\label{eq:observation}
\mu_G(AB_1)+\mu_G(AB_2)\geq\mu_G(A(B_1\cap B_2))+\mu_G(A(B_1\cup B_2)).
\end{equation}
By the fact that $\dis_G(A,B_1)\leq \gamma_1$ and $\dis_G(A,B_2)\leq \gamma_2$, we obtain
\begin{align*}
    \mu_G(AB_1)\leq \mu_G(A)+\mu_G(B_1)+\gamma_1,\text{ and }
     \mu_G(AB_2)\leq \mu_G(A)+\mu_G(B_2)+\gamma_2.
\end{align*}
Therefore, by equation (\ref{eq:observation}) we have
\begin{align*}
    &\mu_G(A(B_1\cap B_2))+\mu_G(A(B_1\cup B_2))\\
    \leq&\, 2\mu_G(A)+\mu_G(B_1\cap B_2)+\mu_G(B_1\cup B_2)+\gamma_1+\gamma_2.
\end{align*}
On the other hand, as $\mu_G(B_1\cap B_2)>0$ and $\mu_G(A)+\mu_G(B_1\cup B_2)\leq1$, and using Kemperman's inequality, we have
\[
\mu_G(A(B_1\cap B_2))\geq \mu_G(A)+\mu_G(B_1\cap B_2),
\]
and 
\[
\mu_G(A(B_1\cup B_2))\geq \mu_G(A)+\mu_G(B_1\cup B_2).
\]
This implies
\[
\mu_G(A(B_1\cap B_2))\leq \mu_G(A)+\mu_G(B_1\cap B_2)+\gamma_1+\gamma_2,
\]
and 
\[
\mu_G(A(B_1\cup B_2))\leq \mu_G(A)+\mu_G(B_1\cup B_2)+\gamma_1+\gamma_2.
\]
Thus we have $\dis_G(A,B_1\cap B_2), \dis_G(A,B_1\cup B_2)\leq \gamma_1+\gamma_2$. 
\end{proof}

When $G$ is compact, sometimes it is not easy to guarantee $\mu_G(A)+\mu_G(B_1\cup B_2)\leq 1$. On the contrary, we have the following simple lemma.

\begin{lemma}\label{lem: when a+b>G}
If either $A, B \subseteq G$ are measurable and $\mu_G(A)+\mu_G(B)> 1$ or  $A, B \subseteq G$ are compact and $\mu_G(A)+\mu_G(B)\geq 1$,  then $AB=G$.
\end{lemma}
\begin{proof}
The cases where either $A$ or $B$ has measure zero are immediate, so we assume that $A,B$ both have positive measures.
 Suppose $g$ is an arbitrary element of $G$.  It suffices to show that $A^{-1}g$ and $ B$ has nonempty intersection. As $G$ is unimodular, $\mu_G(A) =\mu_G(A^{-1})$ by Fact~\ref{fact: measureofinverse}. Hence $\mu_G(A^{-1}g)  +\mu_G(B) = \mu_G(G)$. If $\mu_G(A^{-1}g \cap B)>0$, then we are done. Otherwise, we have $\mu_G(A^{-1}g \cap B)=0$, and so $\mu_G(A^{-1}g \cup B) =\mu_G(G)$ by the inclusion-exclusion principle. As $A$ and $B$ are compact, $A^{-1}g \cup B$ is also compact, and the complement of $A^{-1}g \cup B$ is open. Since nonempty open sets has positive measure,  $\mu_G(A^{-1}g \cup B) =\mu_G(G)$ implies $A^{-1}g \cup B = G$. Now, since $G$ is connected, we have $A^{-1}g \cap B$ must be nonempty.
\end{proof}

By applying the submodularity lemma repeatedly, the following proposition allows us to control the size of our small expansion sets. 
\begin{proposition}\label{prop: red to small sets}
Suppose $A,B$ are $\sigma$-compact subsets of $G$, $\gamma,c>0$ are real numbers, $\dis_G(A,B)\leq \gamma$, and
\[
\mu_G(A),\mu_G(B),1-\mu_G(A)-\mu_G(B)\geq c.
\]Then for every $\kappa\in(0,1/4)$, there are $\sigma$-compact sets $A', B'\subseteq G$ such that $\mu_G(A')=\mu_G(B')=\kappa$ and $\dis_G(A',B')\leq \gamma\frac{\log(1/\kappa)^2}{c}$. Moreover, if $\mu_G(A)>\kappa$, we have that $A'\subseteq A$; if $\mu_G(A)\leq \kappa$, we have that $A'\supseteq A$. Same holds for $B$ and $B'$. 
\end{proposition}
\begin{proof}
We assume first that $\mu_G(A)\geq \mu_G(B)> \kappa$. By Lemma~\ref{lem:trans}, there is $g\in G$ such that $
\mu_G(A\cap gA)=\mu_G(A)^2. 
$
To ensure that $\mu_G(A\cup gA)\leq 1-\mu_G(A)$ in order to apply Lemma~\ref{lem:iep}, we choose $g_0$ in such a way that 
\[
\mu_G(A\cap g_0A) = \max\{\mu_G(A)^2, 2\mu_G(A)+\mu_G(B)-1,\kappa\}.
\]
Let $A_1:= A\cap g_0A$. We will stop if $\mu_G(A_1)=\kappa$. If $\mu_G(A_1)=\mu_G(A)^2$, we choose $g_1$ such that 
\[
\mu_G(A_1\cap g_1A_1) = \max\{\mu_G(A_1)^2, \kappa\}.
\]
We can always choose such a $g_1$, as we have $\mu_G(A)^2\geq 2\mu_G(A)+\mu_G(B)-1$, and by convexity this implies $\mu_G(A)^4\geq 2\mu_G(A)^2+\mu_G(B)-1$, which can be rewritten as $$\mu_G(A_1)^2\geq 2\mu_G(A_1)+\mu_G(B)-1.$$ Similarly, if $\mu_G(A_1)^2>\kappa$, we let $A_2=A_1\cap g_1A_1$, and choose $g_2$ in such a way that $\mu_G(A_2\cap g_2A_2) = \max\{\mu_G(A_2)^2, \kappa\}$. Thus, by applying Lemma~\ref{lem:iep} at most $\log\log 1/\kappa$ times, we will get $A'$ with $\mu_G(A')=\kappa$ and $\dis_G(A',B)\leq \log(1/\kappa)\gamma$. We then apply the same argument to $B$ and get a set $B'$ with $\mu_G(B')=\kappa$ and $\dis_G(A',B')\leq \log(1/\kappa)^2\gamma$ as desired. 

We now consider the case when $2\mu_G(A)+\mu_G(B)-1>\mu_G(A)^2$. We choose $g_0\in G$ such that 
\[
\mu_G(A\cap g_0A)=2\mu_G(A)+\mu_G(B)-1\leq \mu_G(A)-c. 
\]
We also define $A_1$ to be $A_0\cap g_0A_0$. In the next step, we choose $g_1\in G$ such that 
\[
\mu_G(A_1\cap g_1A_1) = \max\{\mu_G(A_1)^2, 2\mu_G(A_1)+\mu_G(B)-1,\kappa\}.
\]
It is still possible that we have $2\mu_G(A_1)+\mu_G(B)-1>\mu_G(A_1)^2$ at this step. Note that $2\mu_G(A_1)+\mu_G(B)-1\leq \mu_G(A_1)-2c$. We apply this procedure $t$ times, and assume in step $t$ we still get $2\mu_G(A_t)+\mu_G(B)-1>\mu_G(A_t)^2$. As $2\mu_G(A_t)+\mu_G(B)-1\leq \mu_G(A_t)-2^tc$, this implies
\[
\mu_G(A)-2^{t+1}c\geq (\mu_G(A)-2^tc)^2, 
\]
so $2^t$ cannot be larger than $\sqrt{\mu_G(A)-\mu_G(A)^2}/c$. This means, after applying the same procedure at most $\log(\sqrt{\mu_G(A)-\mu_G(A)^2}/c)$ times, we arrive at a set $A''$ with $\mu_G(A'')\geq\kappa$, $\mu_G(A'')^2\geq 2\mu_G(A'')+\mu_G(B)-1$, and $$\dis_G(A'' ,B)\leq \frac{\gamma\sqrt{\mu_G(A)-\mu_G(A)^2}}{c}<\frac{\gamma}{c} .$$ Now we choose $g\in G$ such that
\[
\mu_G(A''\cap gA'') = \max\{\mu_G(A'')^2,\kappa\}. 
\]
Similarly, after at most $\log\log 1/\kappa$ steps, we obtain our set $A'$ with $\mu_G(A')=\kappa$ and $$\dis_G(A',B)\leq \frac{\gamma\sqrt{\mu_G(A)-\mu_G(A)^2}\log1/\kappa}{c}<\frac{\gamma\log(1/\kappa)}{c}. $$
Next, we consider the set $B$. Note that now we must have \[ 2\mu_G(B)+\mu_G(A'')-1\leq \mu_G(A)+\mu_G(B)+\kappa-1\leq \kappa-c<\kappa. 
\]
Thus Lemma~\ref{lem:trans} implies that there is $g_1\in G$ such that
\[
\mu_G(B\cap Bg_1) = \max\{\mu_G(B)^2,\kappa\}, 
\]
and $\mu_G(B\cup Bg_1)+\mu_G(A'')\leq 1$. We similarly define $B_1=B\cap Bg_1$, and if $\mu_G(B_1)>\kappa$, we define $B_2=B_1\cap B_1g_2$ with $\mu_G(B_2) = \max\{\mu_G(B_1)^2,\kappa\}.$ By doing the same procedure at most $\log(1/\kappa)$ times, we obtain $B'$ with $\mu_G(B')=\kappa$ and $\dis_G(A',B')<\gamma\frac{(\log 1/\kappa)^2}{c}$. 

The cases when $\mu_G(A)>\kappa>\mu_G(B)$ and $\kappa>\mu_G(A)>\mu_G(B)$ can be handled similarly, by replacing $A\cap gA$ and $B\cap B_g$ by $A\cup gA$ and $B\cup Bg$ suitably, and we omit the details here. 
\end{proof}

In particular, we have the following corollary, which can be used in the case that we only need to reduce one of our sets to smaller sets.
\begin{corollary}\label{cor: small sets may not equal}
Suppose $A,B$ are $\sigma$-compact subsets of $G$, $\gamma,K>1$ are real numbers, $\dis_G(A,B)\leq \gamma$, $\mu_G(A)+\mu_G(B)\leq 1/4$, and $\mu_G(B)^K\leq\mu_G(A)\leq\mu_G(B)$. 
Then there is a $\sigma$-compact set $B'\subseteq B$ such that $\mu_G(A)=\mu_G(B')$ and $\dis_G(A,B')\leq K\gamma$. 
\end{corollary}
\begin{proof}
We apply Lemma~\ref{lem:iep} $\log K$ times and obtain a decreasing sequence of sets $B_1,\dots,B_{\log K}$ so that $\mu_G(B_{i+1})=\mu_G(B_i)^2$ for every $1\leq i\leq \log K$. Let $B'=B_{\log K}$. Thus we can arrange that $\mu_G(B')=\mu_G(A)$ and $\dis_G(A,B')\leq K\gamma$. 
\end{proof}


Next, we will consider a special case when the ambient group $G$ is the one-dimensional torus $\TT:=\RR/\ZZ$. We will prove that if $A,B$ are $\sigma$-compact sets in $\TT$ with small $\dis_\TT(A,B)$, and $A$ is close to an interval in $\TT$, then $B$ should also be close to an interval. This result can be easily derived 
from the inverse theorem of Kneser's inequality by Tao~\cite{T18} or by Christ and Iliopoulou~\cite{ChristIliopoulou} (for a quantitatively better version). Here we will record a simple  proof of this result, as we only need the result for $\TT$. 
\begin{lemma}\label{lem: stability of character in T}
Let $c_\TT$ be the constant in Fact~\ref{fact: new inverse theorem torus}, and $\delta<c_\TT$. Suppose $A,B$ are $\sigma$-compact subsets of $\TT$ with $\mu_\TT(A),\mu_\TT(B)<1/4$ and $\dis_\TT(A,B)<\delta\min\{\mu_\TT(A),\mu_\TT(B)\}$. If $I$ is a compact interval in $\TT$ satisfying $A\subseteq I$ and $\mu_\TT( I\setminus A)<\dis_\TT(A,B)$, then there is a compact interval $J\subseteq \TT$ such that $B\subseteq J$ and $\mu_\TT(J\setminus B)<\dis_\TT(A,B)$.
\end{lemma}

\begin{proof}
We apply Fact~\ref{fact: new inverse theorem torus}, there are compact intervals $I',J'\subseteq \TT$ and a continuous surjective group homomorphism $\chi:\TT\to\TT$ such that $A\subseteq \chi^{-1}(I')$, $B\subseteq \chi^{-1}(J')$, and 
\[
\mu_\TT(I')\leq \mu_\TT(A)+\dis_\TT(A,B),\quad \mu_\TT(J')\leq \mu_\TT(B)+\dis_\TT(A,B). 
\]
It remains to show that $\chi=\mathrm{id}$, which would imply that $B\subseteq J'$. Note that all the continuous surjective group homomorphisms from $\TT$ to itself are $\times n$ maps for some positive integer $n$. If $\chi\neq \mathrm{id}$, we have $\mu_\TT(I\cap \chi^{-1}(I'))\leq \mu_\TT(I)/2$ as $\mu_\TT(I)<1/3$, and this contradicts the fact that $\delta$ is small. 
\end{proof}

\section{Toric  transversal intersection in measure}
In this section, we will control the global geometric structures of sets $A$ and $B$ if $\mu_G(AB)$ is small. 
We first prove the following lemma, which can be seen as a corollary of the quotient integral formula (Fact~\ref{fact: QuotientIF}).

\begin{lemma}\label{lem:int}
Let $G$ be a connected compact group, and $A,B$ are $\sigma$-compact subsets of $G$ with positive measures. 
For every $b\in G$, the following identity holds
 $$ \mu_G\big(A (B \cap Hb)\big) = \int_G \mu_H\big(( A \cap aH )(B \cap Hb)\big) \d\mu_G(a). $$
\end{lemma}
\begin{proof}
Let $C = A (B \cap Hb)$.  From Fact~\ref{fact: QuotientIF}, one has $$\mu_G(C)= \int_G \mu_H( C \cap a b^{-1}H b)  \d\mu_G(a)=  \int_G \mu_H( C \cap a H b)  \d\mu_G(a).$$  Hence, it suffices to check that  $$C \cap a Hb = ( A \cap aH )(B \cap Hb) \quad \text{ for all } a, b \in G.$$
The backward inclusion is clear. Note that $aH Hb = aHb = ab (b^{-1}Hb)$ for all $a$ and $b$ in $G$. For all $a$, $a'$, and $b$ in $G$, we have 
we have  $a'Hb = aHb$ when $aH=a' H$ and $ aHb \cap a'Hb = \varnothing$ otherwise.  
An arbitrary element $c \in C$ is in $(A \cap a'H)(B \cap Hb)$ for some $a'\in A$. Hence, if $c$ is also in $aHb$, we must have $a'H =aH$. So we also get the forward inclusion.
\end{proof}

  We use $\mu_{G/H}$ and $\mu_{H \backslash G}$ to denote the Radon measures on $G/H$ and $H \backslash G$ such that we have the quotient integral formulas (Fact~\ref{fact: QuotientIF} and Lemma~\ref{lem: mesurability}(vi)). We also remind the reader that we normalize the measure whenever a group under consideration is compact, and $\pi: G \to G/H$, $ \widetilde{\pi}: G \to H\backslash G$ are quotient maps. Hence we have
$$ \mu_G(AH)=\mu_{G/H}(\pi A)\quad \text{and}\quad \mu_{G}(HB) =\mu_{H \backslash G}( \widetilde{\pi} B). $$

Suppose $r$ and $s$ are in $\RR$, the sets $ A_{(r,s]}$ and $ \pi A_{(r,s]}$ are given by 
$$  A_{(r,s]}  := \{a \in A: \mu_H(A \cap aH) \in (r,s]  \}  $$
and
$$ \pi A_{(r,s]}:=  \{aH \in G \slash H : \mu_H( A \cap aH) \in (r,s]   \}. $$ 
In particular, $\pi A_{(r,s]}$ is the image of $A_{(r,s]}$ under the map $\pi$. By Lemma~\ref{lem: mesurability}, $\pi A_{(r,s]}$ is $\mu_{G/H}$-measurable, and $A_{(r,s]}H = \pi^{-1}(\pi A_{(r,s]})$ and $A_{(r,s]} $ are $\mu_G$-measurable.
\begin{lemma}
\label{lem: A=AH toric exp prepare}
Let $G$ be a connected compact group, and let $\kappa>0$, $A,B$ be $\sigma$-compact subsets of $G$ with $\mu_G(A)=\mu_G(B)=\kappa$. 
Suppose $\mu_G(AB)<M\kappa$, and $\lambda<1$ is a constant, and either there is $a\in A$ such that $\mu_H(A\cap aH)>\lambda$, or there is $b\in B$ such that $\mu_H(B\cap Hb)>\lambda$. Then,
\[
\min\bigg\{\frac{\mu_G(A)}{\mu_{G/H}(\pi A)}, \frac{\mu_G(B)}{\mu_{H\backslash G}(\widetilde{\pi} B)}\bigg\}\geq \frac{\lambda}{(M+2)^2}.
\]
\end{lemma}
\begin{proof}
Without loss of generality, suppose $\mu_H(B\cap Hb)>\lambda$ for a fixed $b\in B$. By the quotient integral formula, $\kappa=\mu_G(A)$ is at least

\[
\int_{\pi A_{(1/2,1]}} \mu_H(A\cap xH) \d\mu_{G/H}xH> \frac{1}{2}\mu_{G/H}(\pi A_{(1/2,1]}) = \frac{1}{2}\mu_{G}(A_{(1/2,1]}H) ,
\]
Hence,  $\mu_G(A_{(1/2,1]}H)<2\kappa$.  We now prove that $\mu_G(A_{(0,1/2]}H)< M\kappa/\lambda$.   Suppose that it is not the case. By Lemma~\ref{lem:int} we get
\begin{align*}
  \mu_G(A(B\cap Hb))&\geq \int_{A_{(0,1/2]}H} \mu_H((A\cap aH)(B\cap Hb))\d\mu_G(a)
\end{align*}
Observe that $\mu_H((A\cap aH)(B\cap Hb))>\lambda$ since $\mu_H(B\cap Hb)>\lambda$, we have
\[
\mu_G(AB)\geq\mu_G(A(B\cap Hb))\geq \lambda \mu_G(A_{(0,1/2]}H)>M\kappa,
\]
contradicting the assumption that $\dis_G(AB)<(M-1) \kappa$. Hence $\mu_G(AH)<(M+2)\kappa/\lambda$. This implies that there is $a\in A$ such that $\mu_H(A\cap aH)>\lambda/(M+2)$. Now we apply the same argument switching the role of $A$ and $B$, we get $\mu_{H\backslash G}(\widetilde{\pi} B)<(M+2)^2\mu_G(B)/\lambda$ which completes the proof.
\end{proof}

We remind the readers that when $G$ is a compact Lie group, each element $g\in G$ is contained in some maximal tori. Thus we may view one-dimensional tori as analogues of lines in the Euclidean spaces, and the following definition is natural:
\begin{definition}
Let $G$ be a connected compact Lie group. A set $A\subseteq G$ is {\bf $\lambda$-Kakeya} if for every closed one-dimensional torus $H\leq G$, there is $g\in G$ such that $\mu_H(A\cap gH)\geq\lambda$. We say a set $A$ is not $\lambda$-Kakeya with respect to $H$, if $\mu_H(A\cap gH)<\lambda$ for all $g$. 
\end{definition}
 
We will prove that for sets $A,B$ with sufficiently small measures, if they have small productset, then there is a closed torus $H\subseteq G$ such that for every $x,y\in G$,
\begin{equation}\label{eq: short fiber assump <c}    
\min\{\mu_H(A\cap xH),\mu_H( Hy\cap B)\}<\lambda,
\end{equation}
for some given constant $\lambda$. In other words, neither $A$ nor $B$ can be $\lambda$-Kakeya, because of a same one-dimensional torus $H$.

 



Let $X$ be a open set. Recall that $X$ is an {\bf $M$-approximate group} if $X^2\subseteq \Omega X$ where $\Omega$ is a finite set of cardinality at most $M$. 
The following result by Tao~\cite{T08} is one of the foundational results in additive combinatorics. It shows that sets with small expansions are essentially approximate groups. 

\begin{fact}\label{fact: tao approximate groups}
Let $G$ be a compact group, and let $A,B$ be compact subsets of $G$. Suppose $\mu_G(AB)\leq M(\mu_G(A)\mu_G(B))^{1/2}$. Then there is a $64M^{12}$-approximate group $X$ with $\mu_G(X)\leq (\mu_G(A)\mu_G(B))^{1/2}4M^2 $, such that $A\subseteq \Omega X$, $B\subseteq X\Omega$, and $|\Omega|\leq 33M^{12}$.
\end{fact}

The next theorem is our main result. It shows that if two sets $A,B $ have small expansions in a connected compact Lie group, then one can find a closed one-dimensional tori $H$ such that the intersections of $A $ with any cosets of $H$ are small, and so does $B$. 

\begin{theorem}\label{thm: kakeya new}
Let $G$ be a connected compact Lie group, $M\geq2$ an integer, and $A,B$ be $\sigma$-compact subsets of $G$. 
Let $\lambda\in(0,1)$ be a real number, and suppose $\mu_G(A)=\mu_G(B) =\kappa$ with $\kappa<1/(4M^2 (64M^{12})^{66M^{14}/\lambda})$, and $\mu_G(AB)< M\kappa$. Then there is a closed one-dimensional torus $H\leq G$ such that $A$ and $B$ are not $\lambda$-Kakeya with respect to $H$. 
\end{theorem}
\begin{proof}
Suppose for every closed one-dimensional torus subgroup $H$, we have either 
\[
\mu_H(A \cap xH)\geq \lambda
\]
for some $x\in A$, or
\[
\mu_H(B \cap Hy)\geq \lambda
\]
for some $y\in B$. 
Applying Lemma~\ref{lem: A=AH toric exp prepare} we conclude that there is $a\in A$ and $b\in B$ such that 
\[
\min\{\mu_H(A \cap aH),\mu_H(B \cap Hb)\}\geq\frac{\lambda}{(M+2)^2}. 
\]
Now by Fact~\ref{fact: tao approximate groups}, there is a $64M^{12}$-approximate group $X$ and a finite set $\Omega$ with cardinality at most $33K^{12}$ such that $A\subseteq \Omega X$,  $B\subseteq X\Omega$, and $\mu_G(X)\leq 4M^2\kappa$, and
\[
\min\{\mu_H(\Omega X \cap aH),\mu_H(X\Omega \cap Hb)\}\geq\frac{\lambda}{(M+2)^2}. 
\]
As a consequence, there is $a'\subseteq \Omega^{-1}A$ and $b'\subseteq B \Omega^{-1}$ such that
\[
\min\{\mu_H( X \cap a'H),\mu_H(X \cap Hb')\}\geq\frac{\lambda}{132 M^{14}}. 
\]
Note that $((a')^{-1}X)^{-1}=Xa'$, and $Xa'(a')^{-1}X=X^2$, we conclude that
\[
\mu_H(X^2\cap H)\geq \frac{\lambda}{66M^{14}}. 
\]
Let $t=\lceil 66M^{14}/\lambda \rceil$. By using Kemperman's inequality on $H$, we have $\mu_H(X^{2t}\cap H)=1$, which implies that the set $X^{2t}$ contains all the closed one-dimensional tori in $G$. On the other hand, note that the union of all the closed one-dimensional tori is dense in a maximal torus of $G$. As $X$ is open, we have that $X^{4t}$ contains all the maximal tori, and hence $X^{4t}=G$. Thus
\[
1=\mu_G(X^{4t})\leq (64M^{12})^{4t}\mu_G(X), 
\]
which implies that $\mu_G(X)\geq (64M^{12})^{-66M^{14}/\lambda}$, and this contradicts the fact that $\kappa$ is small.  
\end{proof}

Theorem~\ref{thm: kakeya new} asserts that, for sufficiently small sets $A,B$ with small expansions in a connected compact Lie group $G$, one can find a closed one-dimensional torus that transversal in measure simultaneously in $A$ and $B$. 
The next proposition shows that when $A,B$ have nearly minimal expansions, the smallness assumption can be removed.

\begin{proposition}\label{prop: Torictransversal}
Let $G$ be a connected compact Lie group,  and $A,B$ be $\sigma$-compact subsets of $G$. 
Let $\lambda,c\in(0,1)$ be real numbers. Then there is $\delta>0$ such that the following holds. Suppose $\mu_G(AB)= \mu_G(A)+\mu_G(B)+\delta\min\{\mu_G(A),\mu_G(B)\}<1$ and $\mu_G(A)+\mu_G(B)<1-c$. 
Then there is a closed one-dimensional torus $H\leq G$ such that $A$ and $B$ are not $\lambda$-Kakeya with respect to $H$. 
\end{proposition}

\begin{proof}
Without loss of generality we assume $\mu_G(A)\leq \mu_G(B)$. Using Proposition~\ref{prop: red to small sets}, there are $A'\subseteq A$ and $B'\subseteq B$ with 
\[
\mu_G(A')=\mu_G(B')=\frac{1}{3^{\frac{3^{20}}{\lambda}}}
\]
and
\[
\dis_G(A',B')<\frac{\delta 3^{40}}{c\lambda^2}\min\{\mu_G(A),\mu_G(B)\}.
\]
Choosing $\delta=O(c\lambda^2)$ sufficiently small, we obtain that $\dis_G(A',B')\leq\mu_G(A')$. Applying Theorem~\ref{thm: kakeya new}, we conclude that there is some closed one-dimensional torus $H\leq G$ such that $A'$ and $B'$ are not $\lambda$-Kakeya with respect to $H$. As we know that $A'\subseteq A$ and $B'\subseteq B$, this implies that $A$ and $B$ are not $\lambda$-Kakeya with respect to $H$.
\end{proof}

\section{Almost linear pseudometric from small measure growth}\label{sec: geometry II}

In this section, we investigate the global geometric shape of a pair of sets with small measure growth relative to a connected closed proper subgroup of the ambient Lie group such that the cosets of the subgroup intersect the pair ``transversally in measure'' obtained in the earlier section. 
We will then using this to construct ``almost linear pseudometric'', which place the role of an ``approximate Lie algebra homomorphism''. 

Recall that a {\bf pseudometric} on a set $X$ is a function $d: X \times X \to \RR$ satisfying the following three properties:
\begin{enumerate}
    \item (Reflexive) $d(a,a) =0$ for all $a \in X$,
    \item (Symmetry) $d(a,b)= d(b,a)$ for all $a,b \in X$,
    \item (Triangle inequality) $d(a,c)  \leq d(a,b) +d(b,c) \in X$.
\end{enumerate}
Hence, a pseudometric on $X$ is a metric if for all $a,b \in X$, we have  $d(a,b)=0$ implies $a =b$. If $d$ is a pseudometric on $G$, for an element $g\in G$, we set $\|g\|_d = d(\id, g)$.

 Suppose $d$ is a pseudometric on $G$. We say that $d$ is {\bf left-invariant} if for all $g, g_1, g_2 \in G$, we have $d(gg_1, gg_2) = d(g_1,g_2)$. left-invariant pseudometrics arise naturally from measurable sets in a group; the pseudometric we will construct in the section is of this form.

\begin{proposition}\label{prop: construct pseudo-metric}
Suppose  $A$ is a measurable subset of $G$. For $g_1$ and $g_2$ in $G$, define
$$ d(g_1, g_2) = \mu_G(A) - \mu_G( g_1A \cap g_2A). $$
Then
$d$ is a continuous left-invariant pseudometric on $G$.
\end{proposition}
 \begin{proof}
 We first verify the triangle inequality. Let $g_1$, $g_2$, and $g_3$ be in $G$, we need to show that 
\begin{equation}\label{eq: pseudo metric d_A 1}
    \mu_G( A) - \mu_G(g_1A \cap g_3 A) \leq \mu_G( A) - \mu_G(g_1A \cap g_2 A)  + \mu_G( A) - \mu_G(g_2A \cap g_3 A).
\end{equation}
As $\mu_G(A) = \mu_G(g_2 A)$, we have $\mu_G( A) - \mu_G(g_1A \cap g_2 A) = \mu_G( g_2A \setminus g_1 A)$, and   $\mu_G( A) - \mu_G(g_2A \cap g_3 A) = \mu_G( g_2A \setminus g_3 A)$. Hence, \eqref{eq: pseudo metric d_A 1} is equivalent to 
$$   
\mu_G( g_2A) - \mu_G( g_2A\setminus g_1A) -\mu_G(g_2A\setminus g_3A) \leq \mu(g_1A \cap g_3A). 
$$
Note that the left-hand side is at most $ \mu_G( g_1A \cap g_2A \cap g_3A)$, which is less than the right-hand side. Hence, we get the desired conclusion. The continuity of $d$ follows from Fact~\ref{fact: Haarmeasurenew}(vii), and the remaining parts  are straightforward.
 \end{proof}

Another natural source of  left-invariant pseudometrics is group homomorphims onto metric groups. Suppose $\widetilde{d}$ is a continuous left-invariant metric on a group $H$ and $\pi: G \to H$ is a group homomorphism, then for every $g_1,g_2$ in $G$, one can naturally define a pseudometric $d(g_1, g_2) = \widetilde{d}( \pi(g_1), \pi(g_2))$. It is easy to see that such  $d$ is a continuous left-invariant pseudometric, and $\{g\in G : \Vert g \Vert_d=0\}= \ker(\pi)$ is a normal subgroup of $G$.  The latter part of this statement is no longer true for an arbitrary continuous left-invariant pseudometric, but we still have the following:

\begin{lemma} \label{lem: kerd}
Suppose $d$ is a continuous left-invariant pseudometric on $G$. 
Then the set
 $\{g \in G : \|g\|_d =0 \}$
is the underlying set of a closed subgroup of $G$.
\end{lemma}

\begin{proof}
Suppose $g_1$ and $g_2$ are elements in $G$ such that $\|g_1\|_d = \|g_2\|_d=0$. Then 
\[
d(\id, g_1g_2) \leq d(\id, g_1) + d(g_1, g_1g_2) = d(\id, g_1)+ d(\id, g_2)=0.
\]
Now, suppose $(g_n)$ is a sequence of elements in $G$ converging to $g$ with $\|g_n\|_d=0$. Then $\|g\|_d=0$ by continuity, we get the desired conclusions.
\end{proof}

  In many situations, a left-invariant pseudometric allows us to construct surjective continuous group homomorphism to metric groups.
The following lemma tells us precisely when this happens. We omit the proof as the result is motivationally relevant but will not be used later on.
\begin{lemma} Let $d$ be a continuous left-invariant pseudometric on $G$. The following are equivalent,
\begin{enumerate}[\rm (i)]
    \item The set
 $\{g \in G : \|g\|_d =0 \}$
is the underlying set of a closed normal subgroup of $G$.
    \item There is a continuous surjective group homomorphism  $\pi: G \to H$, and $\widetilde{d}$ is a left-invariant metric on $H$. Then
    \[
    d(g_1, g_2) = \widetilde{d}( \pi g_1, \pi g_2). 
    \]
\end{enumerate}
Moreover, when (ii) happens, $\{g \in G : \|g\|_d =0 \} = \ker \pi$, hence $H$ and $\widetilde{d}$ if exist are uniquely determined up to isomorphism. 
\end{lemma}

  The group $\RR$ and $\TT = \RR/ \ZZ$ are naturally equipped with the metrics $d_{\RR}$ and $d_{\TT}$ induced by the Euclidean norms, and these metrics interact in a very special way with the additive structures. Hence one would expect that if there is a group homomorphism from $G$ to either $\RR$ or $\TT$, then $G$ can be equipped with a pseudometric which interacts nontrivially with addition.

 Let $d$ be a left-invariant pseudometric on $G$. The {\bf radius} $\rho$ of $d$ is defined to be $\sup\{\|g\|_d : g \in G\}$; this is also $\sup\{d(g_1, g_2) : g_1, g_2 \in G\}$ by left invariance.
Let $0<\gamma<\rho/2$ be a constant. For a constant $\lambda<1/2$,  we write $I(\lambda)$ for the interval $(-\lambda,\lambda)$ in $\TT$, and we write $N(\lambda)$ for $\{g \in G : \|g\|_d \leq \lambda\}$. By Fact~\ref{fact: Haarmeasurenew}(vii), $N(\lambda)$ is an open set, and hence measurable.

\begin{definition}
We say that a pseudometric $d$ is {\bf $\gamma$-linear} if it satisfies the following conditions:
\begin{enumerate}
    \item $d$ is continuous and left-invariant;
    \item for all $g_1, g_2, g_3 \in G$ with $d(g_1, g_2)+ d(g_2, g_3) <\rho-\gamma$, we have either
\[
d(g_1,g_3) \in   d(g_1, g_2) + d(g_2,g_3) + I(\gamma), 
\]
or
\[
d(g_1,g_3) \in   |d(g_1, g_2) - d(g_2,g_3)| + I(\gamma).
\]
\end{enumerate}
Given $\alpha\leq \rho$, let $N(\alpha)=\{g\in G:\|g\|_d\leq\alpha\}$. 
We say that $d$ is {\bf $\gamma$-monotone} if for all $g \in N(\rho/2-\gamma)$, we have 
\begin{equation}\label{eq: gamma mono}
\|g^2\|_d \in 2\|g\|_d+I(\gamma). 
\end{equation}
\end{definition}

A pseudometric $d$ is {\bf $\gamma$-path-monotone} if we have \eqref{eq: gamma mono} for every $g$ in a one-parameter subgroup of $G$. The main goal of the section is to construct an almost linear and almost path-monotone pseudometric from a pair of set with very small measure growth.

Throughout this section, $G$ is a connected compact Lie group, $H$ is a closed subgroup of $G$, and $H$ is isomorphic to the one-dimensional torus $\TT$. We let $A$ and $B$ be $\sigma$-compact subsets of $G$ such that 
\[
\kappa/2<\mu_G(A) < 2\kappa \text{ and } \mu_G(B) = \kappa,
\]
and $\dis_G(A,B)\leq \eta\kappa$ for some constant $\eta>0$, that is
\[
\mu_G(AB)\leq \mu_G(A)+\mu_G(B)+\eta\kappa.
\]
In this section, we assume $\eta\leq 10^{-12}$. We did not try to optimise $\eta$, so it is very likely that by a more careful computation, one can make $\eta$ much larger. But we believe this method does not allow $\eta$ to be very close to $1$. By Fact~\ref{fact: new inverse theorem torus}, let $=c_\TT$ be the constant obtained from the theorem. We will assume throughout the section that
\[
\max\{\mu_H(A\cap aH), \mu_H(B\cap Hb)\}<c_\TT
\]
 for all $a,b\in G$. This can actually be arranged by applying Proposition~\ref{prop: Torictransversal}.


As in the earlier sections, we set $\mu_{G/H}$ and $\mu_{H \backslash G}$ to be the Radon measures on $G/H$ and $H \backslash G$ such that we have the quotient integral formulas (Fact~\ref{fact: QuotientIF} and Lemma~\ref{lem: mesurability}(vi)). We also remind the reader that we normalize the measure whenever a group under consideration is compact, and $\pi: G \to G/H$, $ \widetilde{\pi}: G \to H\backslash G$ are quotient maps. Hence, we have
$$ \mu_G(AH)=\mu_{G/H}(\pi A)\quad \text{and}\quad \mu_{G}(HB) =\mu_{H \backslash G}( \widetilde{\pi} B). $$
The following lemma records some basic properties of sets $A$ and $B$.
\begin{lemma} \label{lem: Kempermansubgroup}
For all $a \in A$ and $b \in B$, we have
\begin{enumerate}
    \item $\mu_H\big(( A \cap aH)(B \cap Hb)\big) \geq \mu_H(A \cap aH) + \mu_H(B\cap Hb).  $
    \item $\mu_G\big(A ( B \cap Hb)\big) \geq  \mu_G(A) + \mu_{G/H}(\pi A)   \mu_{H} ( B \cap Hb ).$
    \item $\mu_G\big((A \cap aH) B\big) \geq  \mu_{H} ( A \cap aH ) \mu_{H \backslash G}(\widetilde{\pi}B) + \mu_G(B).$  
    \end{enumerate}
The equality in (2) holds if and only if the equality in (1) holds for almost all $a \in AH$. A similar conclusion holds for (3).
\end{lemma}

\begin{proof}
The first inequality comes from a direct application of the Kemperman inequality. For the second inequality, by right translating $B$ and using the unimodularity of $G$, we can arrange that $Hb =H$. The desired conclusion follows from applying (1) and Lemma~\ref{lem:int}. 
\end{proof}




 Towards showing that sets $A$ and $B$ behave rigidly, our next theorem shows that most of the nonempty fibers in $A$ and $B$ have the similar lengths, and the majority of them behaves rigidly fiberwise. 

\begin{theorem}[Near rigidity fiberwise]\label{thm: fibers of same length 1newnew} There is a continuous surjective group homomorphism $\chi:H\to\TT$, two compact intervals $I, J\subseteq \TT$ with 
\[ \mu_\TT(I)=\frac{\mu_G(A)}{\mu_{G/H}(\pi A)}\  \text{ and  } \ \mu_\TT(J)=\frac{\mu_{H}(B)}{\mu_{H \backslash G}(\widetilde{\pi} B)}.
\]
$\sigma$-compact $A'\subseteq A$ and $B'\subseteq B$ with 
\[
\mu_{G/H}(\pi A')>99\mu_{G/H}(\pi A)/100\  \text{ and } \ \mu_{H \backslash G}(\widetilde{\pi} B')>99\mu_{H \backslash G}(\widetilde{\pi} B)/100,
\] 
and   a constant $\nu\leq 10^{-10}$ such that the following statements hold:
\begin{enumerate}[\rm (i)]
\item  we have 
\[
\frac{1}{1+\eta}\mu_{H \backslash G}(\widetilde{\pi}B) \leq  \mu_{G/H}(\pi A) \leq  (1 +\eta) \mu_{H \backslash G}(\widetilde{\pi}B) .
\]
   \item For every $a$ in $A'H$, 
    \[
    (1-\nu) \frac{\mu_G(A)}{\mu_{G/H}(\pi A)}\leq \mu_H(A\cap aH)\leq (1+\nu) \frac{\mu_G(A)}{\mu_{G/H}(\pi A)},
    \]
    and there is $\zeta_{a} \in \TT$ with 
    \[ 
       \mu_H\big((A\cap aH)\tri a\chi^{-1}(\zeta_{a}+I)\big)<\nu\min\Big\{\frac{\mu_G(A)}{\mu_{G/H}(\pi A)},\frac{\mu_G(B)}{\mu_{H\backslash G}(\widetilde{\pi} B )}\Big\}.
      \] 
       \item For every $b$ in $HB'$, 
    \[
    (1-\nu) \frac{\mu_G(B)}{\mu_{H\backslash G}(\widetilde{\pi} B )}\leq \mu_H(B\cap Hb)\leq (1+\nu) \frac{\mu_G(B)}{\mu_{H\backslash G}(\widetilde{\pi}B )},
    \]
    and there is $\widetilde{\zeta}_{b}\in \TT$ with
     \[  
       \mu_H\big((B\cap Hb)\tri \chi^{-1}(\widetilde{\zeta}_{b}(B)+J)b\big)<\nu\min\Big\{\frac{\mu_G(A)}{\mu_{G/H}(\pi A)},\frac{\mu_G(B)}{\mu_{H\backslash G}(\widetilde{\pi} )}\Big\}. 
    \]
\end{enumerate}
\end{theorem}
\begin{proof}
Without loss of generality, we assume that $\mu_{G/H}(\pi A)\geq\mu_{H \backslash G}(\widetilde{\pi}B)$. Let $\beta$ be a constant such that $\beta<(\kappa/800\mu_{H\backslash G}(\widetilde{\pi} B))$.  Obtain $b^* \in G$ such that 
$$\mu_H(B\cap Hb^*)\geq \sup_b\mu_H(B\cap Hb)-\beta  \text{ for all } b\in G,$$ 
and the fiber $B\cap Hb^*$ has at least the average length, that is
\begin{equation}\label{eq:choiceb_0}
    \mu_H(B\cap Hb^*)\geq\mathbb E_{b\in BH}\mu_H(B\cap Hb)=\frac{\mu_G(B)}{\mu_{H\backslash G}(\widetilde{\pi} B)}.
\end{equation}
Set $\delta=100\eta\kappa/\mu_{G/H}(\pi A)$. As $\eta\leq 10^{-12}$, we get $\nu\leq 10^{-10}$ such that 
\begin{equation}\label{eq: nu value}
    \delta<\nu\mu_G(A)/\mu_{G/H}(\pi A).
\end{equation} 
Set
\[  
N = \{ a\in AH : \dis_H(A\cap aH,B\cap Hb^*)>\delta  \}.
\]
Note that $N$ is measurable by Lemma~\ref{lem: mesurability}.
By Lemma~\ref{lem:int} we have
\begin{align*}
&\ \mu_G\big(A (B \cap Hb^*)\big)\\
=&\, \int_{N} \mu_H\big(( A \cap aH )(B \cap Hb^*)\big) \d\mu_G(a)+ \int_{G \setminus N} \mu_H\big(( A \cap aH )(B \cap Hb^*)\big) \d\mu_G(a).
\end{align*}
Since $A\cap aH$ is nonempty for every $a\in AH$, using Kemperman's inequality on $H$ we have that $\mu_G(A(B\cap Hb^*))$ is at least
\[
\int_{N} \big(\mu_H( A \cap aH )+\mu_H(B \cap Hb^*)+\delta\big) \d\mu_G(a)
+ \int_{G\setminus N} \big(\mu_H( A \cap aH )+\mu_H(B \cap Hb^*)\big) \d\mu_G(a).
\]
Suppose we have $\mu_G(N)>\mu_{G/H}(\pi A)/100$. Therefore, by the choice of $b^*$ we get
\begin{align}
\mu_G\big(A (B \cap Hb^*)\big)&>\mu_G(A)+\frac{\delta\mu_{G/H}(\pi A)}{100}+\mu_H(B\cap Hb^*)\mu_{G/H}(\pi A)\label{eq:b^*}\\
& \geq\mu_G(A)+\mu_G(B)\frac{\mu_{G/H}(\pi A)}{\mu_{H\backslash G}(\widetilde{\pi} B)}+\eta\kappa.\nonumber
\end{align}
Since $\mu_{G/H}(\pi A)\geq \mu_{H\backslash G}(\widetilde{\pi} B)$, and $A(B\cap Hb^*)\subseteq AB$,  we have
\begin{equation*}
  \mu_G(AB)>\mu_G(A)+\mu_G(B)+\eta\kappa. 
\end{equation*}
This contradicts the assumption that $\dis_G(A,B)$ is at most  $\eta\kappa$. Using the argument in equation (\ref{eq:b^*}) with trivial lower bound on $\mu_G(N)$ we also get 
\begin{equation}\label{eq:AHandHB}
  \mu_{H\backslash G}(\widetilde{\pi} B)\leq \mu_{G/H}(\pi A)\leq \big(1+\eta\big)\mu_{H\backslash G}(\widetilde{\pi} B),  
\end{equation}
which proves (i). 

From now on, we assume that $\mu_G(N)\leq \mu_{G/H}(\pi A)/100$. Since $\dis_G(A,B)$ is at most $\eta\kappa$, by  (\ref{eq:b^*}) again (using trivial lower bound on $\mu_G(N)$), we have
\[
\mu_H(B\cap Hb^*)\mu_{H\backslash G}(\widetilde{\pi} B)\leq \mu_G(B)+\eta\kappa,
\]
and this in particular implies that for every $b\in G$,  we have
\[
\mu_H(B\cap Hb)\leq\frac{\mu_G(B)}{\mu_{H\backslash G}(\widetilde{\pi} B)}+\frac{\eta\mu_G(B)}{\mu_{H\backslash G}(\widetilde{\pi} B)}+\eta<\big(1+\nu \big)\frac{\mu_G(B)}{\mu_{H\backslash G}(\widetilde{\pi} B)}.
\]
Thus there is $Y\subseteq B$ with $\mu_G(HY)<  \mu_{H\backslash G}(\widetilde{\pi} B)/100$  such that for every $b\in HY$, 
\[
\mu_H(B\cap Hb)\geq \frac{\mu_G(B)}{\mu_{H\backslash G}(\widetilde{\pi} B)}-100\frac{\eta\mu_G(B)}{\mu_{H\backslash G}(\widetilde{\pi} B)}-100\eta>\big(1-\nu\big)\frac{\mu_G(B)}{\mu_{H\backslash G}(\widetilde{\pi} B)}.
\]

Next, we apply the similar argument to $A$. Let $\alpha<(\mu_G(A)-2\eta\kappa)/200\mu_{G/H}(\pi A)$, and choose $a^*$ such that $\mu_H(A\cap a^*H)>\mu_H(A\cap aH)-\alpha$ for all $a\in AH$, and
\[
\mu_H(A\cap a^*H)\geq\mathbb E_{a\in AH}\mu_H(A\cap aH)=\frac{\mu_G(A)}{\mu_{G/H}(\pi A)}.
\] 
Let $N'\subseteq HB$ such that for every $b$ in $N'$, $\dis_H(A\cap a^*H,B\cap Hb) \geq \delta$. Hence we have
\begin{align}
&\ \mu_G\big((A\cap a^*H) B\big)\nonumber\\
=&\, \int_{N'} \mu_H\big(( A \cap a^*H )(B \cap Hb)\big) \d\mu_G(b)+ \int_{G \setminus N'} \mu_H\big(( A \cap a^*H )(B \cap Hb)\big) \d\mu_G(b)\nonumber\\
\geq&\, \mu_G(B)+\mu_H(A\cap a^*H)\mu_{H\backslash G}(\widetilde{\pi} B)+\delta\mu_G(N')\label{eq: lambda'}\\
\geq&\, \mu_G(A)+\mu_G(B)-\frac{\mu_G(A)\eta\kappa}{\mu_G(A)+\eta\kappa}+\delta\mu_G(N').\nonumber
\end{align}
By the fact that $\mu_G(AB)\geq\mu_G((A\cap a^*H)B)$ and $\dis_G(A,B)\leq \eta\kappa$, we have that
\[
\mu_G(N')\leq\frac{1}{200}\mu_{G/H}(\pi A)\leq\frac{1}{150}\mu_{H\backslash G}(\widetilde{\pi} B).
\]
Now, by equation (\ref{eq: lambda'}), and the choice of $a^*$, we have that for all $a\in AH$,
\[
\mu_H(A\cap aH)\leq\frac{\mu_G(A)}{\mu_{G/H}(\pi A)}+\frac{\eta\mu_G(A)}{\mu_{H\backslash G}(\widetilde{\pi} B)}+\alpha<\big(1+\nu\big)\frac{\mu_G(A)}{\mu_{G/H}(\pi A)}.
\]
Again by equation (\ref{eq: lambda'}), there is $X\subseteq A$ with $\mu_G(XH)\leq\mu_{G/H}(\pi A)/200$, such that for every $a\in X$,
\[
\mu_H(A\cap aH)\geq\frac{\mu_G(A)}{\mu_{G/H}(\pi A)}-200\frac{\eta\mu_G(A)}{\mu_{H\backslash G}(\widetilde{\pi} B)}-200\alpha\geq\big(1-\nu\big)\frac{\mu_G(A)}{\mu_{G/H}(\pi A)}.
\]

Let $A'=A\cap(AH\setminus(XH\cup N'))$, and let $B'=B\cap (HB\setminus(HY\cup N))$. Then 
\[
\mu_G(A')\geq \frac{99}{100}\mu_{G/H}(\pi A),\quad \mu_G(B')\geq \frac{99}{100}\mu_{H\backslash G}(\widetilde{\pi} B),
\]
Let $a$ be in $A'H$ and $b$ be in $B'H$. By our construction, the first parts of (ii) and (iii) are satisfied. Moreover, both $\dis_H(A\cap aH,B\cap Hb^*)$ and $\dis_H(A\cap a^*H,B\cap Hb)$ are at most $\delta$. By the way we construct $A'$ and $B'$, we have that $a^*\in A'$ and $b^*\in B'$. Recall that $\mu_H(A\cap aH),\mu_H(B\cap Hb)<\lambda$ for every $a,b\in G$. Therefore, by the inverse theorem on $\TT$ (Fact~\ref{fact: new inverse theorem torus}), and  Lemma~\ref{lem: stability of character in T}, there is a group homomorphism $\chi:H\to\TT$, and two compact intervals $I_A$, $I_B$ in $\TT$, with
\[
\mu_\TT(I_A)=\frac{\mu_G(A)}{\mu_{G/H}(\pi A)},\quad \mu_\TT(I_B)=\frac{\mu_G(B)}{\mu_{H\backslash G}(\widetilde{\pi} B)},
\]
such that for every $a\in A'$ and $b\in B'$, there are elements  $\zeta_{a},\widetilde{\zeta}_{b}(B)$ in $\TT$, and
\[
\mu_H(A\cap aH\tri a\chi^{-1}(\zeta_{a}+I_A))<\delta,\quad 
\mu_H(B\cap Hb\tri \chi^{-1}(\widetilde{\zeta}_{b}+I_B)b)<\delta,
\]
and the theorem follows from \eqref{eq: nu value}.
\end{proof}

  The next corollary gives us an important fact of the structure of the projection of $A$ on $G/H$.

\begin{corollary}[Global structure of $AH$]\label{cor:AH=1} Suppose $\mu_G(A)=\kappa$. Then for all $g \in \Stab^{\kappa/2}_G(A)$, we have
 $$\mu_G(AH\tri gAH)\leq 10\eta \mu_{G/H}(\pi A).$$
\end{corollary}

\begin{proof}
By Lemma~\ref{lem:iep}, $\dis_G(A\cup gA, B)\leq 2\eta\kappa$, and by Theorem~\ref{thm: fibers of same length 1newnew}, we have
\[
\frac{1}{1+2\eta}<\frac{\mu_G(AH\cup gAH)}{\mu_{H\backslash G}(\widetilde{\pi} B)}<1+2\eta.
\]
On the other hand, since $\dis_G(A,B)$ and $\dis_G(gA,B)$ are at most $\eta\kappa$, we have 
\[
\frac{1}{1+\eta}<\frac{\mu_{G}(AH)}{\mu_{H\backslash G}(\widetilde{\pi} B)}, \frac{\mu_G(gAH)}{\mu_{H\backslash G}(\widetilde{\pi} B)}<1+\eta
\]
Since $\eta< 10^{-12}$, by inclusion-exclusion principle, we get the desired conclusion. 
\end{proof}

Theorem~\ref{thm: fibers of same length 1newnew} and Corollary~\ref{cor:AH=1} essentially allows us to define a ``directed linear pseudometric'' on $G$ by ``looking at the generic fiber'' as discussed in the following remark:
\begin{remark}\label{rem}
Fix $a\in AH$ and let the notation be as in Theorem~\ref{thm: fibers of same length 1newnew}. For $g_1, g_2$ in $G$ such that $g_1^{-1}a, g_2^{-1}a \in A'H$, set
\[
\delta_{a,A}(g_1,g_2)=\zeta_{g_1^{-1}a}-\zeta_{g_2^{-1}a}.
\]
 We have the following linearity property of $\delta_{a,A}$ when the relevant terms are defined, which is essentially the linearity property of the metric from $\RR$.  
\begin{enumerate}
    \item $\delta_{a,A}(g_1,g_1)=0$.
    \item $\delta_{a,A}(g_1,g_2)=-\delta_{a,A}(g_2,g_1)$.
    \item $\delta_{a,A}(g_1,g_3)=\delta_{a,A}(g_1,g_2)+\delta_{a,A}(g_2, g_3).$
\end{enumerate}
Properties (1) and (2) are immediate, and property (3) follows from the easy calculation below:
\begin{align*}
  \delta_{a,A}(g_1,g_2)&=\zeta_{g_1^{-1}a}-\zeta_{g_2^{-1}a} \\
    &=\zeta_{g_1^{-1}a}-\zeta_{g_3^{-1}a}+\zeta_{g_3^{-1}a}-\zeta_{g_2^{-1}a} \\
    &= \delta_{a,A}(g_1,g_3)\pm \delta_{a,A}(g_3,g_2).  
\end{align*}
Properties (3) also implies that
\[
| \delta_{a,A}(g_1,g_3)|=\big|\pm |\delta_{a,A}(g_1,g_2)| \pm |\delta_{a,A}(g_2, g_3)|\big|.
\]
which tells us that $|\delta_{a,A}|$ is a linear pseudometric.
The problem with the above definitions is that they are not defined everywhere. 
\end{remark}


Recall that we use $I(\tau)$ to denote the open interval $(-\tau,\tau)$ in $\TT$. As in Remark~\ref{rem}, one can consider
\[
\delta_{a,A}(g_1,g_2)=\zeta_{g_1^{-1}a}-\zeta_{g_2^{-1}a}
\]
for a fix $a\in A$. Recall that
\[
d_A(g_1,g_2)=\mu_G(g_1A\setminus g_2A)
\]
is a pseudometric (see Proposition~\ref{prop: construct pseudo-metric}). The next lemma gives the connection between $\delta_{a,A}(g_1,g_2)$ and $d_A(g_1,g_2)$.  

\begin{lemma}[From local to global] \label{lem: 2 elements near rigidity}
Suppose $\mu_G(A)=\kappa$, $g_1, g_2 \in \Stab^{\kappa/4}_G(A)$, and $\chi: H \to \TT$, $I \subseteq \TT$, $\nu$ are as in Theorem~\ref{thm: fibers of same length 1newnew}. Then there is a $\sigma$-compact $A'' \subseteq A$ with \[\mu_{G/H}(\pi A'') =96\mu_{G/H} (\pi A)/100\]
such that for all $a \in A''H$, the following holds
\begin{enumerate}[\rm (i)]
    \item there are $\zeta_{g^{-1}_1a},  \zeta_{g^{-1}_2a} \in \TT$ such that for $i \in \{1,2\}$;
    \[ \mu_H\big((A\cap aH)\tri a\chi^{-1}(\zeta_{g_{i}^{-1}a}+I)\big)<\frac{\nu\kappa}{\mu_{G/H}(\pi A)}.\]
    \item with any $\zeta_{g^{-1}_1a}, \zeta_{g^{-1}_2a} $ satisfying (i) and $\delta_{a, A}(g_1, g_2)= \zeta_{g^{-1}_2a}- \zeta_{g^{-1}_1a}$, we have
$$ d_A(g_1, g_2) \in  \mu_{G/H}(\pi A)|\delta_{a, A}(g_1, g_2)| + I\big(20\nu\kappa\big).    $$
\end{enumerate}
\end{lemma}

\begin{proof}
Obtain $A', I, J$ as in Theorem~\ref{thm: fibers of same length 1newnew}. Let $A'_1 \subseteq G$ be the $\sigma$-compact set
\[
\{ a \in A: g_1^{-1}a, g_2^{-1}a \in A'H      \}.
\]
It is easy to see that $\mu_{G/H}(\pi A'_1) > 98/100\mu_{G/H} (\pi A)$. Fix $a \in A'_1$, by Theorem~\ref{thm: fibers of same length 1newnew} again
 we then have 
 \[
  \mu_H\big((A\cap g_1^{-1}aH)\tri g_1^{-1}a\chi^{-1}(\zeta_{g_1^{-1}a}+I)\big)<\nu\frac{\kappa}{\mu_{G/H}(\pi A)},
 \]
 and
  \[
  \mu_H\big((A\cap g_2^{-1}aH)\tri g_2^{-1}a\chi^{-1}(\zeta_{g_2^{-1}a}+I)\big)<\nu\frac{\kappa}{\mu_{G/H}(\pi A)}.
 \]
Multiplying by $g_1$ and $g_2$ respectively, we get (i) when $A''\subseteq A'_1$. 

As $g_1, g_2 \in \Stab^{\kappa/4}_G(A)$, we have $\mu_{G}(g_1A \cap g_2A)>0$. Note that $\dis_G(g_1A\cap g_2A, B)\leq 2\eta\kappa$ by Lemma~\ref{lem:iep}. By Theorem~\ref{thm: fibers of same length 1newnew}(i), we have
\begin{equation}\label{eq: compute intersection}
\mu_G((g_1A\cap g_2A)H)\geq \frac{1}{1+2\eta}\mu_G(HB) \geq \frac{1}{(1+2\eta)(1+\eta)}\mu_{G}(AH). 
\end{equation}
Also Theorem~\ref{thm: fibers of same length 1newnew}(ii) implies that there is $A'_2\subseteq g_1A\cap g_2A$ with $\mu_{G/H}(\pi A'_2)>99\mu_{G/H}(\pi (g_1A\cap g_2A))/100$, such that for all $a\in A'_2H$, we have  
\begin{equation}\label{eq: pseu first}
\big(1-\nu\big)  \frac{\mu_G(g_1A)}{\mu_{G/H}(\pi(g_1A))} \leq  \mu_H(g_1A\cap aH) \leq (1+\nu)\frac{\mu_G(g_1A)}{\mu_{G/H}(\pi(g_1A))}\ 
\end{equation}
and
\[ (1-2\nu)\frac{\mu_G(g_1A\cap g_2A)}{\mu_{G/H}(\pi(g_1A\cap g_2A))} \leq \mu_H(g_1A\cap g_2A\cap aH) \leq (1+2\nu)\frac{\mu_G(g_1A\cap g_2A)}{\mu_{G/H}(\pi(g_1A\cap g_2A))}.
\]

Note that 
$\pi(g_1A\cap g_2A) \subseteq  \pi(g_1A) \cap \pi(g_2A).$
However, \eqref{eq: compute intersection} together with Corollary~\ref{cor:AH=1} give us
\begin{align*}
&\mu_{G/H}( \pi(g_1A) \tri \pi A  ) \leq 10\eta \mu_{G/H}(\pi A),\\
&\mu_{G/H}( \pi(g_1A\cap g_2A) \tri \pi A  ) \leq 14\eta \mu_{G/H}(\pi A). 
\end{align*}
Hence, Let $A''=A'_1H\cap A'_2H\cap A$. Note that $\mu_{G/H}(\pi A'')\geq 96\mu_{G/H}(\pi A)/100$, and for all $a \in A''H$, by \eqref{eq: compute intersection}
\begin{equation}\label{eq: pseu second}
(1-5\nu) \frac{\mu_G(g_1A\cap g_2A)}{\mu_{G/H}(\pi(A))} \leq \mu_H(g_1A\cap g_2A\cap aH) \leq (1+5\nu) \frac{\mu_G(g_1A\cap g_2A)}{\mu_{G/H}(\pi(A))}.
\end{equation}
Finally, recall that $d_A(g_1,g_2) =\mu_G(g_1A)-\mu_G(g_1A \cap g_2A)$. Note that
\begin{align*}
\mu_{G/H}(\pi A)|\delta_{a,A}(g_1,g_2)|&=\mu_{G/H}(\pi A)|\zeta_{g_1^{-1}a}-\zeta_{g_2^{-1}a}|\\
&\in \mu_H(g_1A\cap aH)-\mu_H(g_1A\cap g_2A\cap aH)+I(2\nu\kappa).
\end{align*}
Hence, by \eqref{eq: pseu first} and \eqref{eq: pseu second},
\begin{align*}
d_A(g_1,g_2)&\geq \mu_{G/H}(\pi A)\left(\frac{\mu_H(g_1A\cap aH)}{1+\nu}-\frac{\mu_H(g_1A\cap g_2A\cap aH)}{1-5\nu}\right)\\
&\geq \mu_{G/H}(\pi A)\big(\mu_H(g_1A\cap aH)-\mu_H(g_1A\cap g_2A\cap aH)\big)-18\nu\kappa\\
&\geq \mu_{G/H}(\pi A)|\delta_{a,A}(g_1,g_2)|-20\nu\kappa. 
\end{align*}
The upper bound on $d_A(g_1,g_2)$ can be computed using a similar method, and this finishes the proof.
\end{proof}


We now deduce key properties of the pseudometric $d_A$. Besides the almost linearity, we also need the path monotonicity of the pseudometric to control the ``direction''. 

\begin{proposition}[Almost linearity and path monotonicity of the pseudometric]\label{prop: almost linear metric from local}
Assume that $\mu_G(A)=\kappa$, and let $\nu$ be as in Theorem~\ref{thm: fibers of same length 1newnew}. Then we have the following:
\begin{enumerate}[\rm (i)]
    \item For all $g_1,g_2,g_3$ in $\Stab^{\kappa/2}_G(  A )$, we have 
    \[
    d_A(g_1,g_2)\in |\pm d_A(g_1,g_3)\pm d_A(g_2,g_3)|+I\big(60\nu\kappa\big),
    \]
       \item Let $\mathfrak{g}$ be the Lie algebra of $G$, and let $\exp: \mathfrak{g}\to G$ be the exponential map. For every $X\in \mathfrak{g}$, either 
       $$d_A(\exp(Xt), \id)<\kappa/4 \text{ for all } t\in\RR$$ 
       or  there is $t_0>0$ with $d_A(\exp(Xt_0), \id)\geq \kappa/4$  such that for every $t\in[0,t_0]$, 
    \begin{align*}
   &\, d_A(\exp(X(t+t_0)),\id)\\
    \in&\, d_A(\exp(X(t+t_0)),\exp(Xt_0))+d_A(\exp(Xt_0),\id)+I\big(180\nu \kappa\big).
    \end{align*}
\end{enumerate}
\end{proposition}
\begin{proof}
We first prove (i). Let $\chi$ and $I$ be as in Theorem~\ref{thm: fibers of same length 1newnew}. Applying Lemma~\ref{lem: 2 elements near rigidity}, we get $a \in  AH$ and  $\zeta_{g^{-1}_1a},  \zeta_{g^{-1}_2a}, \zeta_{g^{-1}_1a} \in \TT$ such that for $i \in \{1, 2, 3\}$, we have
\[ \mu_H\big((A\cap aH)\tri a\chi^{-1}(\zeta_{g_{i}^{-1}a}+I)\big)<\frac{\nu\kappa}{\mu_{G/H}(\pi A)},\]
and for $i, j \in \{1, 2, 3\}$, we have
$$ d_A(g_i, g_j) \in  \mu_{G/H}(\pi A)|\delta_{a, A}(g_i, g_j)| + I\big(20\nu\kappa\big).    $$
with $\delta_{a, A}(g_i, g_j)= \zeta_{g^{-1}_ja}- \zeta_{g^{-1}_ia}$.
As $\delta_{a, A}(g_1, g_2)= \delta_{a, A}(g_1, g_3)+ \delta_{a, A}(g_3, g_2)$, we get the desired conclusion.

Next, we prove (ii). 
Let $X\in \mathfrak{g}$, and suppose there is $t>0$ such that 
$$d_A(\exp(Xt),\id)\geq \kappa.$$
Using the continuity of $g \mapsto \mu_G(A\setminus gA)$ (Fact~\ref{fact: Haarmeasurenew}(vii)), we obtain $t_0>0$ such that $t_0$ the smallest positive real number with $d_A(\id, \exp(Xt_0))\geq \kappa/10$. Fix $t\in [0,t_0]$, and set 
$$g_0= \exp(Xt_0) \text{ and } g = \exp(Xt). $$
Note that $gg_0=g_0g$ as $g_0$ and $g$ are on the same one parameter subgroup of $G$. 
One can easily check that $g_0$, $g$, $g_0g$ are in $\Stab^{\kappa/2}_G(A)$. Again, let $\chi$ and $I$ be as in Theorem~\ref{thm: fibers of same length 1newnew} and apply Lemma~\ref{lem: 2 elements near rigidity} to get $a \in  AH$ and  $\zeta_{g^{-1}_ia} \in \TT$ for $g_i \in \{ \id, g, g_0, gg_0 \}$ such that   
\begin{equation} \label{8.18.1}
 \mu_H\big((A\cap aH)\tri (a\chi^{-1}(\zeta_{g_{i}^{-1}a}+I)\big)<\frac{\nu\kappa}{\mu_{G/H}(\pi A)},
\end{equation}
and  for  $g_i, g_j \in \{ \id, g, g_0, gg_0 \}$, we have
\begin{equation} \label{8.18.2}
 d_A(g_i, g_j) \in  \mu_{G/H}(\pi A)|\delta_{a, A}(g_i, g_j)| + I\big(20\nu\kappa\big)   
\end{equation} 

As $gg_0= g_0g$,  we have
\begin{equation} \label{18.8.3}
   \delta_{a, A}(  \id, g)+\delta_{a,A}(g,gg_0) = \delta_{a, A}(  \id, g_0g) = \delta_{a, A}(  \id, g_0) + \delta_{a, A}(  g_0, g_0g)   
\end{equation}
Using \eqref{8.18.1}, \eqref{8.18.2}, and the fact that $d_A(\id, g_0) = d_A(g, gg_0) $, we get
\begin{align*}
\delta_{a,A}(g, gg_0)\in \pm \delta_{a,A}(\id, g_0)+I\big(60\nu\kappa\big). 
\end{align*}
By a similar argument, $
\delta_{a,A}(g_0, g_0g)\in \pm \delta_{a,A}(\id, g)+I\big(60\nu\kappa\big)$.
Combining with \eqref{18.8.3}, we get that
 $\delta_{a,A}(\id,gg_i)$ is in both
\[
\delta_{a,A}(\id, g_0)\pm \delta_{a,A}(\id,g)+I\big(60\nu\kappa\big)
\]
and
\[
\delta_{a,A}(\id, g)\pm \delta_{a,A}(\id,g_0)+I\big(60\nu\kappa\big).
\]
Using the fact that $\nu<10^{-6}$, and considering all the four possibilities, we deduce 
\[
\delta_{a,A}(\id, gg_0) =\delta_{a,A}(\id, g_0)+ \delta_{a,A}(\id,g)+I\big(120\nu\kappa\big).
\]
Applying \eqref{8.18.1} and \eqref{8.18.2} again, we get the desired conclusion.
\end{proof}

\section{Properties of the almost linear pseudometrics}\label{section:pseudometric}

Throughout this section,  $d$ is a pseudometric on $G$ with radius $\rho>0$, and $\gamma$ is a constant with $0<\gamma<10^{-8}\rho$.
 The next lemma says that under the $\gamma$-linearity condition, the group $G$ essentially has only one ``direction'': if there are three elements have the same distance to $\id$, then at least two of them are very close to each other. 
\begin{lemma} \label{lem:dichotomy}
Suppose $d$ is a $\gamma$-linear pseudometric on $G$. If $g, g_1, g_2 \in G$ such that
\[
\| g \|_d =\| g_1 \|_d = \| g_2 \|_d  \in    I(\rho/4-\gamma) \setminus I(2\gamma),
\]
and $d(g_1, g_2) \in 2\| g \|_d + I(\gamma)$. Then either $d(g, g_1) \in I(\gamma)$ or $d(g, g_2) \in I(\gamma)$.
\end{lemma}
\begin{proof}
Suppose both $d(g, g_1)$ and $d(g, g_2)$ are not in $I(\gamma)$. By $\gamma$-linearity of $d$, we have 
\[
d(g,g_1)\in |d(\id,g)\pm d(\id,g_1)|+I(\gamma),
\]
and so $d(g,g_1)\in 2\|g\|_d+I(\gamma)$. Similarly, we have  $d(g,g_2) 2\|g\|_d+I(\gamma)$. 

Suppose first that $ d(g_1, g_2)\in d(g, g_1)+d(g, g_2) + I(\gamma)$, then
 \[
 d(g_1, g_2)\in 4\|g\|_d+I(3\gamma).
 \]
 On the other hand, by $\gamma$-linearity we have $d(g_1,g_2)\leq 2\|g\|_d+\gamma.$
 Hence, we have $\|g\|_d \in I(2\gamma)$, a contradiction. 
 
  The other two possibilities are
$d(g_1, g_2)+ d(g, g_2) \in d(g, g_1) +I (\gamma)$ or $d(g_1, g_2)+d(g, g_1) \in d(g, g_2)+ I (\gamma)$, but similar calculations also lead to contradictions.
\end{proof}

In the previous section, Proposition~\ref{prop: almost linear metric from local} shows that the pseudometric we obtained also satisfies a local monotonicity property on the one parameter subgroups. The next proposition shows that, it is enough to derive a global monotonicity property for any pseudometric with a local monotonicity property.

\begin{proposition}[Path monotonicity implies global monotonicity]\label{prop: localmonotoneimplyglobalmonotone}
Let $\mathfrak{g}$  be the Lie algebra of $G$, $\exp: \mathfrak{g} \to G$ the exponential map, and $d$ a $\gamma$-linear pseudometric on $G$. Suppose for each $X$ in $\mathfrak{g}$, we have one of the following two possibilities:
\begin{enumerate}[\rm (i)]
    \item $\|\exp(tX)\|_d < \gamma$ for all $t \in \RR$;
    \item there is $t_0\in  \RR^{>0}$ with 
$\|\exp(t_0X)\|_d \in I(\rho/2-\gamma)\setminus I(\rho/4)$,  
\begin{equation}\label{eq: condition (1)}
\|\exp(2t_0X)\|_d = 2\|\exp(t_0X)\|_d+ I(\gamma), \end{equation} 
and 
\begin{equation}\label{eq: condition (2)}
\|\exp(tX)\|_d + \|\exp((t_0-t)X)\|_d \in \|\exp(t_0X)\|_d+I(\gamma)
\end{equation}
for all $t\in [0, t_0]$.
\end{enumerate}
Then $d$ is $(9\gamma)$-monotone.
\end{proposition}
\begin{proof}
 Fix an element $g$ of $G$ with $\|g\|_d \in I(\rho/2 -16\gamma)$. Our job is to show that $\|g^2\|_d \in 2\|g\|_d + I(9\gamma)$. Since $G$ is compact and connected, the exponential map $\exp$ is surjective. We get $X \in \mathfrak{g}$ such that $g\in \{ \exp(tX) : t \in \RR\}$. If we are in scenario (i), then  $\|g\|_d< \gamma$, hence $ \|g^2\|_d \in 2\|g\|_d +I(3\gamma)$. Therefore, it remains to deal with the case where we have an $t_0$ as in (ii). 
 
  Set $g_0 = \exp(t_0X)$. We consider first the special case where $\|g\|_d < \|g_0\|_d-2\gamma$.  As $d$ is continuous, there is $t_1 \in [0, t_0]$ such that with $g_1= \exp(t_1X)$, we have $\|g_1\|_d =\|g\|_d$. Let $t_2 = -t_1$, and $g_2 = \exp(t_2X) = g_1^{-1}$. Since $d$ is invariant, 
  \[
  \|g_2\|_d = d( g^{-1}_1, \id) = d( \id, g_1) = \|g_1\|_d.
  \]
  Hence,  $\|g_1\|_d = \|g_2\|_d = \| g \|_d$. If $\| g\|_d <2\gamma$, then $\|g^2\|_d \in 2\|g\|_d+ I(5\gamma)$ and we are done. Thus we suppose $\| g\|_d \geq 2\gamma$. Then, by Lemma~\ref{lem:dichotomy}, either $d(g, g_1)<\gamma$, or $d(g, g_2) <\gamma$. 
  
  Since these two cases are similar, we assume that $d(g, g_1)<\gamma$. By $\gamma$-linearity, $\|g_1^2\|_d$ is in either
 $
  2\|g_1\|_d+I(\gamma)$ or $I(\gamma)$. Using $\|g_0^2\|_d \in 2\|g_0\|_d+I(\gamma)$ and the assumption that $\|g\|_d < \|g_0\|_d-2\gamma$, in either case, we have 
  \begin{equation} \label{g1andg0}
      \| g^2_1\|_d< \|g^2_0\|_d-2\gamma. 
  \end{equation}
   
  Since $g^{-1}_0g_1= g_1 g_0^{-1}$, and by $\gamma$-linearity of $d$,  we get
  \begin{equation}\label{eq:d(g_1^2, g^2_0)}
       d(g_1^2, g^2_0) = d( \id, g^{-2}_1g_0^{2}) = d( \id, (g^{-1}_1g_0)^2) \in  \{ 0, 2d(g_1, g_0)\} + I(\gamma).  
       \end{equation}
  By \eqref{eq: condition (2)}, we have $\|g_1\|_d+ d(g_1, g_0) \in \| g_0\|_d+I(\gamma)$. Recalling that $\|g_1\|_d =\|g\|_d >2\gamma$, and from \eqref{eq: condition (1)} and \eqref{eq:d(g_1^2, g^2_0)}, we have
  \begin{equation}\label{eq: g_1^2, g_0^2}
  d(g^2_1, g^2_0)< 2\| g_0\|_d-3\gamma = \|g_0^2 \| -2\gamma.
    \end{equation}
By \eqref{g1andg0}, \eqref{eq: g_1^2, g_0^2}, and the $\gamma$-linearity of $d$, we have
\[
\|g_1^2\|_d\in \|g_0^2\|_d-d(g_1^2,g_0^2)+I(\gamma).
\]
Therefore by \eqref{eq: condition (2)} and \eqref{eq:d(g_1^2, g^2_0)}, we have either
\begin{align*}
   \|g_1^2\|_d\in  2\|g_1\|_d +I(5\gamma) \quad\text{or}\quad \|g_1^2\|_d\in  2\|g_0\|_d +I(3\gamma).
\end{align*}
As $\|g_1\|_d^2 \leq  2\|g_1\|+ \gamma <2\|g_0\|-5\gamma$, we must have $\|g_1^2\| \in 2\|g_1\|+I(5 \gamma)$. Now, since $\|g^{-1}_1g\|_d= d(g_1, g) < \gamma$, again by the $\gamma$-linearity we conclude that
\[
d(g^2_1, g^2)= \|(g^{-1}_1g)^2\|_d < 3\gamma. 
\]
  Thus, $\|g^2\|_d \in  2\|g\|_d+I(9\gamma).$
  
  Finally, we consider the other special case where $\|g_0\|_d+2\gamma<\|g\|_d<\rho/2-16\gamma$. For $g_1=\exp(t_1X)$ with $t_1\in [0,t_0]$,  we have $\|g_1^2\|_d\in 2\|g_1\|+I(8\gamma)$ by a similar argument as above. Using continuity, we can choose $t_1$ such that $\|g_1^2\|_d=\|g\|_d$, and let $g_2 = g_1^{-1}$. The argument goes in exactly the same way with the role of $g_1$ replaced by $g_1^2$ and the role of $g_2$ replaced by $g_2^2$.
\end{proof}

In $\RR$, there are obviously two directions: positive and negative, or simply left and right. Proposition~\ref{prop: localmonotoneimplyglobalmonotone} ensured that directions are also well-defined for our almost linear pseudometric:

\begin{definition}
  Suppose $d$ is $\gamma$-linear.
We define $s(g_1,g_2)$ to be the {\bf relative sign}   for  $g_1, g_2 \in G$ satisfying   $\|g_1\|_d +\|g_2\|_d< \rho -\gamma$ by
$$ s(g_1, g_2) = 
\begin{cases}
0 &\text{ if }  \min\{\|g_1\|_d, \|g_2\|_d \}\leq  4\gamma, \\
1 &\text{ if } \min\{\| g_1 \|_d, \| g_2 \|_d \}>  4\gamma \text{ and }  \| g_1g_2 \|_d \in \| g_1\|_d + \| g_2 \|_d + I(\gamma).  \\
-1 &\text{ if } \min\{\| g_1 \|_d, \| g_2 \|_d \}>  4\gamma \text{ and }  | g_1g_2 |_d \in \big| \| g_1\|_d - \| g_2 \|_d \big| + I(\gamma).
\end{cases}
$$
Note that this is well-defined because when $\min\{\| g_1 \|_d, \| g_2 \|_d \}\geq  4\gamma$ in the above definition, the differences between $\big| \| g_1\|_d - \| g_2 \|_d \big|$ and $ \| g_1\|_d + \| g_2 \|_d$ is at least $6\gamma$.
\end{definition}

 The following lemma gives us tools to relate signs between different elements.

\begin{proposition}\label{prop: propertoesofsign}
Suppose $d$ is $\gamma$-linear and $\gamma$-monotone. Then for $g_1$, $g_2$,  and $g_3$ in $N(\rho/4-\gamma)\setminus N(4\gamma)$, we have the following
\begin{enumerate}[\rm (i)]
    \item $s(g_1, g^{-1}_1)=-1$ and $s(g_1, g_1)=1$.
    \item $s(g_1, g_2) =s(g_2, g_1)$.
    \item $ s(g_1, g_2) = s(g^{-1}_1, g^{-1}_2) = -s(g^{-1}_1, g_2) = -s(g_1, g^{-1}_2)  $.
    \item $s(g_1, g_2)s(g_2,g_3)s(g_3,g_1)=1.$
    \item If $\|g_1\|_d\leq  \|g_2\|_d$, and $g_1g_2$ is in $N(\rho/4-\gamma)\setminus N(4\gamma)$, then $$s(g_0,g_1g_2) = s(g_0,g_2g_1) = s(g_0, g_2).$$
\end{enumerate}
\end{proposition}

\begin{proof}
As  $g_1$, $g_2$,  and $g_3$ are in $N(\rho/4-\gamma)\setminus N(4\gamma)$, one has $s(g_i, g_j) \neq 0$ for all $i, j \in \{1, 2, 3\}$.
The first part of (i) is immediate from the fact that $\|\id\|_d =0$, and the second part of (i) follows from  the $\gamma$-monotonicity and the definition of the relative sign. 

We now prove (ii). Suppose to the contrary that $s(g_1,g_2) = -s(g_2,g_1)$. Without loss of generality, assume $s(g_1, g_2)=1$. Then $\|g_1g_2g_1g_2\|_d$ is in $2\|g_1g_2\|_d  + I(\gamma)$, which is a subset of $2\|g_1\|_d+ 2\|g_2\|_d + I(3 \gamma)$.  On the other hand, as $s(g_2,g_1)=-1$, we have
\[
\|g_1g_2g_1g_2\|_d\in \big| \|g_1\|_d \pm (  \|g_2\|_d -\|g_1\|_d)\pm\|g_2\|_d \big| +I(3 \gamma). 
\]
This contradicts the assumption that $g_1$ and $g_2$ are not in $N(4\gamma)$.

Next, we prove the first and third equality in (iii).
Note that $\|g\|_d = \|g^{-1}\|_d$ for all $g\in G$ as $d$ is symmetric and invariant. Hence, $\|g_1g_2\|_d = \|g_2^{-1}g_1^{-1}\|_d$. This implies that $s(g_1, g_2) = s(g_2^{-1}, g_1^{-1})$. Combining  with (ii), we get the first equality in (iii). The third equality in (iii) is a consequence of the first  equality in (iii). 

Now, consider the second equality in (iii). Suppose $s(g^{-1}_1, g_2^{-1}) = s(g^{-1}_1, g_2)$. Then, from (ii) and the first equality of (iii), we get $s(g_2, g_1) = s(g^{-1}_1, g_2)$. Hence, either 
$$\|g_2g_1 g_1^{-1} g_2\|_d \in 2\left(\|g_1\|_d+ \|g_2\|_d\right) + I(3 \gamma)$$  or $$\|g_2g_1 g_1^{-1} g_2\|_d \in 2\big|\|g_1\|_d-\|g_2\|_d\big| + I(3 \gamma).$$ On the other hand, $\|g_2g_1 g_1^{-1} g_2\|_d = \|g_2^2\|_d$, which is in $2\|g_2\|_d + I(\gamma)$. We get a contradiction with the fact that $g_1$ and $g_2$ are not in $N( 4\gamma)$.

We now prove (iv).
Without loss of generality, assume $\|g_1\|_d \leq \|g_2\|_d \leq \|g_3\|_d$. Using (iii) to replace $g_3$ with $g_3^{-1}$ if necessary, we can further assume that  $s(g_2, g_3) =1$. We need to show that $s(g_1, g_2)= s(g_1,g_3)$. Suppose to the contrary. Then, from (iii), we get $s(g_1, g_2) = s(g^{-1}_1,g_3)$.  Using (iii) to replacing $g_1$ with $g_1^{-1}$ if necessary, we can assume that 
$s(g_1, g_2) = s(g^{-1}_1,g_3)=1.$ Using (ii), we get $s(g_2,g_1)=1$. Hence,  either 
$$\|g_2g_1 g_1^{-1}g_3\|_d \in 2 \|g_1\|_d+\|g_2\|_d+ \|g_3\|_d+ I(3\gamma)$$ or  $$\|g_2g_1 g_1^{-1}g_3\|_d \in \|g_3\|_d -\|g_2\|_d +I(3\gamma).$$ On the other hand,  $\|g_2g_1 g_1^{-1}g_3\|_d = \|g_2g_3\|_d$ is in $\|g_2\|_d+\|g_3\|_d+I(\gamma)$. Hence, we get a contradiction to the fact that $g_1$, $g_2$, and $g_3$ are not in $N(4\gamma)$.

Finally, we prove (v). Using (iv), it suffices to show  $s(g_1g_2, g_2) = s(g_2g_1, g_2)=1$. We will only show the former, as the proof for the latter is similar. Suppose to the contrary that $s(g_1g_2, g_2)=-1$. Then $\|g_1g^{2}_2\|_d$ is in $\big|\|g_1g_2\|_d -\|g_2\|_d\big|+ I(\gamma)$, which is a subset of $\|g_1\|_d+ I(2\gamma)$. On the other hand, $\|g_1g^{2}_2\|_d$ is also in $\big|\|g_1\|_d -\|g^2_2\|_d\big|+ I(\gamma)$ which is a subset of $ 2\|g_2\|_d - \|g_1\|_d + I(2\gamma). $ Hence, we get a contradiction with the assumption that $g_1$ and $g_2$ are not in $N(4\gamma)$.
\end{proof}

  The notion of relative sign corrects the ambiguity in calculating distance, as can be seen in the next result.

\begin{lemma} \label{lem: Estimationusinsignchange}
Suppose $d$ is $\gamma$-monotone  $\gamma$-linear, and $g_1$ and $g_2 $ are in $N(\rho/16- \gamma)$ with $\|g_1\|_d \leq \|g_2\|_d$. Then we have the following
\begin{enumerate}[\rm (i)]
    \item Both $\|g_1g_2\|_d$ and $\|g_2g_1\|_d$ are in  $ s(g_1, g_2)\|g_1\|_d + \|g_2\|_d  + I(5\gamma)$.
    \item  If $g_0$ is in $N(\rho/4) \setminus  N(4 \gamma)$, then both $ s(g_0, g_{1}g_{2}) \|g_{1}g_{2}\|_d$ and $ s(g_0, g_{2}g_{1}) \|g_{2}g_{1}\|_d$ are in 
    $$s(g_0, g_{1}) \|g_{1}\|_d +  s(g_0, g_2) \|g_{2}\|_d + I(25\gamma).  $$
\end{enumerate}
\end{lemma}

\begin{proof}
We first prove (i). When $g_1, g_2 \notin N(4\gamma)$, the statement for $\|g_1g_2\|_d$ is immediate from the definition of the relative sign, and the statement for $\|g_2g_1\|_d$ is a consequence of Proposition~\ref{prop: propertoesofsign}(ii). Now suppose $\|g_1\|_d< 4 \gamma$. From the $\gamma$-linearity, we have  $$ \|g_2\|_d-\|g_1\|_d - \gamma < \|g_1g_2\|_d  <\|g_1\|_d+ \|g_2\|_d + \gamma. $$  
We deal with the case where $\|g_2\|_d< 4 \gamma$ similarly.

We now prove (ii). Fix $g_0$ in $N(\rho/4-\gamma)\setminus N(4\gamma)$. We will consider two cases, when $g_1$ is not in $N(4\gamma)$ and when $g_1$ is in $N(4\gamma)$. 
Suppose we are in the first case, that is $g_1 \notin N(4\gamma)$. As $\|g_1\|_d \leq \|g_2\|_d$, we also have $g_2 \notin N(4\gamma)$.
If both $g_1g_2$ and $g_2g_1$ are not in $N(4\gamma)$,  then the desired conclusion is a consequence of (i) and Proposition~\ref{prop: propertoesofsign}(iv, v). Within the first case, it remains to deal with the situations where $g_1g_2$ is in $N(4\gamma)$ or $g_2g_1$ is in $N(4\gamma)$. 

Since these two situations are similar, we may assume $g_1g_2$ is in $N(4\gamma)$. From (i), we have $s(g_1, g_2)=-1$ and $\|g_2\|_d - \|g_1\|_d$ is at most $5\gamma$. Therefore, $\|g_2g_1\|_d$ is in $I(6\gamma)$. By Proposition~\ref{prop: propertoesofsign}(iv), we have $s(g_0, g_1) = -s(g_0, g_2)$, and so
$$s(g_0, g_{1}) \|g_{1}\|_d +  s(g_0, g_2) \|g_{2}\|_d \in I(6\gamma).$$ Since both $ s(g_0, g_{1}g_{2}) \|g_{1}g_{2}\|_d$ and $ s(g_0, g_{2}g_{1}) \|g_{1}g_{2}\|_d$ are in $I(6\gamma)$,  they are both in
$s(g_0, g_{1}) \|g_{1}\|_d +  s(g_0, g_2) \|g_{2}\|_d + I(12\gamma)$ giving us the desired conclusion.

Continuing from the previous paragraph, we consider the second case when $g_1$ is in $N(4\gamma)$. If $g_2$ is in $N(16\gamma)$, then both $\|g_{1}g_{2}\|_d$ and  $ \|g_{2}g_{1}\|_d$ are in $I(25\gamma)$ by (i), and  the desired conclusion follows.  Now suppose $g_2$ is not in $N(16\gamma)$. Then from (i) and the fact that $g_1\in N(4\gamma)$, we get $g_1g_2$ and $g_2g_1$ are both not in $N(4\gamma)$. Note that $s(g_1g_2, g^{-1}_2)= -1$, because otherwise we get $$\|g_1\|_d \geq \|g_1g_2\|_d+\| g^{-1}_2\|_d -5 \gamma > 4\gamma.$$ A similar argument gives $s(g_2^{-1}, g_2g_1)=-1$. Hence, 
$s(g_1g_2, g_2)= s(g_2g_1, g_2)=1.$ By Proposition~\ref{prop: propertoesofsign}(v), we get $$s(g_0, g_2)= s(g_0, g_1g_2) = s(g_0, g_2g_1).$$ From (i), $\|g_1g_2\|_d$ and $\|g_2g_1\|
_d$ are both in $\|g_2\|_d+I(9\gamma)$. On the other hand, as $s(g_0, g_{1})=0$, we have
$s(g_0, g_{1}) \|g_{1}\|_d +  s(g_0, g_2) \|g_{2}\|_d = s(g_0, g_2) \|g_{2}\|_d$. The desired conclusion follows. 
\end{proof}

The next corollary will be important in the subsequent development.

\begin{corollary}{\label{cor: welldefinedoftotalweight}}
Suppose $d$ is $\gamma$-linear and $\gamma$-monotone,  $g_0$ and $g_0'$ are elements in $N(\rho/4-\gamma)\setminus N(4\gamma)$, and $(g_1, \ldots, g_n)$ is a sequence with $g_i \in N(\rho/4-\gamma)\setminus N(4\gamma)$ for $i \in \{1, \ldots, n\}$. Then
$$   \left|\sum^n_{i=1} s(g_0, g_i) \|g_i\|_d \right|  =  \left| \sum^n_{i=1} s(g'_0, g_i) \|g_i\|_d \right|.  $$
\end{corollary}

\begin{proof}
As $s(g_0, g_i) = s(g_0', g_i)=0$ whenever $\|g_i\|_d< 4\gamma$, we can reduce to the case where $\min_{1\leq i\leq n} \|g_i\|_d \geq 4 \gamma$. Using Proposition~\ref{prop: propertoesofsign}(iii) to replace $g_0$ with $g_0^{-1}$ if necessary, we can assume that $s(g_0, g_1) = s(g'_0, g_1)$. Then by Proposition~\ref{prop: propertoesofsign}(iii), $s(g_0, g_i) = s(g'_0, g_i)$ for all $i \in \{1, \ldots, n\}$. This gives us the desired conclusion.
\end{proof}

The following auxiliary lemma allows us to choose $g_0$ as in Corollary~\ref{cor: welldefinedoftotalweight}.

\begin{lemma} \label{lem: nonemptytodefine}
The set $N(\rho/4-\gamma)\setminus N(4\gamma)$ is not empty.
\end{lemma}

\begin{proof}
It suffices to show that $\mu_G(N(4\gamma))< \mu_G(N(\rho/4-\gamma)$. 
Since $\id$ is in $N(4\gamma)$, $N(4\gamma)$ is a nonempty open set and has $\mu_G(N(4\gamma))>0$. Therefore,  $N^2(4\gamma)$ and $N^4(4\gamma)$ are also open. By $\gamma$-linearity, we have 
$$N^2(4\gamma) \subseteq N_{9\gamma}\quad \text{and}\quad N^4(4\gamma) \subseteq N_{19\gamma}.$$ As $19\gamma<\rho$, we have $N^4(4\gamma) \neq G$. Using Lemma~\ref{lem: when a+b>G}, we get
$$\mu_G( N^2(4\gamma))\leq  2/3 \quad\text{and}\quad  \mu_G( N(4\gamma))<  1/3.$$ Hence, by Kemperman's inequality $\mu_G(N(4\gamma))< \mu_G(N^2(4\gamma)) \leq \mu_G(N(\rho/4-\gamma))$, which is the desired conclusion.
\end{proof}

We end the section by giving the following definition. 
\begin{definition}
Suppose $(g_1, \ldots, g_n)$ is a sequence of elements in $N(\rho/4-\gamma)\setminus N(4\gamma)$. We set
$$t(g_1, \ldots, g_n) = \left|\sum^n_{i=1} s(g_0, g_i) \|g_i\|_d \right|  $$
with $g_0$ is an arbitrary element in $N(\rho/4-\gamma)\setminus N(4\gamma)$, and call this the {\bf total weight} associated to $(g_1, \ldots, g_n)$. This is well-defined by Corollary~\ref{cor: welldefinedoftotalweight} and Lemma~\ref{lem: nonemptytodefine}.
\end{definition}

\section{Group homomorphisms onto tori}\label{sec: 7.3}
In this section, we will use the relative sign function and the total weight function defined in Section~\ref{section:pseudometric} to define a universally measurable multivalued group homomorphism onto $\TT$. We will then use a number or results in descriptive set theory and geometry to refine this into a continuous group homomorphism.

We keep the setting of Section~\ref{section:pseudometric}. Let $s$ and $t$ be the relative sign function and the total weight function defined earlier. Set $\lambda= \rho/36$, and $N[\lambda] = \{g \in G : \|g\|_d \leq \lambda\}$. The set $N[\lambda]$ is compact, and hence measurable. Moreover, Lemma~\ref{lem: Estimationusinsignchange} is applicable when $g_0$ is an arbitrary element in $N(\rho/4-\gamma)\setminus N(4\gamma)$, and $g_1$ are $g_2$ are in $N[\lambda]$. We first introduce several important definitions. 

\begin{definition}
A sequence $(g_1, \ldots, g_n)$ of elements in $G$ is a {\bf $\lambda$-sequence} if $g_i$ is in $N[\lambda]$ for all $i \in \{1, \ldots, n\}$.
\end{definition}

We are interested in expressing an arbitrary $g$ of $G$ as a product of a $\lambda$-sequence where all components are ``in the same direction''. The following notion captures that idea. 
\begin{definition}
A $\lambda$-sequence $(g_1, \ldots, g_n)$  is {\bf irreducible} if for all $2 \leq j \leq 4$, we have $$ g_{i+1} \cdots g_{i+j} \notin N(\lambda).$$    
 A {\bf concatenation} of a $\lambda$-sequence $(g_1, \ldots, g_n)$ is a $\lambda$-sequence $(h_1, \ldots, h_m)$ such that there are  $0 =k_0< k_1< \cdots< k_m =n$ with 
 $$h_i = g_{k_{i-1}+1} \cdots g_{k_i} \text{ for } i \in \{1, \ldots, m\}.$$ 
 \end{definition}
 The next lemma allows us to reduce an arbitrary sequence to irreducible $\lambda$-sequences via concatenation.

\begin{lemma} \label{lem: concatenationcomparison}
Suppose $d$ is $\gamma$-linear and $\gamma$-monotone, and $(g_1, \ldots, g_n)$ is a $\lambda$-sequence. Then $(g_1, \ldots, g_n)$ has an irreducible  concatenation $(g'_1, \ldots, g'_m)$ with
$$ t(g'_1, \ldots, g'_m) \in t(g_1, \ldots, g_n) + I(25(n-m)\gamma). $$
\end{lemma}

\begin{proof}
The statement is immediate when $n=1$. Using induction, suppose we have proven the statement for all smaller values of $n$.
If $(g_1, \ldots, g_n)$ is irreducible, we are done. Consider the case where $g_{i+1}g_{i+2}$ is in $N(\lambda)$ for some $0 \leq i \leq n-2$. Fix $g_0$ in $N(\lambda/4-\gamma)\setminus N(4\gamma)$. Using Lemma~\ref{lem: Estimationusinsignchange}(ii)
$$ s(g_0, g_{i+1}g_{i+2}) \|g_{i+1}g_{i+2}\|_d \in s(g_0, g_{i+1}) \|g_{i+1}\|_d +  s(g_0, g_{i+2}) \|g_{i+2}\|_d + I(25\gamma).  $$
From here, we get the desired conclusion. The cases where either $g_{i+1}g_{i+2}g_{i+3}$ for some $0 \leq i \leq n-3$ or $g_{i+1}g_{i+2}g_{i+3}g_{i+4}$ is in $N(\lambda)$ for some $0 \leq i \leq n-4$ can be dealt with similarly.
\end{proof}

The following lemma makes the earlier intuition of ``in the same direction'' precise:  

\begin{lemma} \label{lem: almostmonotonicityofirredseq}
Suppose $d$ is $\gamma$-linear and $\gamma$-monotone, $g_0$ is in $ N(\rho/4-\gamma)\setminus N(4\gamma)$, and $(g_1, \ldots, g_n)$ is an irreducible $\lambda$-sequence. Then for all $i$, $i'$, $j$, and $j'$ such that  $2\leq j, j' \leq 4$, $0 \leq i \leq n-j$, and $0 \leq i' \leq n-j'$, we have
$$ s( g_0, g_{i+1}\cdots g_{i+j}) = s(g_0,  g_{i'+1}\cdots g_{i'+j'}).  $$
\end{lemma}
\begin{proof}
It suffices to show for fixed $i, j$ with $0 \leq i \leq n-j-1$ and  $2 \leq j \leq 3$ that $$s( g_0, g_{i+1}\cdots g_{i+j}) = s(g_0,  g_{i+1}\cdots g_{i+j+1}).$$ Note that both $g_{i+1}\cdots g_{i+j}$ and $g_{i+1}\cdots g_{i+j+1}$ are in $N(\rho/4-\gamma)\setminus N(4\gamma)$. Hence, applying Proposition~\ref{prop: propertoesofsign}(iv), we reduce the problem to showing  $$s( g^{-1}_{i+j}\cdots g_{i+1}^{-1}, g_{i+1}\cdots g_{i+j+1)}) =-1  .$$ This is the case because otherwise, $\|g_{i+j+1}\|_d \geq 2 \lambda - \gamma > \lambda$, a contradiction.    
\end{proof}

We now get a lower bound for the total distance of an irreducible $\lambda$-sequence:

\begin{corollary} \label{Cor: Totallengthofirreduciblesequence}
Suppose $d$ is $\gamma$-linear and $\gamma$-monotone, and $(g_1, \ldots, g_n)$ is an irreducible $\lambda$-sequence. Then
$$   t(g_1, \ldots, g_n) > n\lambda/4. $$
\end{corollary}
\begin{proof}
If $n=2k$, let $h_i = g_{2i-1}g_{2i}$ for $i \in \{1, \ldots, k\}$. If $n=2k+1$, let $h_i = g_{2i-1}g_{2i}$ for $i \in \{1, \ldots, k-1\}$,   and $h_k = g_{2n-1}g_{2n}g_{2n+1}$. From Lemma~\ref{lem: Estimationusinsignchange},  we have 
\begin{equation}\label{eq: sequence estimate}
   t(h_1, \ldots, h_k) \in t(g_1, \ldots, g_n) + I(25(n-k)\gamma). 
\end{equation}
As $(g_1, \ldots, g_n)$ is irreducible, $h_i$ is in $N(3\lambda)\setminus N(\lambda)$ for $i \in \{1, \ldots, k\}$.
By Lemma~\ref{lem: almostmonotonicityofirredseq}, we get $s(g_0,h_i) = s(g_0, h_j)$ for all $i$ and $j$ in $i \in \{1, \ldots, k\}$. Thus, by the definition of the total weight again, 
$
t(h_1, \ldots, h_k)> n\lambda/3. 
$ Combining with the assumption on $\lambda$ and \eqref{eq: sequence estimate}, we get $t(g_1, \ldots, g_n)>n\lambda/3-11n\gamma > n\lambda/4 $. \end{proof}

  When $(g_1.\dots,g_n)$ is an irreducible $\lambda$-sequence,  $g_1\cdots g_m$ is intuitively closer to $g_0$ than $g_1\cdots g_{m+k}$ for some positive $k$. However, as $G$ is compact, the sequence may ``return back'' to $\id$ when $n$ is large. The next proposition provides a lower bound estimate on such $n$.

\begin{lemma}[Monitor lemma]\label{prop: lower bound on n}
Suppose $d$ is $\gamma$-linear and $\gamma$-monotone, and $(g_1, \ldots, g_n)$ is an irreducible $\lambda$-sequence with $g_1 \cdots g_n =\id$. Then $n \geq 1/ \mu_G(N(4 \lambda)).$
\end{lemma}

\begin{proof}
Let $m>0$. For convenience, when $m>n$ we write $g_m$ to denote the element $g_i$ with $i\leq n$ and $i\equiv m\pmod n$.  Define
\[
N^{(m)}(4 \lambda) = \{ g\in G \mid d( g,  g_1\cdots g_{m})< 4\lambda\}.
\]
Note that we have $N^{(m)}(4\lambda) = N^{(m')}(4\lambda) $ when $m \equiv m' \pmod{n}$. By invariance of $d$ and $\mu_G$, clearly $\mu_G(N^{(m)}(4 \lambda)) = \mu_G( N(4\lambda))$ for all $m$. We also write $N^{(0)}(4\lambda)=N(4\lambda)$. We will show that 
\[
G = \bigcup_{m\in\ZZ} N^{(m)}(4\lambda) =\bigcup_{m =0}^{n-1} N^{(m)}(4\lambda),
\]
which yields the desired conclusion. 

As $g_1 \cdots g_n =\id$, we have $\id$ is in $N^{(0)}(2\lambda)$, and hence in $\bigcup_{m \in\ZZ} N^{(m)}(4\lambda)$. As every element in $G$ can be written as a product of finitely many elements in $N(\lambda)$, it suffices to show for every $g \in \bigcup_{m \in \ZZ} N^{(m)}(4\lambda)$ and $g' = gh$ with $h \in N(\lambda)$ that $g'$ is in $\bigcup_{m \in \ZZ} N^{(m)}(4\lambda)$. The desired conclusion then follows from the induction on the number of translations in $N(\lambda)$.

Fix $m$ which minimizes $d(g, g_1 \ldots g_m)$. We claim that $d(g, g_1 \ldots g_m)< 2\lambda+\gamma$. This claim gives us the desired conclusion because we then have $d(g', g_1 \ldots g_m) < 3\lambda+ 2\gamma < 4\lambda$ by the $\gamma$-linearity of $d$.   

We now prove the claim that  $d(g, g_1 \ldots g_m)< 2\lambda+\gamma$. Suppose to the contrary that $d(g, g_1 \ldots g_m)\geq 2 \lambda+\gamma$. Let $u=(g_1\cdots g_m)^{-1}g$. Now by Lemma~\ref{lem: almostmonotonicityofirredseq} we have either $s(u, g_{m+1}g_{m+2})=1$, or $s(u, g_m^{-1}g_{m-1}^{-1})=1$. Suppose it is the former, since the latter case can be proved similarly. Then $s(u, g_{m+2}^{-1}g_{m+1}^{-1}) =-1$. Note that $g = g_1\cdots g_m u = (g_1 \cdots g_{m+2}) g^{-1}_{m+2}g^{-1}_{m+1}u$. By the definition of $u$, and the linearity of $d$, we have $\|u\|_d\geq 2\lambda+\gamma>\|g_{m+1}g_{m+2}\|_d$, therefore by the irreducibility we have
\begin{align*}
d(g, g_1\cdots g_{m+2}) &= \|g^{-1}_{m+2}g^{-1}_{m+1}u\|_d\\
&< \|u\|_d-\|g^{-1}_{m+2}g^{-1}_{m+1}\|_d +\gamma< \|u\|_d-\lambda +\gamma< \|u\|_d.
\end{align*}
This contradicts our choice of $m$ having $d(g, g_1, \ldots, g_m)$ minimized.
\end{proof}

In the later proofs of this section, we will fix an irreducible $\lambda$ sequence $g_1\cdots g_n=\id$ to serve as ``monitors''. As each element of $G$ will be captured by one of the monitors, this will help us to bound the error terms in the final almost homomorphism we obtained from the pseudometric. 

\begin{definition}
 Suppose $d$ is $\gamma$-linear and $\gamma$-monotone, and $N(\rho/4-\gamma)\setminus N(4\gamma) \neq \varnothing$. Define the {\bf returning weight} of $d$ to be
$$ \omega = \inf\{ t(g_1, \ldots, g_n) : ( g_1, \ldots, g_n) \text{ is an irreducible } \lambda\text{-sequence with } g_1\cdots g_n=\id \}.  $$
     \end{definition}
     
The following corollary translate Lemma~\ref{prop: lower bound on n} to a bound on such $\omega$:

 \begin{corollary}\label{cor: a_lambdanew}
 Suppose $d$ is $\gamma$-linear and $\gamma$-monotone, and $\omega$ is the returning weight of $d$. Then we have the following:
\begin{enumerate}[\rm (i)]
    \item $\lambda/ 4\mu_G(N(4 \lambda)) \leq \omega \leq 4\lambda/ \mu_G(N(\lambda)).$
    \item There is  an irreducible $\lambda$-sequence $(g_1, \ldots, g_n)$ such that $\omega = t(g_1, \ldots, g_n)$ and $ 1/ \mu_G(N(4 \lambda)) \leq n \leq 4/ \mu_G(N(\lambda))$.
\end{enumerate}
 \end{corollary}

\begin{proof}
Note that each irreducible $\lambda$-sequence $(g_1, \ldots, g_n)$  has $n \geq 1/ \mu_G(N(4\lambda))$ by using Lemma~\ref{prop: lower bound on n}.
Hence,  by Corollary~\ref{Cor: Totallengthofirreduciblesequence}, we get 
$
\omega \geq \lambda/ 4\mu_G(N(4 \lambda)) .
$
On the other hand, by Lemma~\ref{lem: when a+b>G}, $G= (N(\lambda))^k$ for all $k> 1/\mu_G(N(\lambda))$. Hence, with Lemma~\ref{lem: concatenationcomparison}, there is an irreducible $\lambda$-sequence $(g_1, \ldots, g_n)$ with $g_1\cdots g_n =\id$ and
$n \leq  4/ \mu_G(N(\lambda)) $. From the definition of $t$, we get $\omega \leq  4\lambda/ \mu_G(N(\lambda))$.

Now if an irreducible $\lambda$-sequence $(g_1, \ldots, g_n)$ has $n > 4/ \mu_G(N(\lambda))$, then by (i) and Corollary~\ref{Cor: Totallengthofirreduciblesequence}, 
\[t(g_1, \ldots, g_n) > \frac{4\lambda}{\mu_G(N(\lambda))}\geq \omega,
\]
a contradiction. Therefore, we have
\begin{align*}
\omega  = \inf\{ t(g_1, \ldots, g_n) : &\,( g_1, \ldots g_n) \text{ is an irreducible } \lambda\text{-sequence with }\\
&\, n \leq  4/ \mu_G(N(\lambda))\text{ and } g_1\cdots g_n=\id \}.
\end{align*}
For fixed $n$ the set of irreducible $\lambda$-sequence of length $n$ is closed under taking limit. Hence, we obtain desired $(g_1, \ldots, g_n) $ using the Bozalno--Wierstrass Theorem.
\end{proof}

  The next lemma allows us to convert between $\mu_G(N(\lambda))$ and $\mu_G(N(4\lambda))$:
\begin{lemma}\label{lem: N lambda}
Suppose $d$ is $\gamma$-linear and $\gamma$-monotone. Then $$\mu_G(N(4\lambda))\leq 16\mu_G(N(\lambda)).$$ 
\end{lemma}
\begin{proof}
Fix $h\in N(\lambda)\setminus N(\lambda/2-\gamma)$. Such $h$ exists since
by $\gamma$-monotonicity we have $N^2(\lambda/2-\gamma)\subseteq N(\lambda)$, and by Kemperman's inequality, $\mu_G(N(\lambda)>2\mu_G( N(\lambda/2-\gamma))$. Let $g$ be an arbitrary element in $N(4\lambda)$, and  assume first $s(g,h)=1$. Let $k\geq0$ be an integer, and define $g_k=g(h^{-1})^k$. Then by Proposition~\ref{prop: propertoesofsign} and Lemma~\ref{lem: Estimationusinsignchange}, 
\[
\|g_k\|_d\in \|g\|_d-k\|h\|_d+I(5k\gamma) \text{ for } k <\|g\|_d/\|h\|_d. 
\]
Hence,  there is $k< 8$ such that $g_k\in N(\lambda)$.  When $s(g,h)=-1$, one can similarly construct $g'_k$ as $gh^k$, and find $k< 8$ such that $g'_k\in N(\lambda)$. Therefore 
\[
N(4\lambda)\subseteq \Big(\bigcup_{i=0}^7 N(\lambda)h^{i}\Big)\cup \Big(\bigcup_{j=0}^7 N(\lambda){h^{-j}}\Big).
\]
Thus, $\mu_G(N(4\lambda))\leq 16\mu_G(N(\lambda))$.
\end{proof}

The following proposition implicitly establish that $t$ defines an approximate multivalued group homomorphism from $G$ to $\RR/\omega\ZZ$.
\begin{proposition} \label{lem: estimationoftotallength}
Suppose $d$ is $\gamma$-linear and $\gamma$-monotone,  $\omega$ is the returning weight of $d$, and $(g_1, \ldots, g_n)$ is a $\lambda$-sequence with 
$g_1 \ldots g_n =\id$ and $n \leq  4/ \mu_G(N(\lambda)) $. 
Then
$$t(g_1, \ldots, g_n) \in \omega\ZZ + I(\omega/400).$$
\end{proposition}

 \begin{proof}
 Let $g_0$ be in $N(\rho/4-\gamma)\setminus N(4\gamma) $. Using Proposition~\ref{prop: propertoesofsign}(iii) to replace $g_0$ with $g_0^{-1}$ if necessary, we can assume that
 $$t(g_1,\dots,g_n) = \sum_{i=1}^ns(g_0, g_i) \|g_i\|_d.$$
As $n\leq 4/ \mu_G(N(\lambda))$,  we have
 $
 t(g_1, \ldots,\allowbreak g_n)\leq  4\lambda/ \mu_G(N(\lambda)).
 $
From Corollary~\ref{cor: a_lambdanew}(i), we have $\lambda \leq 4\omega\mu_G(N(4\lambda))$. Hence, 
\begin{equation}\label{eq: t(g_1,..,g_n)}
t(g_1, \ldots, g_n) < \frac{16\omega \mu_G(N(4\lambda))}{ \mu_G(N(\lambda))}.
\end{equation}
Using Corollary~\ref{cor: a_lambdanew} again, we obtain an irreducible $\lambda$-sequence $(h_1, \ldots, h_m)$   such that $t(h_1, \ldots, h_m)=\omega$ and $ 1/ \mu_G(N(4 \lambda)) \leq m \leq 4/ \mu_G(N(\lambda))$. Using Proposition~\ref{prop: propertoesofsign}(iii) to replace $(h_1, \ldots, h_m)$ with $(h_m^{-1}, \ldots, h_1^{-1})$ if necessary, we can assume that 
$$t(h_1, \ldots, h_m)=-\sum_{i=1}^ns(g_0, h_i) \|h_i\|_d.$$

We now define a sequence $(g'_1,\ldots,g_{n'}')$ such that
\begin{enumerate}
    \item $n'=n+km$ for some integer $k\geq0$.
    \item $g'_i=g_i$ for $1\leq i\leq n$. 
    \item For $i\geq n+1$, $g'_i=h_{j}$ with $j\equiv i-n\pmod m$.
\end{enumerate}
From the definition of the total weight, for $k< t(g_1, \ldots, g_n)/\omega$, we have
\[
t(g'_1,\ldots,g'_{n'})=t(g_1,\ldots,g_n)-k\omega.
\]
We choose an integer $k< t(g_1, \ldots, g_n)/\omega+1$ such that $|t(g'_1,\ldots,g'_{n'})|\leq\omega/2$. Then by \eqref{eq: t(g_1,..,g_n)}, and the trivial bound $\mu_G(N(\lambda))< \mu_G(N(4\lambda))$, we have
\begin{equation*}\label{eq: bound for n'}
n'<n+km< \frac{4}{\mu_G(N(\lambda))}+\left(\frac{ 16\mu_G(N(4\lambda))}{\mu_G(N(\lambda))}+1\right) \frac{4}{\mu_G(N(\lambda))} \leq \frac{ 72\mu_G(N(4\lambda))}{\mu_G^2(N(\lambda))}. 
\end{equation*}

  Note that $(g'_1, \ldots, g'_{n'})$ is a $\lambda$-sequence with 
$g'_1 \ldots g'_{n'} =\id$. We assume further that $0\leq t(g'_1, \ldots, g'_{n'}) <  \omega/2$ as the other case can be dealt with similarly.
 Obtain an irreducible concatenation  $(h'_1, \ldots, h'_{m'})$ of $(g'_1, \ldots, g'_{n'})$. 
 From Lemma~\ref{lem: concatenationcomparison}, we get
 \begin{equation*}
 t(h'_1, \ldots, h'_{m'})< t(g'_1, \ldots, g'_{n'}) +25(n'-m')\gamma
 \leq \frac{\omega}{2}+ \frac{ 1800\mu_G(N(4\lambda))\gamma}{\mu_G^2(N(\lambda))}.
  \end{equation*}
  Using Corollary~\ref{cor: a_lambdanew}(i) and Lemma~\ref{lem: N lambda}, we have
  \[
  \frac{ 1800\mu_G(N(4\lambda))\gamma}{\mu_G^2(N(\lambda))} \leq \frac{ 1800\mu_G(N(4\lambda))\gamma}{\mu^2_G(N(4\lambda))/16^2}<  \frac{5\cdot 10^5\gamma}{N(4\lambda)} \leq  5\cdot 10^5\gamma \frac{4\omega}{\lambda}.
  \]
  As  $\gamma< 10^{-8}\rho$, and $\lambda =\rho/16-\gamma$, one can check that the lass expression is at most $\omega/400$. Hence, 
  $t(h'_1, \ldots, h'_{m'})< \omega$. From the definition of $\omega$, we must have $t(h'_1, \ldots, h'_{m'}) =0$. Thus by Lemma~\ref{lem: concatenationcomparison} again, 
\[
t(g'_1, \ldots, g'_n) \in I(25n'\gamma)\subseteq I(\omega/400), 
\]
which completes the proof. 
  \end{proof}

Recall that a {\bf Polish space} is a topological space which is  separable and completely metrizable. In particular, the underlying topological space of any connected compact Lie group is a Polish space. Let $X$ be a Polish space. A subset $B$ of $X$ is {\bf Borel} if $B$ can be formed from open subsets of $X$   (equivalently, closed subsets of $X$) through taking countable unions, taking countable intersections, and taking complement. A function $f: X\to Y$ between Polish space is Borel, if the inverse image of any Borel subset of $Y$ is Borel. A subset $A$ of $X$ is  {\bf analytic} if it is the continuous image of another Polish space $Y$. Below are some standard facts about these notions; see~\cite{Kechris} for details. 

\begin{fact} \label{fact: analyticsets}
Suppose $X, Y$ are Polish spaces, and $f: X\to Y$ is continuous. We have the following:
\begin{enumerate}[\rm (i)]
    \item Every Borel subset of $X$ is analytic.
    \item  Equipping $X\times Y$ with the product topology, the graph of a Borel function from $X$ to $Y$ is analytic.
    \item The collection of analytic subsets of $X$ is closed under taking countable unions, taking intersections and cartesian products.
    \item Images of analytic subsets in $X$ under $f$ is analytic.
\end{enumerate}

\end{fact}

Given $x\in \RR$, let $\|x\|_\TT$ be the distance of $x$ to the nearest element in $\ZZ$. We now obtain a consequence of Lemma~\ref{lem: estimationoftotallength}.
 
 \begin{corollary}[Analytic multivalued almost homomorphism] \label{cor: analyticmulitvalued}
 There is an analytic subset $\Gamma$ of $G \times \TT$ satisfying the following properties:
  \begin{enumerate}[\rm (i)]
     \item The projection of $\Gamma$ on $G$ is surjective. 
     \item $(\id, \mathrm{id}_{\RR/\omega\ZZ})$ is in $\Gamma$.
      \item If  $g_1, g_2\in G$ and $t_1, t_2, t_3\in \RR$ are such that  $(g_1, t_1/\omega+\ZZ)$, $(g_2, t_2/\omega+\ZZ)$, and $(g_1g_2, t_3/\omega+\ZZ)$ in $\Gamma$, then 
      $$\|(t_1+t_2 -t_3)/\omega\|_\TT< 1/400.$$
      \item There are $g_1, g_2\in G$ and $t_1, t_2 \in \RR$ with that  $(g_1, t_1/\omega+\ZZ), (g_2, t_2/\omega+\ZZ) \in \Gamma$ and $\|(t_1-t_2)/\omega\|_\TT> 1/3$.
  \end{enumerate}
 \end{corollary}
 \begin{proof}
   Let $\Gamma$ consist of $(g,t/\omega+\ZZ) \in G \times \TT$ with $g\in G$ and $t\in \RR$ such that there is  $n\leq 1/\mu_G(N(\lambda))+1$ and an irreducible $\lambda$-sequence $(g_1, \ldots, g_n)$ satisfying 
   \[
   g= g_1 \cdots g_n \quad\text{and}\quad t = t(g_1, \ldots, g_n).
   \]
   Note that the relative sign function $s: G\times G \to \RR$ is Borel, the set $N[\gamma]$ is compact, and the function $x\to \| x\|_d$ is continuous. Hence, by Fact~\ref{fact: analyticsets}(i,ii), the function $(g_1, \ldots, g_n) \mapsto t(g_1, \ldots, g_n)$ is Borel, and its graph is analytic.
 For each $n$, by Fact~\ref{fact: analyticsets}(iii)
  \begin{align*}
  \widetilde{\Gamma}_n:= \{ (g,t, g_1, \ldots, g_n) &\in G \times \RR \times G^n :   \\
  &\|g_i\|_d < \lambda \text{ for } 1 \leq i \leq n, g=g_1\cdots g_n, t=t(g_1, \ldots, g_n)\} 
  \end{align*}
is analytic. Let ${\Gamma}_n$ be the image of $\widetilde{\Gamma}_n$ under the continuous map
$$ (g, t, g_1, \ldots, g_n) \mapsto (g, t/\omega+\ZZ).$$ Then by Fact~\ref{fact: analyticsets}(iv), ${\Gamma}_n$ is analytic. Finally,  $\Gamma = \bigcup_{n< 1/\mu_G(N(\lambda))+1}\Gamma_n$ is analytic by Fact~\ref{fact: analyticsets}(iii). 

We now verify that $\Gamma$ satisfies the desired properties. It is easy to see that (i) and (ii) are immediately from the construction, and (iii) is a consequence of  Lemma~\ref{lem: estimationoftotallength}. We now prove (iv). Using Corollary~\ref{cor: a_lambdanew}, we obtain an irreducible $\lambda$-sequence $(g_1, \ldots, g_n)$ with $t(g_1, \ldots,  g_n) = \omega$ and $n < 4/\mu_G(N(\lambda))$.  Note that 
\[
|t(g_1, \ldots, g_{k+1})-t(g_1, \ldots, g_k)|\leq \lambda.
\]
Hence, there must be $k \in \{1, \ldots n\}$ such that $\omega/3<  t(g_1, \ldots, g_k)<2\omega/3 $. Set $t_1 =  0$ and $t_2 =  t(g_1, \ldots, g_k) $ for such $k$. It is then easy to see that $\|(t_1-t_2)/\omega\|_\TT> 1/3$.
 \end{proof}

  To construct a group homomorphism from $G$ to $\TT$, we will need three more facts. Recall the following measurable selection theorem from descriptive set theory; see~\cite[Theorem~6.9.3]{Measuretheory}.
 
\begin{fact}[Kuratowski and Ryll--Nardzewski measurable selection theorem] \label{KRN}
Let $(X, \mathscr{A})$ be a measurable space, $Y$ a complete separable metric space equipped with the usual Borel $\sigma$-algebra, and $F$ a function on  $X$ with values in the set of nonempty closed subsets
of $Y$. Suppose that for every open $U \subseteq Y$, we have
$$ \{ a \in X: F(a) \cap U \neq \varnothing\} \in \mathscr{A}.$$
Then $F$ has a selection $f: X \to Y$ which is measurable with respect to $\mathscr{A}$.
 \end{fact}

  A  {\bf Polish group} is a topological group whose underlying space is a Polish space. In particular, Lie groups are Polish groups. 
A subset  $A$ of a Polish space $X$ is {\bf universally measurable} if $A$ is measurable with respect to every complete probability measure on $X$ for which every Borel set is measurable. In particular, every analytic set is universally measurable; see ~\cite{Rosendal} for details. A map $f: X \to Y$ between Polish spaces is {\bf universally measurable} if inverse images of open sets are universally measurable.  We have the following recent result from descriptive set theory by~\cite{Rosendal}; in fact, we will only apply it to Lie groups so a special case which follows from an earlier result by Weil~\cite[page 50]{Weils} suffices.
 
 \begin{fact}[Rosendal] \label{automaticcontinuity}
Suppose $G$ and $H$ are Polish groups, $f: G \to H$ is a  universally measurable group homomorphism. Then $f$  is continuous.
\end{fact}

   Finally, we need the following theorem from geometry by Grove, Karcher, and Ruh~\cite{almosthomo} and independently by Kazhdan~\cite{almosthomo2}, that in a compact Lie groups an almost homomorphism is always close to a homomorphism uniformly. We remark that the result is not true for general compact topological groups, as a counterexample is given in~\cite{almosthomonottrueingeneral}.

\begin{fact}[Grove--Karcher--Ruh; Kazhdan]\label{fact:almosthomo}
Let $G,H$ be compact Lie groups. There is a constant $c$ only depending on $H$, such that for every real number $q$ in $[0,c]$, if $\pi: G\to H$ is a  $q$-almost homomorphism, then there is a  homomorphism $\chi:G\to H$ which is $1.36q$-close to $\pi$. Moreover, if $\pi$ is universally measurable, then $\chi$ is universally measurable. When $H=\TT$, we can take $c=\pi/6$. 
\end{fact}

   The next theorem is the main result in this subsection. It tells us from an almost linear pseudometric, one can construct a group homomorphism to $\TT$. 

  \begin{theorem}\label{thm: homfrommeasurecompact}
Let $\lambda=\rho/36$, and $\gamma<10^{-6}\rho$. Suppose $d$ is $\gamma$-linear and $\gamma$-monotone. Then there is a continuous surjective group homomorphism $\chi: G \to \TT$ such that for all $g\in \ker(\chi)\cap N(\lambda)$, we have $\| g\|_d \in N(\lambda/2)$.
 \end{theorem}
 
 \begin{proof}
Let $\omega$ be the returning weight of $d$, and let $\Gamma$ be as in the proof of Corollary~\ref{cor: analyticmulitvalued}. Equip $G$ with the $\sigma$-algebra $\mathscr{A}$ of universally measurable sets. Then $\mathscr{A}$ in particular consists of analytic subsets of $G$.   Define $F$ to be the function from $G$ to the set of closed subsets of $\TT$ given by
\[
F(g) =\overline{\{ t/\omega+\ZZ : t\in \RR,  (g, t/\omega +\ZZ)\in \Gamma\}}. 
\]
If $U$ is an open subset of $\TT$, then $\{g \in G : F(g) \cap U \neq \varnothing\}$ is in $\mathscr{A}$ being the projection on $G$ of the analytic set $\{ (g, t/\omega+\ZZ) \in G \times \TT : (g, t/\omega+\ZZ) \in \Gamma \text{ and } t \in U \}$. Applying Fact~\ref{KRN}, we get a universally measurable $1/400$-almost homomorphism $\pi: G \to \TT$.  
Using Fact~\ref{fact:almosthomo}, we get a universal measurable group homomorphism $\chi: G \to \TT$ satisfying 
\[
\|\chi(g)-\pi(g)\|_\TT < 1.36/400 = 0.0068.
\]The group homomorphism $\chi$ is automatically continuous by Fact~\ref{automaticcontinuity}. Combining with Corollary~\ref{cor: analyticmulitvalued}(iv), we see that $\chi$ cannot be the trivial group homomorphism, so  $\chi$ is surjective.

Finally, for $g$ in $\ker (\chi)\cap (N(\lambda))$, we need to verify that $g$ is in $N(\lambda/2)$. Suppose to the contrary that $g \notin N(\lambda/2)$. 
Choose $n=\lfloor1/\mu_G(N(4\lambda)) \rfloor$, and $(g_1, \ldots, g_n)$ the $\lambda$-sequence such that $g_i =g$ for $i\in \{1, \ldots, n\}$.
By Proposition~\ref{prop: propertoesofsign}, $(g_1, \ldots, g_n)$ is irreducible.   Hence, by Lemma~\ref{prop: lower bound on n}, $t(g_1, \ldots, g_n)< \omega$. As $n \leq 1/\mu_G(N(\lambda)+1)$, by construction and Corollary~\ref{cor: analyticmulitvalued}(iii), we have
\[
\pi(g^n) \in  t(g_1,\ldots,g_n)/\omega+ I(1/400)+\ZZ = n\|g\|_d/\omega + I(1/400)+\ZZ. 
\]
Since $g^n\in\ker\chi$, we  have $\|\pi(g^n)\|_\TT<0.0068$, so $n\|g\|_d/\omega< (0.0068+1/400)$. 
By Corollary~\ref{cor: a_lambdanew}(i) and Lemma~\ref{lem: N lambda}, this implies
\[
\|g\|_d\leq \frac{(0.0068+1/400)\cdot 4\lambda}{\mu_G(N(\lambda))}\mu_G(N(4\lambda))<\frac{\lambda}{2},
\]
which is a contradiction. This completes our proof.
\end{proof}

\section{Preservation of small expansions under quotient}

In this section, $G$ is a connected unimodular group, and $H$ is a \emph{connected compact} normal subgroup of $G$, 
so $H$ and $G/H$ are unimodular by Fact~\ref{fact: unimodular}. Let $\mu_H$, and $\mu_{G/H}$ be the Haar measure on $G$, $H$, and $G/H$, and let $\mut_G$ and $\mut_{G/H}$ be the inner Haar measures on $G$ and $G/H$. Let $\pi:G\to G/H$ be the quotient map. Suppose $A,B$ are $\sigma$-compact subsets of $G$ with positive measures.

Suppose $r$ and $s$ are in $\RR$. We set
$$  A_{(r,s]}  := \{a \in A: \mu_H(A \cap aH) \in (r,s]  \}  $$
and
$$ \pi A_{(r,s]}:=  \{aH \in G \slash H : \mu_H( A \cap aH) \in (r,s]   \}. $$ 
In particular, $\pi A_{(r,s]}$ is the image of $A_{(r,s]}$ under the map $\pi$. We define $B_{(r',s']}$ and $\pi B_{(r',s']}$ likewise for $r', s' \in \RR$.
We have a number of immediate observations.  

\begin{lemma} \label{lem: corofmeasurability}
Let $r, s, r', s'$ be in $\RR^{>0}$. For all $aH \in \pi A_{(r,s]}$, $bH \in \pi B_{(r',s']}$, the sets $A_{(r,s]} \cap aH$, $B_{(r',s']} \cap bH$ are nonempty $\sigma$-compact.
For all subintervals $(r,s]$ of $(0,1]$,    $A_{(r,s]}$  is $\mu_G$-measurable and $\pi A_{(r,s]}$ is $\mu_{G/H}$-measurable.
\end{lemma}
\begin{proof}
The first assertion is immediate from the definition.
Let $\e_A$ be the indicator function of $A$. Then the function 
\begin{align*}
    \e^H_A: G/H &\to R\\
    aH &\mapsto \mu_H(A \cap aH)
\end{align*}
is well-defined and measurable by Lemma~\ref{lem: mesurability}. As $ \pi A_{(r,s]} = (\e^H_A)^{-1}(r, s]$ and 
$$A_{(r,s]} = A \cap \pi^{-1}( \pi A_{(r,s]}),$$ we get the second assertion.
\end{proof}

Note that $\pi A_{(r,s]}\pi B_{(r',s']}$ is not necessarily $\mu_{G/H}$-measurable, so Lemma~\ref{lem: intuition}(ii) does require the inner measure $\mut_{G/H}$.

\begin{lemma} \label{lem: intuition}
We have the following:
\begin{enumerate}[\rm (i)]
      \item For every $aH \in \pi A$ and $bH \in \pi B$,
    $$ \mu_H \big((A \cap aH)( B \cap bH) \big) \geq \min\{ \mu_H( A \cap aH)+ \mu_H(  B \cap bH), 1\}. $$
    \item  If $A_{(r,s]}$ and $B_{(r',s']}$ are nonempty, then 
        $$\mut_{G/H}(\pi A_{(r,s]} \pi B_{(r',s']}) \geq \min\{ \mu_{G/H}(\pi A_{(r,s]}) + \mu_{G/H}(\pi B_{(r',s']}), \mu_{G/H}(G/H)\}.$$
\end{enumerate}
\begin{proof}
Note that both $H$ and $G/H$ are connected. So (i) is a consequence of the generalized Kemperman inequality for $H$ (Fact~\ref{fact: GeneralKemperman}) and (ii) is a consequence of the generalized Kemperman inequality for $G/H$ (Fact~\ref{fact: GeneralKemperman}). 
\end{proof}
    
\end{lemma}

As the functions we are dealing with are not differentiable,  we will need Riemann--Stieltjes integral which we will now recall. Consider a closed interval $[a,b]$ of $\RR$, and functions $f:[a,b] \to \RR$ and $g: \RR \to \RR$. A {\bf partition} $P$ of $[a,b]$ is a sequence $(x_i)_{i=0}^n$ of real numbers with $x_0 =a$, $x_n=b$, and $x_i< x_{i+1}$ for $i \in \{0, \ldots, n-1\}$. For such $P$, its {\bf norm} $\Vert P\Vert$ is defined as $\max_{i=0}^{n-1}|x_{i+1}-x_i|$, and a corresponding {\bf partial sum} is given by $S(P,f, g) = \sum_{i=0}^n f(c_{i+1})(g(x_{i+1})-g(x_i))$ with $c_{i+1} \in [x_i, x_i+1]$.
We then define
$$ \int_{a}^b f(x)\d{g(x)}: = \lim_{\Vert P \Vert \to 0}  S(P,f,g)$$
if this limit exists where we let $P$ range over all the partition of $[a,b]$ and $S(P,f, g)$ ranges over all the corresponding partial sums of $P$. The next fact records some basic properties of the integral.
\begin{fact}\label{fact: RS int}
Let $[a,b]$, $f(x)$,  and $g(x)$ be as above. Then we have:
\begin{enumerate}[\rm (i)]
    \item{\rm (Integrability)} If $f(x)$ is continuous on $[a,b]$, and $g(x)$ is monotone and bounded on $[a,b]$, then $f(x)\d g(x)$ is  Riemann--Stieltjes integrable on $[a,b]$.
    \item{\rm (Integration by parts)} If $f(x) \d g(x)$ is Riemann--Stieltjes integrable on $[a,b]$, then $g(x) \d f(x)$ is also Riemann--Stieltjes integrable on $[a,b]$, and
    \[
\int_a^b f(x)\d g(x)=f(b)g(b)-f(a)g(a)-\int_a^b g(x)\d f(x).
\]
\end{enumerate}
\end{fact}

   The next lemma uses ``spillover'' estimate, which gives us a lower bound estimate on $\mu_G(AB)$ when the projection of $A$ and $B$ are not too large. 

\begin{lemma} \label{lem: Keyestimatequotientcompact}
Suppose $\mu_{G\slash H}( \pi A )+ \mu_{G \slash H} (\pi B) <1$. Set $\alpha =\sup_{a \in A} \mu_H( A \cap aH) $, $\beta =\sup_{b \in B} \mu_H( B \cap bH)$, and $\gamma =\max\{1, \alpha+\beta\}$. Then
\begin{align*}
    \mu_G(AB)  &\geq   \frac{\alpha+\beta}{\gamma}\left( \mu_{G/H}(\pi A_{(\alpha/\gamma,\alpha]}) + \mu_{G/H}(\pi B_{(\beta/\gamma,\beta]})\right)\\
    &\quad + \frac{\alpha+\beta}{\alpha} \mu_G(A_{(0,\alpha/\gamma]}) + \frac{\alpha+\beta}{\beta}\mu_G(B_{(0,\beta/\gamma])}.
\end{align*}
\end{lemma}
\begin{proof}
For $x \in (0,1]$, set $C_x = AB\cap \pi^{-1}(\pi A_{(x\alpha, \alpha]}\pi B_{(x\beta, \beta]})$. One first note that 
$$\mu_G(AB) \geq \mut_G(C_0) .$$ By Fact~\ref{fact: RS int}(1), $\mathrm{d}\mut_G( C_x)$ is Riemann--Stieltjes integrable on any closed  subinterval of $[0,1]$. Hence,
\[
    \mut_G( C_0)  = \mut_G ( C_{1/\gamma}) - \int_0^{\tfrac{1}{\gamma}}\d \mut_G(C_x). 
\]
Lemma~\ref{lem: corofmeasurability} and Lemma~\ref{lem: intuition}(1) give us that
\[
\mut_G (  C_{1/\gamma}) \geq \mut_{G/H} (  \pi{A}_{(\alpha/\gamma,\alpha]}   \pi{B}_{(\beta/\gamma,\beta]}). 
\]
Likewise, for $x, y \in \RR^{>0}$ with $x <y\leq 1/\gamma$,  $\mut_G (C_x) - \mut_G(C_y)  $ is at least  
$$r(\alpha+\beta) \left(\mut_{G/H} (  \pi {A}_{(x\alpha,\alpha]}   \pi {B}_{(x\beta,\beta]}) - \mut_{G/H}(\pi {A}_{(y\alpha,\alpha]}   \pi {B}_{(y\beta,\beta]})\right).$$ 

Therefore, 
 \[\mut_G( C_0) \geq \mut_{G/H} (  \pi{A}_{(\alpha/\gamma,\alpha]}   \pi{B}_{(\beta/\gamma,\beta]}) - \int_0^{\tfrac{1}{\gamma}}(\alpha+\beta) x\d \mut_{G/H}( \pi{A}_{(x\alpha, \alpha]} \pi{B}_{(x\beta, \beta]}).
 \]
Using integral by parts (Fact~\ref{fact: RS int}.2), we get
 \[\mut_G( C_0) \geq  \int_0^{\tfrac{1}{\gamma}}\mut_{G/H}( \pi{A}_{(x\alpha, \alpha]} \pi{B}_{(x\beta, \beta]})\d (\alpha+\beta) x.
 \]
Applying Lemma~\ref{lem: intuition}.2 and using the assumption that  $\mu_{G\slash H}( \pi A )+ \mu_{G \slash H} (\pi B) <1$ , we have
\[\mut_G( C_0)  \geq  \int_0^{\tfrac{1}{\gamma}}(\mu_{G/H}( \pi{A}_{(x\alpha, \alpha]}) + \mu_{G/H}( \pi{B}_{(x\beta, \beta]}))\d(\alpha+\beta) x.
 \]
Using integral by parts (Fact~\ref{fact: RS int}.2), we arrive at
\begin{align*}
    \mut_G( C_0)  &\geq    \frac{\alpha+\beta}{\gamma}\left(\mu_{G/H}( \pi{A}_{(\alpha/\gamma, \alpha]}) + \mu_{G/H}( \pi{B}_{(\beta/\gamma, \beta]}) \right)  \\
    & \quad  -\int_0^{\tfrac{1}{\gamma}}(\alpha+\beta) x\d( \mu_{G/H}( \pi{A}_{(x\alpha, \alpha]})+ \mu_{G/H}(\pi{B}_{(x\beta, \beta]})).
\end{align*}
As $\d( \mu_{G/H}( \pi{A}_{(x\alpha, \alpha]})+ \mu_{G/H}(\pi{B}_{(x\beta, \beta]}))= -\d( \mu_{G/H}( \pi{A}_{(0, x\alpha]})+ \mu_{G/H}(\pi{B}_{(0, x\beta]}))$,
\begin{align*}
    \mut_G( C_0)  &\geq    \frac{\alpha+\beta}{\gamma}\left(\mu_{G/H}( \pi{A}_{(\alpha/\gamma, \alpha]}) + \mu_{G/H}( \pi{B}_{(\beta/\gamma, \beta]}) \right)  \\
    & \quad  +\int_0^{\tfrac{1}{\gamma}}(\alpha+\beta) x\d( \mu_{G/H}( \pi{A}_{(0, x\alpha]})+ \mu_{G/H}(\pi{B}_{(0,x\beta]})).
\end{align*}  
Finally, recall that 
\[ \int_0^{1/\gamma} x\alpha\d\mu_{G/H}( \pi{A}_{(0, x\alpha]})= \mu_G({A}_{(0, \alpha/\gamma]}) \text{ and } \int_0^{1/\gamma} \beta x\d\mu_{G/H}( \pi{B}_{(0, x\beta]})= \mu_G({B}_{(0, \beta/\gamma]}).\] 
Thus, we arrived at the desired conclusion.
\end{proof}

 The next result in the main result in this subsection. It says if the projections of $A$ and $B$ are not too large, the small expansion properties will be kept in the quotient group.

\begin{theorem}[Quotient domination] \label{thm: criticality transfer}
Suppose $ {\mu}_{G/H}( \pi A) +  {\mu}_{G/H} ( \pi B ) <  \mu_{G/H}(G/H) $ and $\dis_G(A,B)< \min\{\mu_G(A), \mu_G(B)\}$.
Then there are $\sigma$-compact $A', B' \subseteq G/H$ such that
$$  \dis_{G \slash H}(A', B')  < 7\dis_G(A,B)$$
and $\max\{ \mu_{G} (A \tri \pi^{-1} A'), \mu_{G} (B \tri \pi^{-1} B')\} < 3\dis_G(A,B).$
\end{theorem}

\begin{proof}
Let $\alpha$ and $\beta$ be as in Lemma~\ref{lem: Keyestimatequotientcompact}. We first show that $\alpha+\beta\geq 1$. Suppose to the contrary that $\alpha+\beta<1$. Then Lemma~\ref{lem: Keyestimatequotientcompact} gives us
\[
\mu_G(AB) \geq \frac{\alpha+\beta}{\alpha} \mu_G(A) + \frac{\alpha+\beta}{\beta}\mu_G(B)
\]
It follows that $\mu_G(AB)> \mu_G(A)+\mu_G(B)+\min\{ \mu_G(A), \mu_G(B)\}$, a contradiction.

Now we have $\alpha+\beta \geq 1$. Hence, Lemma~\ref{lem: Keyestimatequotientcompact} yields
\begin{align*}
    \mu_G(AB)  &\geq    \mu_{G/H}(\pi A_{(\alpha/(\alpha+\beta),\alpha]}) + \mu_{G/H}(\pi B_{(\beta/(\alpha+\beta),\beta]})\\
    &\quad + \frac{\alpha+\beta}{\alpha} \mu_G(A_{(0,\alpha/\gamma]}) + \frac{\alpha+\beta}{\beta}\mu_G(B_{(0,\beta/(\alpha+\beta)])}.
\end{align*}
Choose $\sigma$-compact $A' \subseteq \pi A_{(\alpha/(\alpha+\beta),\alpha]}$ and $B' \subseteq \pi B_{(\beta/(\alpha+\beta),\beta]})$ $\sigma$-compact such that  \[
\mu_{G/H}(A')= \mu_{G/H}(\pi A_{(\alpha/(\alpha+\beta),\alpha]}) \text{ and } \mu_{G/H}(B') = \mu_{G/H}(\pi B_{(\beta/(\alpha+\beta),\beta]}).\] We will verify that $A'$ and $B'$ satisfy the desired conclusion.

Since $\mu_{G/H}(A') \geq (1/ \alpha) \mu_{G}(A_{(\alpha/(\alpha+\beta),\alpha]})$, $\mu_{G/H}(B') \geq (1/ \beta) \mu_{G}(B_{(\beta/(\alpha+\beta),\beta]})$ and $\alpha+\beta>1$, we have
\begin{equation}\label{eq: kernel}
    \mu_G(AB) \geq \frac{1}{\alpha} \mu_G(A)+\frac{1}{\beta}\mu_G(B).  
\end{equation}
From $\mu_G(AB)- \mu_G(A)-\mu_G(B)= \dis_G(A,B)\leq \min\{\mu_G(A),\mu_G(B)\}$, we deduce that $\alpha, \beta \geq 1/2$.

By our assumption $\mu_G(AB) < \mu_G(A)+ \mu_G(B)+\dis_G(A,B)$. Hence, 
\begin{align*}
    \dis_G(A,B)  &\geq    \mu_{G/H}(A') - \mu_{G}( A_{(\alpha/(\alpha+\beta),\alpha]}) + \mu_{G/H}(B') - \mu_{G}( B_{(\beta/(\alpha+\beta),\beta]}) \\
    &\quad + \frac{\beta}{\alpha} \mu_G(A_{(0,\alpha/\gamma]}) + \frac{\alpha}{\beta}\mu_G(B_{(0,\beta/(\alpha+\beta)])}.
\end{align*}
Therefore, $\mu_{G/H}(A') - \mu_{G}( A_{(\alpha/(\alpha+\beta),\alpha]})$ and $(\beta/\alpha) \mu_G(A_{(0,\alpha/\gamma]})$ are at most $ \dis_G(A,B)$. Noting also that $\beta/\alpha \leq 1/2$, we get $\mu_G( A \tri \pi^{-1}(A') \leq 3\dis_G(A,B)$. A similar argument yield $\mu_G( B \tri \pi^{-1}(B') \leq 3\dis_G(A,B)$.

Finally, note that $ \pi^{-1}\left(A'B'\right)$ is equal to $A_{(\alpha/(\alpha+\beta),\alpha]}B_{(\beta/(\alpha+\beta),\beta]}$, which is a subset of $AB$. Combining with $\mu_G(AB) < \mu_G(A)+ \mu_G(B)+\dis_G(A,B)$, we get
\[
\mu_{G/H}(A'B') \leq  \mu_G(A)+\mu_G(B)+\dis_G(A,B)\leq \mu_{G/H}(A')+\mu_{G/H}(B')+ 7\dis_G(A,B),\]
which completes the proof.
\end{proof}

We have the following interesting corollary about $H$.
\begin{corollary}\label{cor: kernel control}
Let $A,B$ be $\sigma$-compact subsets of $G$ with $0<\mu_G(A)\leq\mu_G(B)$, and $\mu_G(AB)<2\mu_G(A)+\mu_G(B)$. Then $H\subseteq AA^{-1}\cap BB^{-1}$. 
\end{corollary}
\begin{proof}
By \eqref{eq: kernel}, we have $\alpha,\beta\geq 1/2$, where $\alpha$ and $\beta$ are defined in Lemma~\ref{lem: Keyestimatequotientcompact}. As $H$ is normal and $G$ is unimodular, this implies that $H\subseteq AA^{-1}\cap BB^{-1}$. 
\end{proof}

The next corollary of the proof of Theorem~\ref{thm: criticality transfer} gives a complementary result for noncompact groups with $\dis_G(A,B)=0$.

\begin{corollary}\label{cor:transfer to Lie noncompact}
Suppose $G$ is noncompact and $\dis_G(A,B)=0$. Then there are $\sigma$-compact $A', B' \subseteq G/H$ such that  $\dis_{G/H}(A',B')=0$, $\mu_G(A\tri \pi^{-1}A')=0$, and $\mu_G(B\tri \pi^{-1}B')=0$.
\end{corollary}

\begin{proof}
Choose an increasing sequence $(A_n)$ of compact subsets of $A$ and an increasing sequence $(B_n)$ of compact subsets of $B$ such that $A=\bigcup^\infty_{n=0} A_n$ and $B= \bigcup_{i=0}^\infty B_n $. Then $\lim_{n \to \infty}\dis_G(A_n, B_n) =0$. For each $n$, $A_n$ and $B_n$ are compact, so $\pi A_n$ and $\pi B_n$ are also compact and has finite measure. Let $A_n'$  and $B'_n$ be defined for $A_n$ and $B_n$ as in the proof of Theorem~\ref{thm: criticality transfer}. Then for $n$ sufficiently large, we have
\[
\mu_G( \pi^{-1}A'_n \tri A_n) < 3\dis_G(A_n, B_n) \text{ and } \mu_G(\pi^{-1}B'_n \tri B_n) < 3\dis_G(A_n,B_n)
\]
 and 
\[
\mu_{G/H}(A_n'B_n')<\mu_{G/H}(A_n')+\mu_{G/H}(B_n')+5\dis_G(A_n, B_n). 
\]
Moreover, we can arrange that the sequences $(A'_n)$ and $(B'_n)$ are increasing.
Take $A' = \bigcup_{n=1}^\infty A_n'$ and $B' = \bigcup_{n=1}^\infty B_n'$. By taking $n\to\infty$, we have
$$  \mu_G(\pi^{-1} A' \tri A) =0 \text{ and } \mu_G(\pi^{-1} B' \tri B) =0.$$
and $\dis_{G/H}(A', B')=0$ as desired.
\end{proof}

\section{Proofs of the main growth gap theorems}\label{sec: 9}

In this section, we prove Theorems~\ref{thm: maingap}(ii) and~\ref{thm: mainassymmetric}. It is clear that Theorem~\ref{thm: maingap}(i) is a direct corollary of Theorem~\ref{thm: mainassymmetric}, except for the quantitative bound for the growth gap, and we will compute it later. The proof of Theorem~\ref{thm: maingap}(ii) uses main results from the first eight sections, and the proof of Theorem~\ref{thm: mainassymmetric} uses results in all previous nine sections. 
We will first prove the following asymmetric generalization of Theorem~\ref{thm: maingap}(ii). 

\begin{theorem}[Asymmetric growth in compact semisimple Lie groups]\label{thm: asymm 1.2}
For every $K\geq 1$, there is a constant $\eta>0$ such that the following hold. Let $G$ be a compact semisimple Lie group, $A,B$ be compact subsets of $G$ of positive measures, and $\mu_G(B)^K\leq \mu_G(A)\leq \mu_G(B)^{1/K}$. Then
\[
\mu(AB)\geq \min\{1, \mu_G(A)+\mu_G(B)+\eta\min\{\mu_G(A),\mu_G(B)\}|1-\mu_G(A)-\mu_G(B)|\}.
\]
\end{theorem}
\begin{proof}
Let $c_\TT>0$ be the real number fixed in Fact~\ref{fact: new inverse theorem torus}, and $\nu$ be in Theorem~\ref{thm: fibers of same length 1newnew}. Suppose $G$ is a compact semisimple Lie group and $A,B\subseteq G$ are compact subsets with $\mu_G(B)\geq \mu_G(A)>0$, and
\[
\mu_G(AB)\leq \mu_G(A)+\mu_G(B)+\eta c\mu_G(A),
\]
where $c=1-\mu_G(A)-\mu_G(B)$ and $\eta$ to be chosen later. 
By Proposition~\ref{prop: Torictransversal}, there are compact sets $A',B'\subseteq G$ with $\mu_G(A')=\mu_G(B')=\min\{ 1/(3^{3^{20}/c_\TT}), \mu_G(A)\}$, and a closed one-dimensional torus $H\leq G$, such that
\[
\mu_G(A'B')\leq \Big(2+\frac{\eta 3^{\frac{3^{20}}{c_\TT}}}{c_\TT^2}\Big)\mu_G(A')=:(2+\varepsilon)\mu_G(A'),
\]
and $A,B$ are not $c_\TT$-Kakeya with respect to $H$. By choosing $\eta$ small enough, we may assume our $\varepsilon<10^{-12}$.  

Let
$
d_{A'}(g_1,g_2)=\mu_G(A')-\mu_G(g_1A'\cap g_2A').
$
  By Proposition~\ref{prop: construct pseudo-metric}, $d_{A'}$ is a pseudometric. Since $\mu_G(A'B')<(2+\varepsilon)\mu_G(A')$, Proposition~\ref{prop: almost linear metric from local} shows that $d_{A'}$ is a $60\nu\mu_G(A')$-linear pseudometric, and it is $180\nu\mu_G(A')$-path-monotone. By Proposition~\ref{prop: localmonotoneimplyglobalmonotone}, $d_{A'}$ is globally $1620\nu\mu_G(A')$-monotone. 
  Let $\gamma=1620\nu\mu_G(A').$ Then $d_{A'}$ is $\gamma$-monotone $\gamma$-linear, and of radius $\rho=\mu_G(A')/2$. As $\nu\leq 10^{-10}$, we have that $10^6\gamma<\rho$. Thus Theorem~\ref{thm: homfrommeasurecompact} implies that there is a continuous surjective group homomorphism mapping $G$ to $\TT$, and this contradicts the fact that $G$ is semisimple. 
\end{proof}

Let us remark that the absolute constant $c$ appeared in Theorem~\ref{thm: maingap} can be determined as follows. The growth gap will be small after we apply Proposition~\ref{prop: red to small sets} to make our sets small, in order to obtain a geometric description from Proposition~\ref{prop: Torictransversal}. Thus by assuming our sets are small at the beginning, we can apply Theorem~\ref{thm: kakeya new} directly without losing the growth gap quantitatively. Hence $c\leq 1/(3^{3^{20}/c_\TT})$. Note that the current bound for $c_\TT$ is roughly $10^{-1550}$ by~\cite{circle19}. Hence $c<10^{-10^{1560}}$ suffices. 

We next move to Theorem~\ref{thm: mainassymmetric}, which considers measure growth in connected compact (not necessarily Lie) groups. For the given $G$, there might be no continuous surjective group homomorphism to either $\TT$ or $\RR$ (e.g. $G = \mathrm{SO}_3(\RR))$. However, the famous theorem below by Gleason~\cite{Gleason} and Yamabe~\cite{Yamabe} allows us to naturally obtain continuous and surjective group homomorphism to a Lie group. 
 The connectedness of $H$ is not often stated as part of the result, but can be arranged by replacing $H$ with its identity component.

\begin{fact}[Gleason--Yamabe Theorem]\label{fact:Yamabe}
For any connected locally compact group $G$
and any neighborhood $U$ of the identity in $G$, there is a connected compact normal subgroup $H\subseteq U$ of $G$
such that $G/H$ is a  Lie group.
\end{fact}

 With some further effort, we can also arrange that $\mu_{G/H}(\pi A)+\mu_{G/H}(\pi B)< \mu_{G/H}(G/H)$ as necessary to apply Theorem~\ref{thm: criticality transfer}. 

\begin{lemma}[Small projection from small expansions]\label{lem:gleason2}
If $\mu_G(AB) \leq K \mu^{1/2}_G(A)\mu^{1/2}_G(B)$ and then there is a connected compact subgroup $H$ of $G$ such that  $G/H$ is a Lie group and, with $\pi: G \to G/H$ the quotient map, $\pi A$ and $\pi B$ have $\mu_{G/H}$-measure at most 
 $2^{15}K^{34}\mu^{1/2}_G(A)\mu^{1/2}_G(B)$. 
\end{lemma}
\begin{proof}
By Fact~\ref{fact: tao approximate groups}, there is an open $64K^{12}$-approximate group $S$, with  
\[ 
\mu_G(S)\leq \frac{\mu_G^{1/2}(A)\mu_G^{1/2}(B)}{4K^2}
\]
such that $A$ can be covered by $ 33K^{12}$ right translates of $S$, and $B$ can be covered by $ 33K^{12}$ left translates of $S$. By Fact~\ref{fact:Yamabe}, there is a closed connected normal subgroup $H$ in $S^2$, such that $G/H$ is a Lie group. Let $\pi$ be the quotient map. Since $H \subseteq S^{2}$, we have
\[
\mu_{G/H}(\pi(S)) = \mu_G(SH)\leq\mu_G(S^3) \leq 2^{10}K^{22}\mu_G^{1/2}(A)\mu_G^{1/2}(B).
\]
Note that $\pi(A)$ can
be covered by $33K^{12}$ right translates of $\pi(S)$, and $\pi(B)$ can
be covered by $33K^{12}$ left translates of $\pi(S)$. 
 Hence, we get the desired conclusion.
\end{proof}

We will need the following classic inequality from probability. 

\begin{fact}[Bhatia--Davis inequality] \label{fact: Bhatia-Davis}
Suppose $(X, \mathscr{A}, \mu)$ is a measure space, $\alpha$ and  $\beta$ are constants, and $f: X \to \RR^{>0}$ is a measurable function with \[\alpha\leq \inf_{x\in X} f(x) < \sup_{x \in X}f(x)\leq \beta\]
Then $(\mathbb E_x f^2(x)) - (\mathbb E_x f(x))^2 \leq (\beta- \mathbb E_x f(x))(\mathbb E_x f(x) -\alpha).$
\end{fact}

  The following lemma will help us to translate the set randomly along a given normal subgroup. 

\begin{lemma}\label{lem: trans along H}
Suppose $K$, $\alpha$, and $\beta$ are constant with $K>1$, $0\leq\alpha\leq\beta\leq1$, $\mu_G(AH) = K\mu_G(A)$, and 
\[
\alpha\leq\inf_{g\in A} \mu_H(A\cap gH)\leq \sup_{g\in A} \mu_H(A\cap gH)\leq \beta.
\]
Then for every number $\gamma \geq (\alpha+\beta-K\alpha\beta) \mu_G(A)$, there is $h\in H$ with $\mu_G(A\cap Ah)=\gamma$. 
\end{lemma}
\begin{proof}
Let $\mu_G$, $\mu_H$ be normalized Haar measures of $G$ and $H$. Choose $h$ from $H$ uniformly at random. Note that
\begin{align*}
\mathbb E_{h\in H}\mu_G(A\cap Ah)&= \int_H \mu_G(A\cap Ah)\d\mu_H(h)\\
&=\int_H\int_G\mathbbm{1}_{A}(g)\mathbbm{1}_{A}(gh)\d\mu_G(g)\d\mu_H(h).
\end{align*}
Using the quotient integral formula (Fact~\ref{fact: QuotientIF}), the above equality is
\begin{align*}
\int_G\mathbbm{1}_{A}(g)\mu_H(A\cap gH) \d\mu_G(g)&=\int_{G/H} \mu_H^2(A\cap gH)\d\mu_{G/H}(gH)\\
                                                &= \mathbb{E}_{gH \in G/H} \big(\mu_H^2(A\cap gH)\big)\\
                                                &= \mu_{G/H}( \pi A)\mathbb{E}_{gH \in \pi A} \big(\mu_H^2(A\cap gH)\big)
\end{align*}
Note that $\mu_{G/H}(\pi A) = \mu_G(AH)= K\mu_G(A)$, and 
$$\mathbb{E}_{gH \in \pi A} \big(\mu_H(A\cap gH)\big)= 1/K.$$ Hence, applying the Bhatia–Davis inequality (Fact~\ref{fact: Bhatia-Davis}), we get
\[
\mathbb E_{h\in H}\mu_G(A\cap Ah)\leq (\alpha+\beta-K\alpha\beta)\mu_G(A).
\]
The desired conclusion follows from the continuity of $H \to \RR, h \mapsto \mu_G(A\cap Ah)$.
\end{proof}

The next lemma allows us to control the size of the projection once we have some control on the kernel of the homomorphism mapping to $\TT$. 

\begin{lemma}\label{lem: making projection small}
Suppose $G$ is a connected compact group,  $\chi: G \to \TT$ is a continuous and surjective group homomorphism with connected kernel, $A, B \subseteq G$ are nonempty and $\sigma$-compact with 
$$ \dis_G(A,B)< \mu_G(A)+\mu_G(B) < \max\{\mu_G(A),\mu_G(B)\}<1/250.   $$
Suppose for every $g\in \ker(\chi)$ with $\mu_G(A\setminus gA)<\mu_G(A)/36$, we further have 
$$\mu_G(A\setminus gA)<\mu_G(A)/72.$$ Then $\mu_\TT(\chi(A))+\mu_\TT(\chi(B))<1/5$. 
\end{lemma}
\begin{proof}
Set $H=\ker(\chi)$. We first show that   $\sup_g\mu_H(A\cap gH)> 1/2$. Suppose to the contrary that $\sup_g\mu_H(A\cap gH)\leq 1/2$. Note that
\[
\frac{\mu_G(AH)}{\mu_G(A)}>\frac{1}{\sup_g \mu_H(A\cap gH)}. 
\]
Hence by Lemma~\ref{lem: trans along H}, for every $\ell>1/2$, there is $h\in H$ such that $\mu_G(A\cap hA)=\ell \mu_G(A)$, and in particular, there is $h\in H$ with
\[
\frac{35\mu_G(A)}{36}<\mu_G(A\cap hA)<\frac{71\mu_G(A)}{72},
\]
which contradicts the assumption. 

Now apply Lemma~\ref{lem: A=AH toric exp prepare}, we get
$
\mu_{\TT}(\chi A) +\mu_{\TT}(\chi B) \leq 50 ( \mu_G(A)+\mu_G(B) )=1/5
$
as desired.
\end{proof}

Next, we prove Theorem~\ref{thm: mainassymmetric}. 
The proof is essentially a combination of Lemma~\ref{lem:gleason2}, Lemma~\ref{lem: making projection small}, Theorem~\ref{thm: criticality transfer}, and Theorem~\ref{thm: maingap}(ii). In the proof, we will assume our sets $A,B$ having small measures, as otherwise we can always apply Proposition~\ref{prop: red to small sets} in a same fashion as in the proof of Theorem~\ref{thm: maingap}(ii) to make sets small. Under this smallness assumption, we can make our measure growth quantitatively better. In fact we are going to prove the following theorem.
\begin{theorem}
Let $K\geq 1$ and $c = 10^{-10^{1560}}$. Suppose $G$ is a connected compact group, and  $0< \mu(A), \mu(B) \leq c$ such that $\mu_G(B)^K\leq \mu_G(A)\leq \mu_G(B)^{1/K}$ and
\[
\mu(AB)<\mu(A) + \mu(B) +\frac{10^{-13}}{K} \min\{\mu(A), \mu(B)\}.
\]
Then there is a continuous surjective group homomorphism $\chi: G \to \TT$ with 
$\ker(\chi) \subseteq AA^{-1} \cap BB^{-1}.$
\end{theorem}

\begin{proof}
Let $\eta=10^{-12}$, $c>0$ to be chosen later, and suppose $A,B$ are compact subsets of $G$ with $0<\mu_G(A)\leq\mu_G(B)\leq c$, $\mu_G(A)\geq\mu_G(B)^K$, and 
\[
\mu_G(AB)\leq \mu_G(A)+\mu_G(B)+\eta\mu_G(A).
\]
By Lemma~\ref{lem:gleason2}, by choosing $c$ sufficiently small, we conclude that there is a quotient map $\pi:G\to G/H$ with kernel $H$ such that $G/H$ is a connected Lie group, and $\mu_{G/H}(\pi A)+\mu_{G/H}(\pi B)<1$. By Corollary~\ref{cor: kernel control}, $H\subseteq AA^{-1}\cap BB^{-1}$.

We now apply Theorem~\ref{thm: criticality transfer}, there are $A',B'\subseteq G/H$ such that 
\[
\mu_{G/H}(A'B')\leq \mu_{G/H}(A')+\mu_{G/H}(B')+7\eta\mu_{G/H}(A).
\]
and $\max\{\mu_G(A\tri \pi^{-1}A'), \mu_G(B\tri \pi^{-1}B')\}<3\eta\mu_G(A)$. By Corollary~\ref{cor: small sets may not equal} and Theorem~\ref{thm: kakeya new}, by choosing $c$ sufficiently small, there are sets $A''$ and $B''$ with  $A'',B''\subseteq G/H$, $A''\subseteq A'$, $B''\subseteq B'$, with $\mu_{G/H}(A'')=\mu_{G/H}(B'')\leq  1/(3^{3^{20}/c_\TT})$, and a closed one-dimensional torus $T\leq G/H$, such that
\[
\mu_{G/H}(A''B'')\leq (2+10K\eta)\mu_{G/H}(A''),
\]
and $A'',B''$ are not $c_\TT$-Kakeya with respect to $T$.
The rest of the proof are the same as the proof in Theorem~\ref{thm: asymm 1.2}, we apply Theorem~\ref{thm: homfrommeasurecompact} and conclude that there is a continuous surjective group homomorphism $\iota :G/H\to\TT$. By replacing $\iota$ with the quotient map $G/H \to (G/H)/(\ker(\iota)_0)$ if necessary, where $(\ker(\iota)_0)$ is the identity component of $\ker(\iota)$,  we can arrange that $\iota$ has a connected kernel. Note that also by Theorem~\ref{thm: homfrommeasurecompact}, for every $g\in\ker\iota\cap N(\lambda)$ with $\lambda=\rho/36$, we have $\mu_{G/H}(A''\setminus gA'')<\mu_G(A'')/72$. Thus Lemma~\ref{lem: making projection small} shows that $\mu_\TT(\iota A'')+\mu_\TT(\iota B'')<1/5$, and thus Corollary~\ref{cor: kernel control} implies that $\ker \iota \subseteq A''A''^{-1}\cap B''\cap B''^{-1}$. 

Finally, let $\chi=\iota\circ\pi$. Thus $\chi:G\to\TT$ is a continuous surjective group homomorphism, and $\ker\chi\subseteq AA^{-1}\cap BB^{-1}$. 
\end{proof}

\section{The Kemperman Inverse Problem}\label{sec: 6.3}

In this section, we study the structures of sets $A$ and $B$ if $\mu_G(AB)$ is small.  The following useful lemma is a corollary of Theorem~\ref{thm: criticality transfer}, which will be used at various points in the later proofs. 
 
\begin{lemma}\label{lem:fromsmalltobig}
Suppose $G$ is compact, $\chi: G\to\TT$ is a continuous surjective group homomorphism, $J\subseteq \TT$ is a compact interval.
Suppose we have
\begin{enumerate}[\rm (i)]
\item $\dis_G(A,B)<\delta\min\{\mu_G(A),\mu_G(B)\}$ with $\delta<c_\TT$ where $c_\TT$ is from Fact~\ref{fact: new inverse theorem torus};
    \item $\chi B\subseteq J$ and $\mu_\TT(J\setminus \chi B)\leq \dis_G(A,B)$;
    \item $\mu_G(A),\mu_G(B)\leq 1/5$ and there is a constant $K$ such that $1/K<\mu_G(A)/\mu_G(B)<K$.
\end{enumerate}
 Then there is a compact interval $I$ in $\TT$ such that $\chi A\subseteq I$ and $$\mu_\TT(I\setminus A)\leq 10\dis_G(A,B).$$
\end{lemma}
\begin{proof}
Let $H=\ker(\chi)$.
Note that $\chi$ is an open map, and hence by Fact~\ref{fact: first iso thm} we have $G/H\cong \TT$. By Lemma~\ref{lem:gleason2} and Theorem~\ref{thm: criticality transfer}, there are sets $A'$ and $B'$ in $\TT$ such that $\dis_{\TT}(A',B')\leq 7\dis_G(A,B)$, and 
\[
\mu_G(A\tri \chi^{-1}(A'))=3\dis_G(A,B),\quad \mu_G(B\tri \chi^{-1}(B'))=3\dis_G(A,B).
\]
Then $\mu_\TT(J\setminus B')\leq 4\dis_G(A,B)$. Using Lemma~\ref{lem: stability of character in T}, there is an interval $I\subseteq \TT$, so that $A'\subseteq I$ and $\mu_\TT(I\setminus A')<10\dis_G(A,B)$. 
\end{proof}

In this paper, when we discuss the structures of sets, we sometimes say that a set is contained in (the preimage of) some interval, and sometimes we also say the symmetric difference of our set and the given interval is small. The next lemma shows that these two descriptions are the same.  To handle arbitrary sets, we use inner measure here.
\begin{lemma}\label{lem: from sym dif to subset}
Suppose $G$ is compact, $\widetilde{\mu}_G$ is the inner measure associated to $\mu_G$, and $A,B \subseteq G$ has $\widetilde{\dis}_G(A,B)<\varepsilon$ with 
\[
\widetilde{\dis}_G(A,B) = \widetilde{\mu}_G(A,B)-\widetilde{\mu}_G(A)- \widetilde{\mu}_G(B). 
\]
Assume further that $\widetilde{A}\subseteq A$ and $\widetilde{B}\subseteq B$ are $\sigma$-compact with $\mu_G(\widetilde{A}) =\widetilde{\mu}_G(A)$ and $\mu_G(\widetilde{B}) =\widetilde{\mu}_G(B)$, $\chi:G\to \TT$ is a continuous surjective group homomorphism, and $I,J$ compact intervals in $\TT$, with $\mu_\TT(I)=\mu_G(A),\mu_\TT(J)=\mu_G(B)$, and
\[
\mu_G(\widetilde{A}\tri\chi^{-1}(I))<\varepsilon,\quad \mu_G(\widetilde{B}\tri \chi^{-1}(J))<\varepsilon.
\]
Then there are intervals $I',J'\subseteq\TT$, such that $A\subseteq\chi^{-1}(I')$, $B\subseteq\chi^{-1}(J')$, and 
\[
\mu_\TT(I')-\widetilde{\mu}_G(A)<10\varepsilon,\quad\mu_\TT(J')-\widetilde{\mu}_G(B)<10\varepsilon. 
\]
\end{lemma}
\begin{proof}

We will show that for all $g \in A\setminus \chi^{-1}(I)$, the distance between $\chi(g)$ and $I$   in $\TT$ is at most $5\varepsilon$, and  for all $g' \in B\setminus \chi^{-1}(J)$, the distance between $g'$ and $J$ is at most $5\varepsilon$. This implies there are intervals $I',J'$ in $\TT$ such that $A\subseteq \chi^{-1}(I')$ and $B\subseteq \chi^{-1}(J')$, and
\[
\mu_G(\chi^{-1}(I')\setminus A)<10\varepsilon\, \quad\mu_\TT(\chi^{-1}(J')\setminus B)<10\varepsilon,
\]
as desired. Observe that $\widetilde{A}\widetilde{B}$ is a $\sigma$-compact subset of $AB$, and so $\dis_G(\widetilde{A}, \widetilde{B})<\varepsilon$.

By symmetry, it suffices to show the statement for $g \in A\setminus \chi^{-1}(I)$. Suppose to the contrary that $g$ is in $A\setminus \chi^{-1}(I)$, and the distance between $\chi(g)$ and $I$ in $\TT$ is strictly greater than $5\varepsilon$. By replacing $\widetilde{A}$ with $\widetilde{A}\cup\{g\}$ if necessary, we can assume $g \in \widetilde{A}$.
We then have
\[
\mu_\TT(\chi(g)\chi(\widetilde{B})\setminus I\chi(\widetilde{B}))\geq 5\varepsilon-\mu_\TT(J\setminus \chi(\widetilde{B}))\geq 4\varepsilon. 
\]
and this implies that $\mu_G(g\widetilde{B}\setminus \chi^{-1}(I)\chi^{-1}(J))\geq3\varepsilon$. Therefore,
\begin{align*}
\mu_G(\widetilde{A}\widetilde{B})&\geq \mu_G\big((\chi^{-1}(I)\cap \widetilde{A})(\chi^{-1}(J)\cap \widetilde{B})\big)+\mu_G(g\widetilde{B}\setminus \chi^{-1}(I)\chi^{-1}(J))\\
&\geq \mu_G(\widetilde{A})+\mu_G(\widetilde{B})-2\varepsilon+3\varepsilon,
\end{align*}
and this contradicts the fact that $\dis_G(\widetilde{A},\widetilde{B})<\varepsilon$.  
\end{proof}

Lemma~\ref{lem:fromsmalltobig}, together with Theorem~\ref{thm: criticality transfer}, will be enough to derive a different proof of a theorem by Tao~\cite{T18}, with a sharp exponent bound up to a constant factor. The same result was also obtained by Christ and Iliopoulou~\cite{ChristIliopoulou} recently, via a different approach. It worth noting that the proof below is short, as it does not use results proved in Sections 5,6,7, and 8. 
\begin{theorem}[Theorem~\ref{thm:mainapproximate} for compact abelian groups]\label{thm: abelian case}
   Let $G$ be a connected compact abelian group, and $A,B$ be compact subsets of $G$ with positive measure. Set $$\lambda =\min\{\mu_G(A),\mu_G(B),1-\mu_G(A)-\mu_G(B)\}.$$
  Given $0<\varepsilon<1$, there is a constant $K=K(\lambda)$ does not depend on $G$, such that if $\delta<K\varepsilon$ and
  \[
  \mu_G(A+B)<\mu_G(A)+\mu_G(B)+\delta\min\{\mu_G(A),\mu_G(B)\}.
  \]
 Then there is a surjective continuous group homomorphism $\chi: G \to \TT$ together with two compact intervals $I,J\in \mathbb T$ with
 \[
 \mu_\TT(I)-\mu_G(A)<\varepsilon\mu_G(A),\quad \mu_\TT(J)-\mu_G(B)<\varepsilon\mu_G(B),
 \]
 and $A\subseteq \chi^{-1}(I)$, $B\subseteq\chi^{-1}(J)$. 
\end{theorem}
\begin{proof}
We first assume that $\dis_G(A,B)$ is sufficiently small, and we will compute the bound on $\dis_G(A,B)$  later. As $G$ is abelian, there is a quotient map $\pi: G\to\TT^d$, and $A',B'\subseteq \TT^d$, such that  \[ \mu_G(A \tri\pi^{-1}A')< 3 \dis_G(A,B) \text{ and } \mu_G(B \tri\pi^{-1}B')< 3 \dis_G(A,B)\]
and $\dis_{G/H}(A', B') < 7\dis_{G}(A, B).$ Let $c=c_\TT$ be as in Fact~\ref{fact: new inverse theorem torus}, and by Proposition~\ref{prop: red to small sets}, there is a constant $L$ depending only on $\lambda$ and $c$, and sets $A'',B''\subseteq\TT^d$ with $\mu_{\TT^d}(A'')=\mu_{\TT^d}(B'')=c$ such that
\begin{equation}\label{eq: other small sets}
\max\{\dis_{\TT^d}(A'',B'), \dis_{\TT^d}(A'',B''), \dis_{\TT^d}(A',B'')\}<L\dis_G(A,B). 
\end{equation}
By Fact~\ref{fact: new inverse theorem torus}, there are intervals $I'',J''\subseteq\TT$ and a continuous surjective group homomorphism $\rho:\TT^d\to\TT$ such that $A''\subseteq \rho^{-1}(I'')$, $B''\subseteq \rho^{-1}(J'')$, and
\[
\mu_{\TT^d}(\rho^{-1}(I'')\setminus A'')<L\dis_G(A,B)\quad\text{and}\quad \mu_{\TT^d}( \rho^{-1}(J'')\setminus B'' )<L\dis_G(A,B). 
\]
By Lemma~\ref{lem:fromsmalltobig} and \eqref{eq: other small sets}, there are intervals $I',J'\subseteq \TT$ with
\[
\mu_{\TT^d}( \rho^{-1}(I')\setminus A')<10L\dis_G(A,B)\text{ and } \mu_{\TT^d}( \rho^{-1}(J')\setminus B')<10L\dis_G(A,B). 
\]
Let $\chi=\pi\circ\rho$. Hence, we have
\[
\mu_{G}(A\tri \chi^{-1}(I'))<(3+10L)\dis_G(A,B)\text{ and } \mu_{G}(B\tri \chi^{-1}(J'))<(3+10L)\dis_G(A,B). 
\]
Using Lemma~\ref{lem: from sym dif to subset}, there are intervals $I,J\subseteq\TT$, such that $A\subseteq\chi^{-1}(I)$, $B\subseteq\chi^{-1}(J)$, and
\begin{align*}
    &\mu_\TT(I)-\mu_G(A)<(30+100L)\dis_G(A,B),\\
    &\mu_\TT(J)-\mu_G(B)<(30+100L)\dis_G(A,B).
\end{align*}
Now, we fix 
\[
K:=\min\Big\{\frac{1}{30+100L}, \frac{c}{L}\Big\},
\]
and $\delta<K\varepsilon$, where $\dis_G(A,B)=\delta\min\{\mu_G(A),\mu_G(B)\}$. Clearly, we will have 
\begin{align*}
    &\mu_\TT(I)-\mu_G(A)<\varepsilon\min\{\mu_G(A),\mu_G(B)\},\\
    &\mu_\TT(J)-\mu_G(B)<\varepsilon\min\{\mu_G(A),\mu_G(B)\}.
\end{align*}
Note that in the above argument, we apply Fact~\ref{fact: new inverse theorem torus} on $A'',B''$, and this would require that $L\dis_G(A,B)<c$. This is fine as by the way we choose $K$, we already have
\[
L\dis_G(A,B)=L\delta\min\{\mu_G(A),\mu_G(B)\}<c,
\]
as desired. 
\end{proof}

Theorem~\ref{thm:mainapproximate} follows easily from Theorem~\ref{thm: mainassymmetric} and the following proposition. 

\begin{proposition}[Toric domination from a given homomorphism]\label{prop: struc on AB}
Suppose $A, B$ have $\dis_G(A,B)<\min\{\mu_G(A), \mu_G(B)\}$, and  $\chi:G\to\TT$ is a continuous surjective group homomorphism such that $\mu_\TT(\chi(A))+\mu_\TT(\chi(B))<1/5$. Then there is a continuous and surjective group homomorphism $\rho: G \to \TT$, a constant $K_0$ only depending on $\min\{\mu_G(A),\mu_G(B)\}$, and compact intervals $I, J \subseteq \TT$ with $A\subseteq \rho^{-1}(I)$ and $B\subseteq \rho^{-1}(J)$, such that 
\[
\mu_G(\rho^{-1}(I)\setminus A)<K_0\dis_G(A,B),\quad \text{and}\quad\mu_G(\rho^{-1}(J)\setminus J)<K_0\dis_G(A,B).
\]
\end{proposition}
\begin{proof}
By Theorem~\ref{thm: criticality transfer}, there are $A',B'\subseteq \TT$, such that 
\begin{equation}\label{eq: struc for AB R case}
\mu_G(A\tri \chi^{-1}(A'))<3\dis_G(A,B) \text{ and } \mu_G(B\tri \chi^{-1}(B'))<3\dis_G(A,B),
\end{equation}
and $\dis_\TT(A',B')<7\dis_G(A,B)$. By Theorem~\ref{thm: abelian case}, there are continuous surjective group homomorphism $\eta: \TT \to \TT$, a constant $L$ depending only on $\min\{\mu_G(A),\mu_G(B)\}$, and compact intervals $I, J \subseteq \TT$ such that $A'\subseteq\eta^{-1}(I)$,  $B'\subseteq\eta^{-1}(J)$, and 
\begin{equation}\label{eq: struc for AB R case2}
\mu_\TT(\eta^{-1}(I)\setminus )<L\dis_G(A,B) \text{ and } \mu_\TT( \eta^{-1}(J)\setminus )<L\dis_G(A,B).
\end{equation}
Set $\rho = \eta \circ \chi$. The conclusion follows from \eqref{eq: struc for AB R case}, \eqref{eq: struc for AB R case2}, and Lemma~\ref{lem: from sym dif to subset} with $K_0=10L+30$.
\end{proof}

Finally, let us discuses what happen when the equality holds in
\[
\mu_G(AB)=\min\{\mu_G(A)+\mu_G(B),1\}
\]
for some connected compact group $G$. The same proof for Theorem~\ref{thm:mainapproximate} works when $\delta=0$ and $\mu_G(A)+\mu_G(B)<1$, and we will obtain $\varepsilon=0$. There are some other easier cases, and the full classification theorem is the following.

\begin{theorem}\label{thm:mainequal}
Let $G$ be a connected compact group, and let $A,B$ be nonempty compact subsets of $G$.  If
\[
\mu_G(AB)=\min\{\mu_G(A)+\mu_G(B), 1\}. 
\]
 then we have the following:
\begin{enumerate}[\rm (i)]
    \item $\mu_G(A) +\mu_G(B)=0$ implies $\mu_G(AB) =0$; 
    \item $\mu_G(A) +\mu_G(B)\geq 1$ implies   $AB=G$;
    \item   $\mu_G(A)=0$ and $0<\mu_G(B)< 1$ imply there is compact  $H \leq G$ and compact $B_1 \subseteq B$ such that with $B_2= B\setminus B_1$, we have $HB_1=B_1$,  $\mu_G(AB_2)=0$, and $A \subseteq gH$ for some $g \in G$;
    \item  $\mu_G(B)=0$ and $0<\mu_G(A)< 1$ imply there is compact  $H \leq G$ and compact $A_1 \subseteq A$ such that, with $A_2=A \setminus A_1$, we have $A_1H=A_1$, $\mu_G(A_2B)=0$, and $B \subseteq Hg$ for some $g \in G$; 
    \item $0<\min\{\mu_G(A),\mu_G(B),1-\mu_G(A)-\mu_G(B)\}$ imply that there is a surjective continuous group homomorphism $\chi: G \to \TT$ and compact intervals $I$ and $J$ in $\TT$ with $I+J\neq \TT$ such that $A = \chi^{-1}(I)$ and $B = \chi^{-1}(J)$
\end{enumerate}
Moreover, $\mu_G(AB)=\min\{\mu_G(A)+\mu_G(B), 1\}$ holds if and only if we are in exactly one of the implied scenarios in (i-v).
\end{theorem}

The proof for the scenario (v) follows from Theorem~\ref{thm:mainapproximate}, and (ii) follows from Lemma~\ref{lem: when a+b>G}. Now we clarify the situation in (iii) of Theorem~\ref{thm:mainequal}, situation (iv) can be proven in the same way.

\begin{proposition}

Suppose $A, B\subseteq G$ are nonempty, compact, and with $\mu_G(A)=0$, $0<\mu_G(B)< 1$, and $$\mu_G(AB)=\min\{\mu_G(A)+\mu_G(B), 1\}.$$
Then there is a compact subgroup $H$ of $G$ such that $A \subseteq gH$ for some $g\in G$, and $B= B_1 \cup B_2$ with $HB_1=B_1$ and $\mu_G(AB_2)=0$. 
\end{proposition}

\begin{proof}
Without loss of generality, we can assume that $A$ and $B$ both contain $\text{id}_G$. Let $H =\Stab^0_G(B)$, let $B_1$ be the set of $b\in B$ such that whenever $U$ is an open neighborhood of $b$, we have $\mu_G(U\cap B) >0$, and let $B_2= B \setminus B_1$. We will now verify that $H$, $B_1$, and $B_2$ are as desired.

We make a number of immediate observations.  As $\mu_G(AB)= \mu_G(B)$ and $\text{id}_G$ is in $A$, we must have $A \subseteq H$. Note that $B_2$ consists of $b\in B$ such that there is open neighborhood $U$ of $B$ with $\mu_G(U \cap B)=0$. So $B_2$ is open in $B$. Moreover, if $K$ is a compact subset of $B_2$, then $K$ has a finite cover $(U)_{i=1}^n$ such that $\mu_G(U_i \cap B)=0$, which implies $\mu_G(K)=0$. It follows from inner regularity of $\mu_G$ (Fact~\ref{fact: Haarmeasurenew}(iv)), that $\mu_G(B_2)=0$. Hence, $B_1$ is a closed subset of $B$ with $\mu_G(B_1)=\mu_G(B)$. Since $B$ is compact, $B_1$ is also compact.
As $H$ is a closed subset of $G$, the compactness of $H$ follows immediately if we can show that $HB_1=HB$

It remains to verify that $HB_1=B_1$.
 As $\mu_G$ is both left and right translation invariant, we also have that for all $g\in G$, if $U$ is an open neighborhood of $gb \in gB_1$, then $\mu_G(U \cap gB_1)>0.$ 
Suppose $hb$ is in $HB_1 \setminus B_1$ with $h\in H$. Set $U = G \setminus B_1$. Then $U$ is an open neighborhood of $hb$. From the earlier discussion, we then have $\mu_G(U \cap hB_1) >0$. This implies that $\mu_G( hB \setminus B) >0$ contradicting the fact that $h \in H= \Stab^0_G$.
\end{proof}

The ``if and only if'' statement of Theorem~\ref{thm:mainequal} can be verified easily:

\begin{proposition} \label{prop:backwardof1.1}
Suppose $A, B\subseteq G$ are nonempty and compact and one of the situation listed in Theorem~\ref{thm:mainequal}
 holds, then $\mu_G(AB)=\min\{\mu_G(A)+\mu_G(B), 1\}.$
\end{proposition}
\begin{proof}
We will only consider situation (v) because (i-iv) are immediate. Suppose we are in situation (v) of Theorem~\ref{thm:mainequal}.  As $\chi$ is a group homomorphism, we have $AB =\chi^{-1}(I+J)$. Note that by quotient integral formula, we have $\mu_G(A) = \mu_\TT(I)$, $\mu_G(B) = \mu_\TT(J)$, $\mu_G(AB) = \mu_\TT(I+J)$. The desired conclusion follows from the easy that $\mu_\TT(I+J) = \mu_\TT(I)+\mu_\TT(J)$.
\end{proof}




\section{Some remarks on noncompact groups}

In general when $G$ is noncompact, with $\mu_G$ a left Haar measure on $G$, we may have $\mu_G(AB)<\mu_G(A)$. Thus in this section, we only consider connected unimodular noncompact locally compact groups, in which Kemperman's inequality holds:
\begin{equation}\label{eq: noncompact}
\mu_G(AB)\geq \mu_G(A)+\mu_G(B). 
\end{equation}
It is natural to ask when does the equality holds in the above inequality. Using the machinery developed in this paper, the following theorem can be proved:
\begin{theorem}\label{thm: kemperman noncompact}
 Let $G$ be a connected unimodular noncompact locally compact group, and $A,B$ be compact subsets of $G$ of positive measures. Suppose
  \[
  \mu_G(AB)= \mu_G(A)+\mu_G(B).
  \]
 Then there is a surjective continuous group homomorphism $\chi: G \to \RR$ together with two compact intervals $I,J\in \RR$ with
 \[
 A=\chi^{-1}(I),\qquad B=\chi^{-1}(J). 
 \]
\end{theorem}
Our proof for Theorem~\ref{thm: kemperman noncompact} is actually simpler than the compact case, and let us sketch a proof below. For a detailed proof, we refer the readers to the earlier version of the paper (which is available at \texttt{arXiv:2006.01824v3}).

Let us review what we did in the earlier sections. When $G$ is compact, we may possibly have $\mu_G(AB)<\mu_G(A)+\mu_G(B)$, but this will never happen when $G$ is noncompact. Thus results in Section 4 and Section 5 are no longer needed. In Section 6, we obtained a almost linear path monotone pseudometric. The same proofs still work for noncompact groups (by replacing $\TT$ by $\RR$), and since we only consider the case when equality happens in~\eqref{eq: noncompact}, we will obtain an exact linear pseudometric here. Moreover, we do not need to study path monotonicity in this case, as we will see below that exact linear pseudometric would imply global  monotonicity. This can be seen as an application of ``no small subgroup'' property of Lie groups, see Proposition~\ref{prop: automaticmonotone}. In Sections 7 and 8, we derived global monotonicity from path monotonicity, and constructed a group homomorphism from our monotone almost linear pseudometric. In the proof of Theorem~\ref{thm: kemperman noncompact}, as we already got an exact linear pseudometric, this will immediately give us a group homomorphism, see Proposition~\ref{prop: strong linear} below. Finally, the structures of sets come from the inverse Brunn--Minkowski inequality for $\RR$. 

We will first show that linear pseudometric implies global  monotonicity.  Let $d$ be a left-invariant pseudometric on $G$. Recall that the radius $\rho$ of $d$ is defined to be $\sup\{\|g\|_d : g \in G\}$; this is also $\sup\{d(g_1, g_2) : g_1, g_2 \in G\}$ by left invariance. We say that  $d$ is {\bf locally linear} if it satisfies the following properties:
\begin{enumerate}
    \item $d$ is continuous and left-invariant;
    \item for all $g_1$, $g_2$, and $g_3$ with $d(g_1, g_2)+d(g_2,g_3) < \rho$, we have either
    \begin{equation}\label{eq: key of linear pseudometric}
        d(g_1,g_3) = d(g_1, g_2) + d(g_2,g_3), \text{ or } d(g_1,g_3) = |d(g_1, g_2) - d(g_2,g_3)|.
    \end{equation}
\end{enumerate}
A pseudometric $d$ is {\bf monotone} if for all $g\in G$ such that $\|g\|_d< \rho/2$, we have $$\|g^2\|_d= 2 \|g\|_d.$$ To investigate the property of this notion further, we need the following fact about the adjoint representations of Lie groups~\cite[Proposition 9.2.21]{hilgert}.
\begin{fact}\label{fact: adjoint}
Let $\mathfrak{g}$ be the Lie algebra of $G$, and let $\mathrm{Ad}: G\to \mathrm{Aut}(\mathfrak{g})$ be the adjoint representation. Then $\ker(\mathrm{Ad})$ is the center of $G$. 
\end{fact}

\begin{proposition}\label{prop: automaticmonotone}
If $d$ is a locally linear pseudometric on $G$,  then $d$ is monotone.
\end{proposition}
\begin{proof}
We first prove an auxiliary statement.
\begin{claim}
Suppose $s: G \to G, g \mapsto g^2$ is the squaring map. Then there is no open $U \subseteq G$ and proper closed subgroup $H$ of $G$ such that $s(U) \subseteq H$.\medskip

\noindent\emph{Proof of Claim.}
Consider the case where $G$ is a connected component of a linear algebraic subgroup of $\mathrm{GL}_n(\RR)$. Let $J_s$ be the Jacobian of the function $s$. Then the set $$\{ g\in G : \det J_s(g)=0\}$$
has the form $G \cap Z$ where $Z$ is a solution set of a system of polynomial equations. It is not possible to have $G\cap Z = G$, as $s$ is a local diffeomorphism at $\id$. Hence, $G\cap Z$ must be of strictly lower dimension than $G$. By the inverse function theorem, $s|_{G \setminus Z}$ is open. Hence $s(U)$ is not contained in a subgroup of $G$ with smaller dimension.

We also note a stronger conclusion for abelian Lie group: If $V$ is an open subset of a not necessarily connected abelian Lie group $A$, then the image of $A$ under $a \mapsto a^2$ is not contained in a closed subset of $A$ with smaller dimension. Indeed, $A$ is isomorphic as a topological group to $D \times \TT^m \times \RR^m$, with $D$ a discrete group. If $$U \subseteq D \times \TT^m \times \RR^m, $$ then it is easy to see that $\{ a^2 : a\in V\}$ contains a subset of $D \times \TT^m \times \RR^m$, and is therefore not a subset of a closed subset of $A$ with smaller dimension.

Finally, we consider the general case. Suppose to the contrary that $s(U) \subseteq H$ with $H$ a proper closed subgroup of $G$. Let $Z(G)$ be the center of $G$, $G'= G/Z(G)$, $\pi: G \to G'$ be the quotient map, $U'= \pi(U)$, and
$$s': G' \to G', g' \mapsto (g')^2.$$ Then $U'$ is an open subset of $G'$, which is isomorphic as a topological group to a connected component of an algebraic group by Fact~\ref{fact: adjoint}. By the earlier case, $s'(U')$ is not contained in any proper closed subgroup of $G'$, so we must have $\pi(H)=G'$. In particular, this implies $\dim(H \cap Z(G))< \dim Z(G)$, and $HZ(G)=G$. Choose $h \in  H$ such that $hZ(G) \cap U$ is nonempty. Then 
$$ s(hZ(G) \cap U) = \{  h^2  a^2 : a \in Z(G) \cap h^{-1} U   \}. $$
As $s(hZ(G) \cap U)  \subseteq H$, we must have $\{ a^2 : a \in Z(G) \cap h^{-1} U \}$ is a subset of $H\cap Z(G)$. Using the case for abelian Lie groups, this is a contradiction, because  $H\cap Z(G)$ is a closed subset of $Z(G)$ with smaller dimension.
\end{claim}

We now get back to the problem of showing that $d$ is monotone. As $d$ is invariant, $d(\id, g)= d(g, g^2)$ for all $g\in G$.
From local linearity of $d$, for all $g\in G$ with $\|g\|_d< \rho/2$, we either have $$\|g^2\|_d= 2 \|g\|_d \quad\text{or}\quad \|g^2\|_d=0.$$ It suffices to rule out the possibility that $0<\|g\|_d<\rho/4$, and $\|g^2\|_d= 0$. 

As $d$ is continuous, there is an open neighborhood $W$ of $g$ such that for all $g' \in W$, we have $\|g'\|_d>0$ and $\|(g')^2\|_d=0$. From Lemma~\ref{lem: kerd}, the set $\{g \in G :  \|g\|_d =0\}$ is a closed subgroup of $G$. As $d$ is nontrivial and $G$ is a connected Lie group,  $\{g \in G :  \|g\|_d =0\}$ must be a Lie group with smaller dimension. Therefore, we only need to show that if $W$ is an open subset of $G$, then $s(W)$ is not contained in a closed subgroup of $G$ with smaller dimension, where $s:G\to G$ is the squaring map, and this is guaranteed by the earlier claim.
\end{proof}

 The next result confirms our earlier intuition: locally linear pseudometric in $G$ will induce a homomorphism mapping to either $\TT$ or $\RR$.

\begin{proposition}\label{prop: strong linear}
Suppose $d$ is a locally linear pseudometric with radius $\rho>0$. Then $\ker d$ is a normal subgroup of $G$, $G/\ker d$ is isomorphic to $\TT$ if $G$ is compact, and $G/\ker d$ is isomorphic to $\RR$ if $G$ is noncompact.
\end{proposition}
\begin{proof}
We first prove that $\ker d$ is a normal subgroup of $G$.
Suppose $\|g\|_d =0$ and $h\in G$ satisfies $\|h\|_d< \rho/4$. We have 
\begin{align*}
d(h, hgh^{-1}) &= d(\id, gh^{-1}) \\
&=|d(\id,g)\pm d(g, gh^{-1})| = d(\id, h^{-1})= d(\id, h).    
\end{align*}
Hence, $d(\id, hgh^{-1})=|d(\id,h)\pm d(h,hgh^{-1})|$ is either $0$ or $2d(\id, h)$. Assume first that $\|hgh^{-1}\|_d=0$ for every such $h$ when $\|g\|_d=0$. Let 
$$U:=\{h: \|h\|_d< \rho/4\}.$$ By the continuity of $d$, $U$ is open. Hence for every $h$ in $G$, $h$ can be written as a finite product of elements in $U$. By induction, we conclude that for every $h\in G$, $\|hgh^{-1}\|_d=0$ given $\|g\|_d=0$, and this implies that $\ker d$ is normal in $G$.

Suppose $\|hgh^{-1}\|_d = 2 \|h\|_d$. By Proposition~\ref{prop: automaticmonotone}, $d$ is monotone. Hence, we have   $$\| hg^2h^{-1}\|_d = 4 \| h\|_d.$$ On the other hand, as $\|g\|_d=0$, repeating the argument above, we get $\| hg^2h^{-1}\|_d$ is either $0$ or $2\| h\|_d$. Hence, $\| h\|_d=0$, and so $\| hgh^{-1}\|_d=0$.

We now show that $G'= G/\ker d$ has dimension $1$. Let $d'$ be the pseudometric on $G'$ induced by $d$.  Choose $g\in G'$ in the neighborhood of $\mathrm{id}_{G'}$ such that $g$ is in the image of the exponential map and $\|g\|_{d'}< \rho/4$. If $g'$ is another element in the neighborhood of $\mathrm{id}_{G'}$ which is in the image of the exponential map and $\| g'\|_{d'}< \rho/4$. Without loss of generality, we may assume $\|g'\|_{d'}\leq\| g\|_{d'}$. Suppose $g'=\exp(X)$. Then, by monotonicity, there is $k\geq1$ such that $\|(g')^k\|_{d'}\geq\| g\|_{d'}$. By the continuity of the exponential map, there is $t\in (0,1]$ such that $$\|g\|_{d'}=\|\exp(tkX)\|_{d'}.$$ This implies that $g$ and $g'$ are on the same one parameter subgroup, which is the desired conclusion.
\end{proof}

\section*{Acknowledgements}
  The authors would like to thank Lou van den Dries, Arturo Rodriguez Fanlo, Kyle Gannon, Ben Green, John Griesmer, Daniel Hoffmann, Ehud Hrushovski, Simon Machado, Pierre Perruchaud, Daniel Studenmund, Jun Su, Jinhe Ye, and Ruixiang Zhang for valuable discussions.  
  Part of the work was carried out while the first author was visiting the second author at the Department of Mathematics of University of Notre Dame, and the Department of Mathematics of National University of Singapore. He would like to thank both departments for the hospitality he received.

\bibliographystyle{amsalpha}
\bibliography{ref}

\providecommand{\bysame}{\leavevmode\hbox to3em{\hrulefill}\thinspace}
\providecommand{\MR}{\relax\ifhmode\unskip\space\fi MR }
\providecommand{\MRhref}[2]{%
  \href{http://www.ams.org/mathscinet-getitem?mr=#1}{#2}
}
\providecommand{\href}[2]{#2}
\begin{thebibliography}{GKR74}

\bibitem[BG08]{BG08}
Jean Bourgain and Alex Gamburd, \emph{On the spectral gap for
  finitely-generated subgroups of {$\rm SU(2)$}}, Invent. Math. \textbf{171}
  (2008), no.~1, 83--121. \MR{2358056}

\bibitem[BG12]{BG12}
\bysame, \emph{A spectral gap theorem in {${\rm SU}(d)$}}, J. Eur. Math. Soc.
  (JEMS) \textbf{14} (2012), no.~5, 1455--1511. \MR{2966656}

\bibitem[BGT11]{BGT11}
Emmanuel Breuillard, Ben Green, and Terence Tao, \emph{Approximate subgroups of
  linear groups}, Geom. Funct. Anal. \textbf{21} (2011), no.~4, 774--819.
  \MR{2827010}

\bibitem[Bil98]{Bilu}
Yuri Bilu, \emph{The {$(\alpha+2\beta)$}-inequality on a torus}, J. London
  Math. Soc. (2) \textbf{57} (1998), no.~3, 513--528. \MR{1659821}

\bibitem[BL18]{Breuillard18}
Emmanuel Breuillard and Alexander Lubotzky, \emph{Expansion in simple groups},
  arXiv:1807.03879 (2018).

\bibitem[Bog07]{Measuretheory}
Vladimir~I. Bogachev, \emph{Measure theory. {V}ol. {II}}, Springer-Verlag,
  Berlin, 2007. \MR{2267655}

\bibitem[Car15]{thesis}
Pietro~Kreitlon Carolino, \emph{The {S}tructure of {L}ocally {C}ompact
  {A}pproximate {G}roups}, ProQuest LLC, Ann Arbor, MI, 2015, Thesis
  (Ph.D.)--University of California, Los Angeles. \MR{3438951}

\bibitem[CDR19]{circle19}
Pablo Candela and Anne De~Roton, \emph{On sets with small sumset in the
  circle}, Q. J. Math. \textbf{70} (2019), no.~1, 49--69. \MR{3927843}

\bibitem[CI21]{ChristIliopoulou}
Michael Christ and Marina Iliopoulou, \emph{Inequalities of {R}iesz-{S}obolev
  type for compact connected abelian groups}, to appear at Amer. J. Math.
  (2021).

\bibitem[DE09]{Harmonicanalysis}
Anton Deitmar and Siegfried Echterhoff, \emph{Principles of harmonic analysis},
  Universitext, Springer, New York, 2009. \MR{2457798}

\bibitem[GKR74]{almosthomo}
Karsten Grove, Hermann Karcher, and Ernst~A. Ruh, \emph{Group actions and
  curvature}, Invent. Math. \textbf{23} (1974), 31--48. \MR{385750}

\bibitem[Gle52]{Gleason}
Andrew~M. Gleason, \emph{Groups without small subgroups}, Ann. of Math. (2)
  \textbf{56} (1952), 193--212. \MR{49203}

\bibitem[Gre]{BenBook}
Ben Green, \emph{100 open problems}, manuscript.

\bibitem[Gri19]{G19}
John~T. Griesmer, \emph{Semicontinuity of structure for small sumsets in
  compact abelian groups}, Discrete Anal. (2019), Paper No. 18, 46.
  \MR{4042161}

\bibitem[Gry13]{Grynkiewicz}
David~J. Grynkiewicz, \emph{Structural additive theory}, Developments in
  Mathematics, vol.~30, Springer, Cham, 2013. \MR{3097619}

\bibitem[Hel08]{Helfgott08}
Harald Helfgott, \emph{Growth and generation in {${\rm SL}_2(\mathbb Z/p\mathbb
  Z)$}}, Ann. of Math. (2) \textbf{167} (2008), no.~2, 601--623. \MR{2415382}

\bibitem[HN12]{hilgert}
Joachim Hilgert and Karl-Hermann Neeb, \emph{Structure and geometry of {L}ie
  groups}, Springer Monographs in Mathematics, Springer, New York, 2012.
  \MR{3025417}

\bibitem[JTZ21]{JTZ}
Yifan Jing, Chieu-Minh Tran, and Ruixiang Zhang, \emph{A nonabelian
  {B}runn-{M}inkowski inequality}, arXiv:2101.07782 (2021).

\bibitem[Kaz82]{almosthomo2}
David Kazhdan, \emph{On {$\varepsilon $}-representations}, Israel J. Math.
  \textbf{43} (1982), no.~4, 315--323. \MR{693352}

\bibitem[Kec95]{Kechris}
Alexander~S. Kechris, \emph{Classical descriptive set theory}, Graduate Texts
  in Mathematics, vol. 156, Springer-Verlag, New York, 1995. \MR{1321597}

\bibitem[Kem64]{Kemperman}
Johannes Kemperman, \emph{On products of sets in a locally compact group},
  Fund. Math. \textbf{56} (1964), 51--68. \MR{202913}

\bibitem[Kne56]{Kneser}
Martin Kneser, \emph{Summenmengen in lokalkompakten abelschen {G}ruppen}, Math.
  Z. \textbf{66} (1956), 88--110. \MR{81438}

\bibitem[PS16]{PS16}
L\'{a}szl\'{o} Pyber and Endre Szab\'{o}, \emph{Growth in finite simple groups
  of {L}ie type}, J. Amer. Math. Soc. \textbf{29} (2016), no.~1, 95--146.
  \MR{3402696}

\bibitem[Ros19]{Rosendal}
Christian Rosendal, \emph{Continuity of universally measurable homomorphisms},
  Forum Math. Pi \textbf{7} (2019), e5, 20. \MR{3996719}

\bibitem[Tao08]{T08}
Terence Tao, \emph{Product set estimates for non-commutative groups},
  Combinatorica \textbf{28} (2008), no.~5, 547--594. \MR{2501249}

\bibitem[Tao15]{Tao-expansion}
\bysame, \emph{Expansion in finite simple groups of {L}ie type}, Graduate
  Studies in Mathematics, vol. 164, American Mathematical Society, Providence,
  RI, 2015. \MR{3309986}

\bibitem[Tao18]{T18}
\bysame, \emph{An inverse theorem for an inequality of {K}neser}, Proc. Steklov
  Inst. Math. \textbf{303} (2018), no.~1, 193--219, Published in Russian in Tr.
  Mat. Inst. Steklova {{\bf{3}}03} (2018), 209--238. \MR{3920221}

\bibitem[TT19]{TT18}
Terence Tao and Joni Ter\"{a}v\"{a}inen, \emph{The structure of logarithmically
  averaged correlations of multiplicative functions, with applications to the
  {C}howla and {E}lliott conjectures}, Duke Math. J. \textbf{168} (2019),
  no.~11, 1977--2027. \MR{3992031}

\bibitem[vZ04]{almosthomonottrueingeneral}
J\'{a}n \v{S}pakula and Pavol Zlato\v{s}, \emph{Almost homomorphisms of compact
  groups}, Illinois J. Math. \textbf{48} (2004), no.~4, 1183--1189.
  \MR{2113671}

\bibitem[Wei40]{Weils}
Andr\'{e} Weil, \emph{L'int\'{e}gration dans les groupes topologiques et ses
  applications}, Actual. Sci. Ind., no. 869, Hermann et Cie., Paris, 1940,
  [This book has been republished by the author at Princeton, N. J., 1941.].
  \MR{0005741}

\bibitem[Yam53]{Yamabe}
Hidehiko Yamabe, \emph{A generalization of a theorem of {G}leason}, Ann. of
  Math. (2) \textbf{58} (1953), 351--365. \MR{58607}

\end{thebibliography}

\end{document}